\documentclass[11pt]{article}
\usepackage[margin=3cm]{geometry}
\usepackage{amsmath,amsthm,amsfonts,amssymb,amscd}
\usepackage{tikz}

\usepackage{dsfont}
\usepackage{framed}
\usepackage[utf8]{inputenc}
\usepackage{mathtools}
\usepackage{multicol}
\usepackage{multirow}
\usepackage{enumerate}
\usepackage{bbm}

 \usepackage[colorlinks = true]{hyperref}
 \hypersetup{
 	linkcolor = blue,
 	urlcolor = blue,
 	citecolor = blue
 }

\newcommand{\reff}[1]{(\ref{#1})}

\newcommand{\IH}{\mathbb{H}}
\newcommand{\ID}{\mathbb{D}}
\newcommand{\IW}{\mathbb{W}}
\newcommand{\IE}{\mathbb{E}}
\newcommand{\IK}{\mathbb{K}}

\newcommand{\sG}{\mathcal{G}}

\newcommand{\pto}{\overset{p}{\to}}
\newcommand{\dto}{\overset{d}{\to}}
\newcommand{\R}{\mathbb{R}}

\newcommand{\Z}{\mathbb{Z}}
\newcommand{\N}{\mathbb{N}}
\newcommand{\E}{\mathbb{E}}

\newcommand{\IP}{\mathbb{P}}

\newcommand{\sF}{\mathcal{F}}

\newcommand{\sH}{\mathcal{H}}
\newcommand{\G}{\mathbb{G}}
\newcommand{\Ii}{\mathbbm{1}}
\newcommand{\Cov}{\mathrm{Cov}}
\newcommand{\Var}{\mathrm{Var}}

\newcommand{\eps}{\varepsilon}

\newcommand{\norm}[1]{\left\lVert#1\right\rVert}

\makeatletter
\newcommand{\undersim}[1]{\mathrel{\mathpalette\@undersim{#1}}}
\newcommand{\@undersim}[2]{%
  \vcenter{%
    \ialign{%
      ##\cr
      $\m@th#1#2$\cr
      \noalign{\nointerlineskip\kern.2ex}
      $\m@th#1\sim$\cr
      \noalign{\kern-.4ex}
    }%
  }%
}

\makeatother

\newtheorem{thm}{Theorem}[section]
\newtheorem{lem}[thm]{Lemma}
\newtheorem{cor}[thm]{Corollary}

\newtheorem{assum}[thm]{Assumption}

\theoremstyle{definition}
\newtheorem{defi}[thm]{Definition}
\newtheorem{rem}[thm]{Remark}
\newtheorem{ex}[thm]{Example}

\numberwithin{equation}{section}

\RequirePackage[numbers]{natbib}

\begin{document}

\title{Empirical process theory for locally stationary processes}

\thispagestyle{empty}

\begin{center}
{\LARGE \bf Empirical process theory for locally stationary processes}\\
{\large Nathawut Phandoidaen, Stefan Richter}\\

{phandoidaen@math.uni-heidelberg.de, stefan.richter@iwr.uni-heidelberg.de}\\

{\small Institut für angewandte Mathematik, Im Neuenheimer Feld 205, Universität Heidelberg}\\

{\small To appear in Bernoulli}\\

\today
\end{center}

\begin{abstract}
We provide a framework for empirical process theory of locally stationary processes using the functional dependence measure. Our results extend known results for stationary Markov chains and mixing sequences by another common possibility to measure dependence and allow for additional time dependence. Our main result is a functional central limit theorem for locally stationary processes. Moreover, maximal inequalities for expectations of sums are developed. We show the applicability of our theory in some examples, for instance we provide uniform convergence rates for nonparametric regression with locally stationary noise.
\end{abstract}

\section{Introduction}
\label{sec_intro}

Empirical process theory is a powerful tool to prove uniform convergence rates and weak convergence of composite functionals. The theory for independent variables is well-studied (cf. \cite{dudley}, \cite{nickl}, \cite{wellner} or \cite{Vaart98} for an overview) based on the original ideas of \cite{donsker1952}, \cite{dudley1966}, \cite{dudley1978},  \cite{pollard1982} and \cite{ossiander}  among others. For random variables with dependence structure, various approaches have been discussed. There exists a well-developed empirical process theory and large deviation results for Harris-recurrent Markov chains based on regenerative schemes (cf. \cite{levental}, \cite{samur}, \cite{ellis} and \cite{adamczak} among others) or geometric ergodicity (cf. \cite{kulik}). To quantify the speed of convergence in maximal inequalities, additional assumptions like $\beta$-recurrence (cf. \cite{tjost01}) have to be imposed. The theory covers a rich class of Markov chains, but for instance, can not discuss linear processes. 

An empirical process theory for stationary processes under high-level assumptions on the moments of means was derived in \cite{dehling1} and further discussion papers. In the paradigm of weak dependence (which measures the size of covariances of  Lipschitz functions of the random variables), \cite{doukhan07} derived Bernstein-type inequalities. Focusing on the analysis of the empirical distribution function (EDF), much more techniques were discussed: For instance, \cite[Theorem 4]{durieu14} provide uniform convergence of the EDF by using bounds for covariances of H\"older functions of the random variables. Another abstract concept was introduced by \cite{berkes09} via S-mixing (for stationary mixing), which imposes the existence of $m$-dependent approximations of the original observations. They then derive strong approximations and uniform central limit theorems for the EDF.

A different idea to measure dependence of random variables is given by mixing coefficients. Here, several concepts were introduced, the most common (with increasing strength) being $\alpha$-, $\beta$- and $\phi$-mixing (for an overview about mixing coefficients, cf. \cite{doukhan94}). Large deviation results and uniform central limit theorems for general classes of functions (not only EDF) were derived by using coupling techniques, cf. \cite{rio1995}, \cite{liebscher1996} for $\alpha$-mixing, \cite{arcones}, \cite{yu_bin}, and successively refined by \cite{doukhan1995}, \cite{rio98}, \cite{dedecker07}, \cite{dedecker10} (the last two developed for EDFs only) and  \cite{rio17} for $\beta$-mixing, and \cite{Dedeck02}, \cite{dehling2001} for $\beta$- and $\phi$-mixing. See also \cite{pollard_overview}, \cite{Dedeck02} and \cite{rio17} for comprehensive overviews.

In \cite{Dedeck02} it is argued that  $\beta$-mixing is the weakest mixing assumption that allows for a ``complete'' empirical process theory which incorporates maximal inequalities and uniform central limit theorems. There exist explicit upper bounds for $\beta$-mixing coefficients for Markov chains (cf. \cite{heinrich}) and for so-called V-geometric mixing coefficients (cf. \cite{mt09}).  For several stationary time series models like linear processes (cf. \cite{pham_tran} for $\alpha$-mixing), ARMA (cf. \cite{mixing_arma}), nonlinear AR (cf. \cite{tjost01}) and GARCH processes (cf. \cite{mixing_garch}) there also exist upper bounds on mixing coefficients. A common assumption in these results is that the observed process or, more often, the innovations of the corresponding process, have a continuous distribution. This is a crucial assumption to handle the relatively complicated mixing coefficients defined over a supremum over two different sigma-algebras. A relaxation of $\beta$-mixing coefficients was investigated by \cite[Theorem 1]{dedecker07} and is specifically designed for the analysis of the EDF. In opposite to sigma-algebras, these smaller coefficients are defined with conditional expectations of certain classes of functions and are easier to upper bound for a wide range of time series models.

During the last years, another measure for dependence, the so-called functional dependence measure, became popular (cf. \cite{wu2005anotherlook}), which uses a Bernoulli shift representation (see \reff{representation_x} below) and decomposition into martingales and $m$-dependent sequences. It has been shown in various applications that the functional dependence measure allows, when combined with the rich theory of martingales, for sharp large deviation inequalities (cf. \cite{Wu13} or \cite{wuzhang2017}). In  \cite{wu2008} and \cite{mayer2019}, uniform central limit theorems for the EDF were derived for stationary and piecewise locally stationary processes.

Up to now, no \emph{general} empirical process theory (allowing for general classes of functions) using the functional dependence measure is available. 
In this paper we will fill this gap and prove maximal inequalities and functional central limit theorems under functional dependence. Furthermore, we will draw connections and compare our results to already existing empirical process concepts for dependent data we mentioned above. While the empirical process theory for Markov chains and mixing cited above was developed for stationary processes, we will work in the framework of locally stationary processes and therefore automatically provide the first general empirical process theory in this setting (\cite{dahlhauspolonik} investigated spectral empirical processes for linear processes, \cite{mayer2019} proved a functional central limit theorem for a localized empirical distribution function). Locally stationary processes allow for a smooth change of the distribution over time but can locally approximated by stationary processes. Therefore, they provide more flexible time series models (cf. \cite{dahlhaus2019} for an introduction).

The functional dependence measure uses a representation of the given process as a Bernoulli shift process and quantifies dependence with a $L^{\nu}$-norm. More precisely, we assume that $X_i =  (X_{ij})_{j=1,...,d}$, $i = 1,...,n$, is a $d$-dimensional process of the form
\begin{equation}
    X_i = J_{i,n}(\sG_i),\label{representation_x}
\end{equation}
where $\sG_i = \sigma(\varepsilon_i,\varepsilon_{i-1},...)$ is the sigma-algebra generated by $\varepsilon_i$, $i \in\Z$, a sequence of i.i.d. random variables in $\R^{\tilde d}$ ($d,\tilde d \in \N$), and some measurable function $J_{i,n}:(\R^{\tilde d})^{\N_0}\to \R$, $i=1,...,n$, $n\in\N$. For a real-valued random variable $W$ and some $\nu > 0$, we define $\|W\|_{\nu} := \IE[|W|^{\nu}]^{1/\nu}$. 
If $\varepsilon_k^{*}$ is an independent copy of $\varepsilon_k$, independent of $\varepsilon_i, i\in\Z$, we define $\sG_i^{*(i-k)} := (\varepsilon_i,...,\varepsilon_{i-k+1},\varepsilon_{i-k}^{*},\varepsilon_{i-k-1},...)$ and $X_i^{*(i-k)} := J_{i,n}(\sG_{i}^{*(i-k)})$. The uniform functional dependence measure is given by
\begin{equation}
    \delta_{\nu}^{X}(k) = \sup_{i=1,...,n}\sup_{j=1,...,d}\big\|X_{ij} - X_{ij}^{*(i-k)}\big\|_\nu.\label{definition_uniform_functional_dependence_measure}
\end{equation}
Graphically, $\delta_{\nu}^{X}$ measures the impact of $\varepsilon_0$ in $X_k$. Although representation \reff{representation_x} appears to be rather restrictive, it does cover a large variety of  processes. In \cite{borkar1993} it was motivated that the set of all processes of the form $X_i = J(\varepsilon_i,\varepsilon_{i-1},...)$ should be equal to the set of all stationary and ergodic processes. We additionally allow $J$ to vary with $i$ and $n$ to cover processes which change their stochastic behavior over time. This is exactly the form of the so-called locally stationary processes discussed in \cite{dahlhaus2019}.

Since we are working in the time series context, many applications ask for functions $f$ that not only depend on the actual observation of the process but on the whole (infinite) past $Z_i := (X_i,X_{i-1},X_{i-2},...)$. In the course of this paper, we aim to derive asymptotic properties of the empirical process
\begin{equation}
    \G_n(f) := \frac{1}{\sqrt{n}}\sum_{i=1}^{n}\big\{ f(Z_i,\frac{i}{n}) - \E f(Z_i,\frac{i}{n})\big\}, \quad f \in \sF,\label{definition_empirical_process}
\end{equation}
where
\[
    \sF \subset \{f:(\R^d)^{\N_0}\times [0,1] \to \R \text{ measurable}\}.
\]
Let $\IH(\varepsilon,\sF,\|\cdot\|)$ denote the bracketing entropy, that is, the logarithm of the number of $\varepsilon$-brackets with respect to some distance $\|\cdot\|$ that is necessary to cover $\sF$ (this is made precise at the end of this section). We will define a distance $V_n$ which guarantees weak convergence of \reff{definition_empirical_process} if the corresponding bracketing entropy integral $\int_0^{1}\sqrt{\IH(\varepsilon,\sF,V_n)}d\varepsilon$ is finite.

The definition of the functional dependence measure for locally stationary processes is similar to its stationary version and is easy to calculate for many time series models. It does not rely on the stationarity assumption but on the representation of the process as a Bernoulli shift. Therefore, many well-known upper bounds for stationary time series given in  \cite{wu2011}, including recursively defined models and linear models, directly carry over to the locally stationary case \reff{definition_uniform_functional_dependence_measure}. It seems reasonable to use it as a starting point to generalize empirical process theory for stationary processes to the more general setting of local stationarity. While the other two paradigms mentioned above should also allow such a generalization in principle, there are open questions:
\begin{itemize}
    \item The theory for Harris-recurrent Markov chain relies on stationarity and intrinsically needs some knowledge about the whole time series due to the assumption of null-recurrence. There exist generalizations of local stationary Markov chains (cf. for instance  \cite{truquet}), but the corresponding recurrence properties and examples of locally stationary time series are not worked out yet. Furthermore, it is not directly clear how to deal with processes $Z_i$ which incorporate the infinite past of $X_i$, and linear processes can not be discussed easily.
    \item Absolutely regular $\beta$-mixing is shown in most examples by assuming some continuity in the distribution or in the corresponding innovations. Especially for linear processes, the bounds are quite hard to obtain and seem not to be optimal.  Moreover, there exist no ``invariance rules'' which would directly allow to transfer the mixing properties of $X_i$ to $f(Z_i,\frac{i}{n})$ which incorporates \emph{infinitely} many lags of $X_i$.
    \item Many of the more elaborated dependence concepts developed in e.g. \cite{dedecker07}, \cite{dehling2001}, \cite{durieu14}, \cite{berkes09} are restricted (at least in their original formulation) to the discussion of the EDF or connected one-dimensional indexed function classes.
\end{itemize}
Contrary to $\beta$-mixing, the functional dependence measure can easily deal with $f(Z_i,\frac{i}{n})$ by using H\"older-type assumptions on $f$. Furthermore, it easily can be calculated in many situations and is not restricted to noncontinuous distributions of $X_i$. Also, linear processes $X_i$ are covered.

However, there are also some peculiarities using the functional dependence measure  \reff{definition_uniform_functional_dependence_measure}. While for Harris-recurrent Markov chains and $\beta$-mixing, the empirical process theory is independent of the function class considered, the situation for the functional dependence measure is more complicated. In order to quantify the dependence of $f(Z_i,\frac{i}{n})$ by $\delta_{\nu}^{X}$, we have to impose smoothness conditions on $f$ in direction of its first argument. The distance $V_n$ therefore will not only change with the dependence structure of $X$, but also has to be ``compatible'' with the function class $\sF$. The smoothness condition on $f$ also poses a challenging issue when considering chaining procedures where rare events are excluded by (non-smooth) indicator functions.

Our main contributions in this paper are the following:
\begin{itemize}
    \item We derive maximal inequalities for $\G_n(f)$ for classes of functions $\sF$, 
    \item a chaining device which preserves smoothness during the chaining procedure and
    \item conditions to ensure asymptotic tightness and functional convergence of $\G_n(f)$, $f\in\sF$.
\end{itemize}



The paper is organized as follows. In Section \ref{sec_model}, we present our main result Theorem \ref{cor_functional_central_limit_theorem}, the functional central limit theorem under minimal moment conditions. As a special case, we derive a version for stationary processes. We give a discussion on the distance $V_n$ and compare our result with the empirical process theory for $\beta$-mixing. Some Assumptions are postponed to Section \ref{sec_clt}, where a new multivariate central limit theorem for locally stationary processes is presented. In Section \ref{sec_max_asymptotic_tightness}, we provide new maximal inequalities for $\G_n(f)$ for both finite and infinite $\sF$. In Section \ref{sec_examples}, we apply our theory to prove uniform convergence rates and weak convergences of several estimators. The aim of the last section is to highlight the wide range of applicability of our theory and to provide the typical conditions which have to be imposed as well as some discussion. In Section \ref{sec_conclusion}, a conclusion is drawn. We illustrate the main steps of the proofs in the Appendix of the article but postpone all detailed proofs to the  Supplementary Material.

We now introduce some basic notation. For $a,b\in\R$, let $a\wedge b := \min\{a,b\}$, $a\vee b := \max\{a,b\}$. For $k\in\N$,
\begin{equation}
    H(k) := 1 \vee \log(k) \label{H_var}
\end{equation}
which naturally appears in large deviation inequalities. For a given finite class $\sF$, let $|\sF|$ denote its cardinality. We use the abbreviation
\begin{equation}
    H = H(|\sF|) = 1 \vee \log |\sF| \label{H_fix}
\end{equation}
if no confusion arises. For some distance $\|\cdot\|$, let $\N(\varepsilon,\sF,\|\cdot\|)$ denote the bracketing numbers, that is, the smallest number of $\varepsilon$-brackets $[l_j,u_j] := \{f\in \sF: l_j \le f \le u_j\}$ (i.e. measurable functions $l_j,u_j:(\R^{d})^{\N_0}\times[0,1]\to\R$ with $\|u_j - l_j\| \le \varepsilon$ for all $j$) to cover $\sF$. Let $\IH(\varepsilon,\sF,\|\cdot\|) := \log \N(\varepsilon,\sF,\|\cdot\|)$ denote the bracketing entropy. For $\nu \ge 1$, let
\[
    \|f\|_{\nu,n} := \Big(\frac{1}{n}\sum_{i=1}^{n}\big\|f\big(Z_{i},\frac{i}{n}\big)\big\|_\nu^\nu\Big)^{1/\nu}.
\]

\section{A new functional central limit theorem}
\label{sec_model}

Roughly speaking, a process $X_i$, $i=1,...,n$ is called locally stationary if for each $u\in[0,1]$, there exists a stationary process $\tilde X_i(u)$, $i=1,...,n$ such that $X_i \approx \tilde X_i(u)$ if $|u - \frac{i}{n}|$ is small (cf. \cite{dahlhaus2019}). Typical estimators are of the form
\[
    \frac{1}{nh}\sum_{i=1}^{n}K\big(\frac{i/n-u}{h}\big) \bar f(Z_i,\frac{i}{n})
\]
where $K$ is a kernel function and $h = h_n > 0$ is a bandwidth. Clearly, such a localization changes the convergence rate. To cover these cases, we suppose that any $f\in \sF$ has a representation
\begin{equation}
    f(z,u) = D_{f,n}(u)\cdot \bar f(z,u), \quad\quad z\in(\R^d)^{\N_0}, u\in[0,1],\label{f_form_decomposition}
\end{equation}
where $\bar f$ is independent of $n$ and $D_{f,n}(u)$ is independent of $z$. We put
\begin{equation}
    \bar\sF := \{\bar f: f\in \sF\}.\label{definition_fbar}
\end{equation}

The function class $\bar\sF$ is considered to consist of H\"older-continuous functions in direction of $z$. For $s \in (0,1]$, a sequence $z = (z_i)_{i\in\N_0}$ of elements of $\R^d$ (equipped with the maximum norm $|\cdot|_{\infty}$) and an absolutely summable sequence $\chi = (\chi_i)_{i\in\N_0}$ of nonnegative real numbers, we set
\[
    |z|_{\chi,s} := \Big(\sum_{i=0}^{\infty}\chi_i |z_i|_{\infty}^s\Big)^{1/s}
\]
and $|z|_{\chi} := |z|_{\chi,1}$.

\begin{defi} $\bar\sF$ is called a $(L_{\sF},s,R,C)$-class if $L_{\sF} = (L_{\sF,i})_{i\in\N_0}$ is a sequence of nonnegative real numbers, $s \in (0,1]$ and $R:(\R^d)^{\N_0} \times [0,1] \to [0,\infty)$ satisfies for all $u\in[0,1]$, $z,z' \in (\R^d)^{\N_0}$, $\bar f \in \bar \sF$,
    \[
         |\bar f(z,u) - \bar f(z',u)| \le |z-z'|_{L_{\sF},s}^s\cdot \big[R(z,u) + R(z',u)\big].
    \]
    Furthermore, $C = (C_R,C_{\bar f}) \in (0,\infty)^2$ satisfies $\sup_u|\bar f(0,u)| \le C_{\bar f}$, $\sup_u|R(0,u)| \le C_R$.
\end{defi}

The basic assumption for our main result is the following compatibility condition on $\sF$.

\begin{assum}\label{ass1_vorher}
    $\bar\sF$ is a $(L_{\sF},s,R,C)$-class. There exists $p\in (1,\infty]$, $C_X > 0$ such that
    \begin{equation}
        \sup_{i,u}\| R(Z_{i},u)\|_{2p} \le C_R, \quad\quad \sup_{i,j}\|X_{ij}\|_{\frac{2sp}{p-1}} \le C_X.\label{ass1_eq1_vorher}
    \end{equation}
    Let $\ID_n \ge 0$, $\Delta(k) \ge 0$ be such that for all $k\in\N_0$,
    \[
        2d C_R \cdot \sum_{j = 0}^kL_{\sF,j} (\delta_{\frac{2sp}{p-1}}^X(k-j))^s \le \Delta(k), \quad\quad \sup_{f\in \sF} \Big(\frac{1}{n}\sum_{i=1}^{n}\big|D_{f,n}(\frac{i}{n})\big|^{2}\Big)^{1/2} \le \ID_{n}.
    \]
\end{assum}

While \reff{ass1_eq1_vorher} summarizes moment assumptions on $X_{ij}$ which are balanced by $p$, the sequence $\Delta(k)$ reflects the intrinsic dependence of $f(Z_i,\frac{i}{n})$ and $\ID_n$ measures the influence of the factor $D_{f,n}(u)$ to the convergence rate of $\G_n(f)$.

Based on Assumption \ref{ass1_vorher}, we define for $f\in\sF$
\begin{equation}
    V_n(f) := \|f\|_{2,n} + \sum_{k=1}^{\infty}\min\{\|f\|_{2,n},\ID_n \Delta(k)\}.\label{definition_v}
\end{equation}
Clearly, $V_n$ satisfies a triangle inequality. Therefore, $V_n(f-g)$ is a distance between $f,g\in \sF$. We are now able to state our main result. The weak convergence takes place in the normed space
\begin{equation}
    \ell^{\infty}(\sF) = \{\G:\sF \to \R \,|\, \|\G\|_{\infty} := \sup_{f\in\sF}|\G(f)| < \infty\},\label{definition_linf}
\end{equation}
cf. \cite{Vaart98} for a detailed discussion of this space.

\begin{thm}\label{cor_functional_central_limit_theorem}
    Let $\sF$ satisfy Assumptions \ref{ass1_vorher}, \ref{ass_clt_process}, \ref{ass_clt_fcont} and \ref{ass_clt_d}. Assume that
    \[
        \sup_{n\in\N}\int_0^{1}\sqrt{\IH(\varepsilon,\sF,V_n)} d \varepsilon < \infty.
    \]
    Then it holds in $\ell^{\infty}(\sF)$ that
    \[
        \big[\G_n(f)\big]_{f\in\sF} \dto \big[\G(f)\big]_{f\in\sF},
    \]
    where $(\G(f))_{f\in\sF}$ is a centered Gaussian process with covariances
    \[
        \Cov(\G(f),\G(g)) = \lim_{n\to\infty}\Cov(\G_n(f), \G_n(g)).
    \]
\end{thm}

The proof of Theorem \ref{cor_functional_central_limit_theorem} consists of two ingredients, convergence of the finite-dimensional distributions (cf. Theorem \ref{theorem_clt_mult}) and asymptotic tightness (cf. Corollary \ref{cor_equicont}). The more challenging part is asymptotic tightness; its proof only relies on Assumption \ref{ass1_vorher} and consists of a new maximal inequality presented in Theorem \ref{proposition_rosenthal_bound} which may be of independent interest. To ensure convergence of the finite-dimensional distributions, we have to formalize local stationarity (Assumption \ref{ass_clt_process}) and pose conditions in time direction on $\bar f(z,\cdot)$ (cf. Assumption \ref{ass_clt_fcont}) and $D_{f,n}(\cdot)$ (cf. Assumption \ref{ass_clt_d}) which is done in Section \ref{sec_clt}. In particular, it is needed that $D_{f,n}(u)$ is properly normalized.

Let us note that in the case that $X_i$ is stationary, $\bar f(z,u) = \bar f(z)$ and $D_{f,n}(u) = 1$, Assumptions \ref{ass_clt_process}, \ref{ass_clt_fcont} and \ref{ass_clt_d} are directly fulfilled. That is, in the stationary case, Assumption \ref{ass1_vorher} is sufficient for Theorem \ref{cor_functional_central_limit_theorem}. We formulate this finding as a simple corollary. Let 
\[
    \tilde\G_n(h) := \frac{1}{\sqrt{n}}\sum_{i=1}^{n}\big\{h(X_i) - \IE h(X_i)\big\},
\]
where $X_i = J(\sG_i)$, $i=1,...,n$, is a stationary process and $h$ are functions from
\[
    \sH \subset \{h:\R^d \to \R \text{ measurable}\}
\]
with the property that for all $x,y\in\R^d$, $|h(x) - h(y)| \le L_{\sH}|x-y|_{\infty}^s$.

\begin{cor}\label{cor_functional_central_limit_theorem_cor} Suppose that $\|X_1\|_{2s} < \infty$. Let $\Delta(k) := 2d L_{\sH}\delta_{2s}^{X}(k)^s$ and $\ID_n := 1$. Assume that
\begin{equation}
    \sup_{n\in\N}\int_0^{1}\sqrt{\IH(\varepsilon,\sH,V_n)} d \varepsilon < \infty.\label{depmeas_entropycondition}
\end{equation}
    Then it holds in $\ell^{\infty}(\sH)$ that
    \[
        \big[\tilde\G_n(h)\big]_{h\in\sH} \dto \big[\tilde\G(h)\big]_{h\in\sH}
    \]
    where $(\tilde\G(h))_{h\in\sH}$ is a centered Gaussian process with covariances
    \[
        \Cov(\tilde\G(h_1),\tilde\G(h_2)) = \sum_{k\in\Z}\Cov(h_1(X_0), h_2(X_k)).
    \]
\end{cor}

\subsection{Form of $V_n$ and discussion on $\Delta(k)$}

\subsubsection{Form of $V_n$}
Suppose that $\ID_n \in (0,\infty)$ is independent of $n\in\N$. Based on decay rates of $\Delta(k)$, simpler forms of $V_n$ can be derived and are given in Table \ref{table_v_values}. These results are elementary and are proved in Lemmas \ref{lemma_form_v2}, \ref{cor_integralcondition_explicit} in the Supplementary Material.

\renewcommand{\arraystretch}{2}

\begin{table}[h!]
    \centering
        \begin{tabular}{l|l|l}
         & \multicolumn{2}{c}{$\Delta(j)$ }  \\
         &  $c j^{-\alpha}$, $\alpha > 1, c > 0$         & $c \rho^j$, $\rho \in (0,1)$, $c > 0$    \\
         \hline
         \hline
        $V_n(f)$ & $\|f\|_{2,n}\max\{\|f\|_{2,n}^{-\frac{1}{\alpha}},1\}$ & $\|f\|_{2,n}\max\{\log(\norm{f}_{2,n}^{-1}),1\}$ \\ \hline
        $\int_0^\sigma \sqrt{\IH(\eps,\sF,V_n)}d\eps$ & $\int_0^{\tilde \sigma} \eps^{-\frac{1}{\alpha}}\sqrt{\IH(\eps,\sF, \|\cdot\|_{2,n})}d\eps$ & $\int_0^{\tilde\sigma} \log(\eps^{-1})\sqrt{\IH(\eps,\sF, \|\cdot\|_{2,n})}d\eps$
        \end{tabular}
        
        \caption{Equivalent expressions of $V_n$ and the corresponding entropy integral under the condition that $\ID_n \in (0,\infty)$ is independent of $n$. We omitted the lower and upper bound constants which are only depending on $c, \rho, \alpha$ and $\ID_n$. Furthermore, $\tilde \sigma = \tilde \sigma (\sigma)$ fulfills $\tilde \sigma \to 0$ for $\sigma \to 0$. }
        \label{table_v_values}
    \end{table}
    
If $f(Z_i,\frac{i}{n})$, $i=1,...,n$, are independent, we can choose $\Delta(k)=0$ for $k \ge 1$ and thus $V_n(f)$ is proportional to $\|f\|_{2,n}$. We therefore exactly recover the case of independent variables with our theory.

\subsubsection{Discussion on $\Delta(k)$}
Assumption \ref{ass1_vorher} asks $\Delta(k)$ to upper bound
\[
    \sum_{j=0}^{k}L_{\sF,j}(\delta_{\frac{2sp}{p-1}}^X(k-j))^{s},
\]
which is a convolution of the uniform H\"older constants $L_{\sF,j}$ of $f \in \sF$ and the dependence measure $\delta_{\frac{2sp}{p-1}}^X(k)$ of $X$. Therefore, the specific form of $f\in \sF$ has an impact on the dependence structure which is then introduced in $V_n$. This is contrary to other typical chaining approaches for Harris-recurrent Markov chains or $\beta$-mixing sequences where the dependence structure of $X_i$ simply transfers to functions $f(X_i)$ without further conditions.

Furthermore, contrary to other chaining approaches, we \emph{have to} ask for the existence of moments of $X_i$ in Assumption \ref{ass1_vorher} even though $\G_n(f)$ only involves $f(X_i)$. This is due to the linear nature of the functional dependence measure \reff{definition_uniform_functional_dependence_measure}. If $f$ is Lipschitz continuous with respect to its first argument ($s=1$ in Assumption \ref{ass1_vorher}), we have to impose $\sup_{i,j}\|X_{ij}\|_2 < \infty$. However, these moment assumptions can be relaxed at the cost of larger $\Delta(k)$ as follows. Let us consider the special case that $f(Z_i,\frac{i}{n})$ only depends on $X_i, \frac{i}{n}$, that is, $f(z,u) = f(z_0,u)$. If $f$ is bounded and Lipschitz continuous with respect to its first argument with Lipschitz constant $L$, for any $s\in (0,1]$,
\[
    |f(z_0,u) - f(z_0',u)| \le \min\{2\|f\|_{\infty}, L|z_0 - z_0'|\} \le (2\|f\|_{\infty})^{1-s} L^s |z_0 - z_0'|^s.
\]
Thus, $\Delta(k)$ can be chosen proportional to $\delta_{2s}^X(k)^s$. This means that we can reduce the moment assumption to $\sup_{i,j}\|X_{ij}\|_{2s} < \infty$ at the cost of having a larger norm $V_n$.

\subsection{Comparison to empirical process theory with $\beta$-mixing}
\label{sec_discussion_mixing}

In this section, we compare our functional central limit theorem for stationary processes from Corollary \ref{cor_functional_central_limit_theorem_cor} under functional dependence with similar results obtained under $\beta$-mixing. Unfortunately, we were not able to find a general setting under which the functional dependence measure $\delta_2^{X}$ can be compared with the $\beta$-mixing coefficients $\beta^{X}$ of $X_i$, $i=1,...,n$. However, in some special cases, both quantities can be upper bounded.

\subsubsection{Upper bounds for dependence coefficients of linear processes}
Consider the linear process
\[
    X_i = \sum_{k=0}^{\infty}a_k \varepsilon_{i-k},\quad i = 1,...,n,
\]
with an absolutely summable sequence $a_k$, $k\in\N_0$, and i.i.d. $\varepsilon_k$, $k\in\Z$, with $\IE \varepsilon_1=0$. Then it is immediate that
\[
    \delta_2^{X}(k) \le 2|a_k|\cdot \|\varepsilon_1\|_2.
\]
From \cite{pham_tran} (cf. also \cite{doukhan94}, Section 2.3.1), we have the following result. If for some $\nu \ge 1$,  $\|\varepsilon_1\|_{\nu} < \infty$, $\varepsilon_1$ has a Lipschitz-continuous Lebesgue-density and the process $X_i$ is invertible, then for some constant $\zeta > 0$,
\[
    \beta^{X}(k) \le \zeta \cdot \big(\sum_{m=k}^{\infty}|A_{m,\nu}|^{\frac{1}{1+\nu}}\big) \vee \big(\sum_{m=k}^{\infty}L(A_{m,2})\big),
\]
where $A_{m,s} := \sum_{k=m}^{\infty}|a_k|^{s}$ and $L(u) = \sqrt{u(1 \vee |\log(u)|)}$. If $a_k = O(k^{-\alpha})$ for some $\alpha > 1$,
\begin{equation}
    \delta_2^{X}(k) = O(k^{-\alpha}), \quad\quad \beta^{X}(k) = O\big(k^{-\alpha + \frac{1+\alpha}{1+\nu}+1} \vee (k^{-\alpha+\frac{3}{2}}\log(k)^{1/2})\big).\label{mixing_linearprocess}
\end{equation}
Note that even for this specific example, the calculation of the functional dependence measure is much easier and possible with much less assumptions. Moreover, bounds for $\beta^{X}(k)$ is typically larger than $\delta_2^{X}(k)$. The reason being the simple structure of $\delta_2^X$ compared to the much more involved formulation of dependence through sigma-algebras in the $\beta$-mixing coefficients. For recursively defined processes with a finite number of lags, $\delta_2^X$ are  typically upper bounded by geometric decaying coefficients (cf. \cite{wu2011}, \cite{dahlhaus2019}); the same holds true for $\beta^{X}(k)$ under additional continuity assumptions (cf. \cite{doukhan94}, Section 2.4. or \cite{kulik}, \cite{heinrich} among others). 

\subsubsection{Entropy integral}
In \cite{doukhan1995} (cf. also \cite{Dedeck02}), it was shown that if $X_i$, $i=1,...,n$, is stationary and $\beta$-mixing with coefficients $\beta(k)$, $k\in\N_0$, then
\begin{equation}
    \int_0^{1}\sqrt{\IH(\varepsilon,\sH, \|\cdot\|_{2,\beta})} d \varepsilon < \infty\label{mixing_entropycondition}
\end{equation}
implies weak convergence of $(\G_n(h))_{h\in \sH}$ in $\ell^{\infty}(\sH)$. Here, the $\|\cdot\|_{2,\beta}$-norm is defined as follows. If $\beta^{-1}$ denotes the inverse cadlag of the decreasing function $t \mapsto \beta(\lfloor t\rfloor)$ and $Q_h$ the inverse cadlag of the tail function $t \mapsto \IP(h(X_1) > t)$, then
\[
    \|h\|_{2,\beta} := \int_0^{1}\beta^{-1}(u) Q_h(u)^2 du.
\]
Condition \reff{mixing_entropycondition} was later relaxed in \cite[Theorem 8.3]{rio17}. It could be shown that if $\sF$ consists of indicator functions of specific classes of sets (in particular, $\sF$ corresponds to the empirical distribution function), weak convergence can be obtained under less restrictive conditions than \reff{mixing_entropycondition}. Since our theory does not directly allow us to analyze indicator functions because $\sF$ has to be a $(L_{\sF},s,R,C)$-class, we do not discuss their generalization here in detail.

In the special cases of polynomial and geometric decay, simple upper bounds for $\|h\|_{2,\beta}$ are available (cf. \cite{Dedeck02}). If $\sum_{k=0}^{\infty}k^{b-1}\beta(k) < \infty$ for some $b \ge 1$, then $\|\cdot \|_{2,\beta}$ is upper bounded by $\|\cdot \|_{\frac{2b}{b-1}}$.

Generally speaking, \reff{mixing_entropycondition} asks for $\frac{2b}{b-1}$ moments of the process $f(X_i)$ to exist while our condition in \reff{depmeas_entropycondition} only asks for $2$ moments of $f(X_i)$ but allows for smaller function classes through the additional factors given in the entropy integral (cf. Table \ref{table_v_values}). In specific examples (cf. \reff{mixing_linearprocess}) it may occur that the entropy integral \reff{depmeas_entropycondition} is finite while \reff{mixing_entropycondition} is infinite due to missing summability of $\beta^{X}(k)$.

To give a precise comparison, consider the situation of linear processes from \reff{mixing_linearprocess}. If $\nu >  2\alpha+1$, we can choose $b = \alpha-\frac{3}{2}$. Then, the two entropy integrals from Corollary \ref{cor_functional_central_limit_theorem_cor} (left) and \reff{mixing_entropycondition} read
\[
    \int_{0}^{1}\varepsilon^{-\frac{1}{\alpha}}\sqrt{\IH(\varepsilon,\sF, \|\cdot\|_{2})} d \varepsilon \quad\quad\text{vs.}\quad\quad \int_0^{1}\sqrt{\IH(\varepsilon,\sF, \|\cdot\|_{\frac{4\alpha-6}{2\alpha-5}})} d \varepsilon.
\]
Here, the entropy integral for mixing only exists if $\alpha > \frac{5}{2}$. The difference in the behavior is due to different bounds used for the variance of $\G_n(f)$.

\subsection{Integration into other empirical process results for the empirical distribution function of dependent data}

While our approach does allow for a general empirical process theory for H\"older continuous function classes, some more general dependence concepts already have been introduced (only) for the discussion of the empirical distribution function (EDF) based on the one-dimensional class $\sF = \{x \mapsto \Ii_{\{x \le t\}}:t \in \R\}$. We mention \cite{dehling2001}, \cite{dedecker10}, \cite{berkes09} and \cite{durieu14}. The conditions therein cover the case where $X_i = J(\sG_i)$ is a stationary Bernoulli shift with $\sG_i = (\varepsilon_i,\varepsilon_{i-1},...)$ and dependence is measured with the (stationary) functional dependence measure
\[
    \delta_{\nu}^{X}(k) = \|X_{i} - X_{i}^{*(i-k)}\|_{\nu}
\]
and its summed up version, $D_{\nu}(k) := \sum_{j=k+1}^{\infty}\delta_{\nu}^{X}(j)$. \cite{dehling2001} introduces so-called $\nu$-approximation coefficients $a_{k}$, $k\in\N$, which can be viewed as another formulation of functional dependence. However, their final result \cite[Theorem 5]{dehling2001} for the convergence of the EDF is stated with summability conditions \emph{both} on $a_k$ and absolutely regular mixing coefficients, we therefore do not discuss it in detail here. On the other hand, \cite[Theorem 2.1]{dedecker10} in combination with \cite[Section 6.1]{dedecker07} show convergence of the EDF if for some $\nu >1$ and $\gamma > 0$,
\[
    D_{\nu}(k) = O(k^{-\frac{(1+\gamma)(\nu+1)}{\nu}}).
\]
This is done by introducing simplified $\beta$-mixing coefficients which can then be upper bounded by $D_{\nu}(k)$. By using independent approximations of the original process, \cite[Theorem 1, Corollary 1]{berkes09} obtain convergence of the EDF if for some $\nu \ge 1$ and $A > 4$, $(D_{\nu}(k)k^{A})^{\nu} \le k^{-A}$, or equivalently,
\[
    D_{\nu}(k) \le k^{-\frac{A(\nu+1)}{\nu}}.
\]
\cite{durieu14} discusses convergence of the EDF under a general growth condition imposed on the moments of $\sum_{i=1}^{n}\{h(X_i) - \IE h(X_i)\}$ where $h \in \sH_{\alpha}$, the set of all H\"older-continuous functions with exponent $\alpha \in (0,1]$. Their condition is fulfilled if
\[
    \sum_{k=1}^{\infty}k^{2p-2}D_{\nu}(k)^{\alpha} < \infty
\]
for some $p > \frac{\nu}{\nu-1}$, $\nu > 1$.

\section{A general central limit theorem for locally stationary processes}
\label{sec_clt}

In this section, we introduce the remaining assumptions needed in Theorem \ref{cor_functional_central_limit_theorem} which pose regularity conditions on the process $X_i$ and the function class $\sF$ in time direction. They are used to derive a multivariate central limit theorem for $(\G_n(f_1),...,\G_n(f_k))$ under minimal moment conditions in Theorem \ref{theorem_clt_mult}. Comparable results in different and more specific contexts were shown in \cite{dahlhaus2019} or \cite{truquet}.

We first formalize the property of $X_i$ to be locally stationary (cf. \cite{dahlhaus2019}). We ask that for each $u\in[0,1]$ there exists a stationary process $\tilde X_i(u)$, $i=1,...,n$, such that $X_i \approx \tilde X_i(u)$ if $|u - \frac{i}{n}|$ is small. Recall $R(\cdot),s,p$ from Assumption \ref{ass1_vorher}.

\begin{assum}\label{ass_clt_process}
    For each $u\in[0,1]$, there exists a process $\tilde X_i(u) = J(\sG_i,u)$, $i\in\Z$, where $J$ is a measurable function. Furthermore, there exists some $C_X > 0$, $\varsigma \in (0,1]$ such that for every $i\in \{1,...,n\}$, $u_1,u_2\in [0,1]$,
    \[
        \|X_i - \tilde X_i(\frac{i}{n})\|_{\frac{2sp}{p-1}} \le C_X n^{-\varsigma}, \quad\quad \|\tilde X_i(u_1) - \tilde X_i(u_2)\|_{\frac{2sp}{p-1}} \le C_X|u_1-u_2|^{\varsigma}.
    \]
    For $\tilde Z_i(u) = (\tilde X_{i}(u), \tilde X_{i-1}(u),...)$ it holds that $\sup_{v,u}\|R(\tilde Z_0(v),u)\|_{2p} < \infty$.
\end{assum}

The behavior of the functions $f(z,u) = D_{f,n}(u)\cdot \bar f(z,u)$ of the class $\sF$ in the direction of time $u\in[0,1]$ is controlled by the following two continuity assumptions which state conditions on $\bar f(z,\cdot)$ and $D_{f,n}(\cdot)$ separately. 

\begin{assum}\label{ass_clt_fcont}
    There exists some $\varsigma \in (0,1]$ such that for every $\bar f\in \bar\sF$,
    \[
        \sup_{v,u_1,u_2}\Big\|\frac{|\bar f(\tilde Z_0(v),u_1) - \bar f(\tilde Z_0(v),u_2)|}{|u_1 - u_2|^{\varsigma}}\Big\|_2 < \infty.
    \]
\end{assum}

For $f \in \sF$, let $D^{\infty}_{f,n} := \sup_{i=1,...,n}D_{f,n}(\frac{i}{n})$.

\begin{assum}\label{ass_clt_d}
    For all $f\in \sF$, the function $\frac{D_{f,n}(\cdot)}{D_{f,n}^{\infty}}$ has bounded variation uniformly in $n$, and
    \begin{equation}
       \sup_{n\in\N}\frac{1}{n}\sum_{i=1}^{n}D_{f,n}(\frac{i}{n})^2 < \infty, \quad\quad \frac{D_{f,n}^{\infty}}{\sqrt{n}} \to 0.\label{ass_clt_d_tech_cond1}
    \end{equation}
    One of the two following cases hold.
    \begin{enumerate}
        \item Case $\IK=1$ (global): For all $f,g\in \sF$,  $u \mapsto \IE[\IE[\bar f(\tilde Z_{j_1}(u),u)|\sG_0]\cdot \IE[\bar g(\tilde Z_{j_2}(u),u)|\sG_0]]$ has bounded variation for all $j_1,j_2 \in\N_0$ and the following limit exists:
        \[
            \Sigma_{fg}^{(1)} := \lim_{n\to\infty}\int_0^{1}D_{f,n}(u)D_{g,n}(u) \cdot \sum_{j\in\Z}\text{Cov}(\bar f(\tilde Z_0(u), u), \bar g(\tilde Z_j(u),u)) du.
        \]
        \item Case $\IK=2$ (local): There exists a sequence $h_n > 0$ and $v\in [0,1]$ such that $\mathrm{supp} D_{f,n}(\cdot) \subset [v-h_n,v+h_n]$. It holds that
        \[
            h_n \to 0,\quad\quad \sup_{n\in\N}(h_n^{1/2}\cdot D_{f,n}^{\infty}) < \infty.
        \]
        The following limit exists for all $f,g\in \sF$:
        \[
            \Sigma_{fg}^{(2)} := \lim_{n\to\infty}\int_0^{1}D_{f,n}(u) D_{g,n}(u) du \cdot \sum_{j\in\Z}\text{Cov}(\bar f(\tilde Z_0(v), v), \bar g(\tilde Z_j(v),v)).
        \]
    \end{enumerate}
\end{assum}

Assumption \ref{ass_clt_d} looks rather technical. The first part including \reff{ass_clt_d_tech_cond1} guarantees the right normalization of $D_{f,n}(\cdot)$. The second part ensures the convergence of the asymptotic variances $\Var(\G_n(f))$ and covariances $\Cov(\G_n(f),\G_n(g))$.

We obtain the following central limit theorem.

\begin{thm}\label{theorem_clt_mult}
    Let $\sF$ satisfy Assumptions \ref{ass1_vorher}, \ref{ass_clt_process}, \ref{ass_clt_fcont} and  \ref{ass_clt_d}. Let $m\in\N$ and $f_1,...,f_m \in \sF$.  Let $\Sigma^{(\IK)} = (\Sigma^{(\IK)}_{f_kf_l})_{k,l=1,...,m}$. Then
    \[
        \frac{1}{\sqrt{n}}\sum_{i=1}^{n}\Big\{\begin{pmatrix}
            f_1(Z_i,\frac{i}{n})\\
            \vdots\\
            f_m(Z_i,\frac{i}{n})
        \end{pmatrix}- \E \begin{pmatrix}
            f_1(Z_i,\frac{i}{n})\\
            \vdots\\
            f_m(Z_i,\frac{i}{n})
        \end{pmatrix}\Big\} \dto N(0,\Sigma^{(\IK)}).
    \]
\end{thm}

Theorem \ref{theorem_clt_mult} generalizes the one-dimensional central limit theorem from \cite{dahlhaus2019}. We now comment on the assumptions. 
\begin{rem}
    Assumptions \ref{ass_clt_process}, \ref{ass_clt_fcont} and \ref{ass_clt_d} allow for very general structures of $f\in \sF$. However, in many special cases, a subset of them is  automatically fulfilled:
    \begin{itemize}
        \item If $X_i$ is stationary, then Assumption \ref{ass1_vorher} already implies Assumption \ref{ass_clt_process}.
        \item If $\bar f(z,u) = \bar f(z)$ does not depend on $u$, Assumption \ref{ass_clt_fcont} is fulfilled.
    \end{itemize}
    Regarding Assumption \ref{ass_clt_d} we have:
    \begin{itemize}
        \item If $D_{f,n}(u) = 1$, $X_i$ is stationary and $\bar f(z,u) = \bar f(z)$, then Assumption \ref{ass1_vorher} already implies Assumption \ref{ass_clt_d}(i) with $\Sigma_{fg}^{(1)} = \sum_{j\in\Z}\Cov(\bar f(Z_0),\bar f(Z_j))$.
        \item If $D_{f,n}(u) = 1$ and Assumption \ref{ass1_vorher}, \ref{ass_clt_fcont} hold with $s = \varsigma =1$, then Assumption \ref{ass_clt_d}(i) holds with $\Sigma_{fg}^{(1)} = \int_0^{1}\sum_{j\in\Z}\Cov(\bar f(\tilde Z_0(u),u),\bar f(\tilde Z_j(u),u)) du$.
        \item If $h_n \to 0$, $nh_n \to \infty$ and $D_{f,n}(u) = \frac{1}{\sqrt{h_n}}K(\frac{u-v}{h_n})$ with some Lipschitz-continuous kernel $K:\R\to\R$ with support $\subset[-1,1]$ and fixed $v\in (0,1)$, then Assumption \ref{ass_clt_d}(ii) holds with $\Sigma_{fg}^{(2)} = \int_0^{1} K(x)^2 dx\cdot \sum_{j\in\Z}\Cov(\bar f(\tilde Z_0(v),v),\bar f(\tilde Z_j(v),v))$.
    \end{itemize}
\end{rem}

\section{Maximal inequalities and asymptotic tightness under functional dependence}
\label{sec_max_asymptotic_tightness}

In this section, we provide the necessary ingredients for the proof of asymptotic tightness of $\G_n(f)$. We derive a new maximal inequality for finite $\sF$ under functional dependence in Theorem \ref{proposition_rosenthal_bound}. We then generalize this bound to arbitrary $\sF$ using chaining techniques in Section \ref{sec_asymptotic_tightness}.

\subsection{Maximal inequalities}
\label{sec_maximal}

We first derive a maximal inequality which is a main ingredient for chaining devices but also is of independent interest. To state the result, let
\[
    \beta(q) := \sum_{j=q}^{\infty}\Delta(k).
\]
and define
\[
    q^{*}(x) := \min\{q\in\N:\beta(q) \le q\cdot x\}.
\]
Set $D_{n}^{\infty}(u) := \sup_{f\in \sF}|D_{f,n}(u)|$. For $\nu \ge 2$, choose $\ID_{\nu, n}^{\infty}$ such that
\begin{equation}
   \Big(\frac{1}{n}\sum_{i=1}^{n}D_{n}^{\infty}(\frac{i}{n})^{\nu}\Big)^{1/\nu} \le \ID_{\nu,n}^{\infty}.\label{definition_dinf}
\end{equation}
Put $\ID_n^{\infty} = \ID_{2,n}^{\infty}$. Recall that $H = H(|\sF|) = 1 \vee \log |\sF|$ as in \reff{H_fix}. 

\begin{thm}\label{proposition_rosenthal_bound}
    Suppose that $\sF$ satisfies $|\sF| < \infty$ and Assumption \ref{ass1_vorher}. Then there exists some universal constant $c > 0$ such that the following holds: If $\sup_{f\in\sF}\|f\|_{\infty}\le M$ and $\sup_{f\in\sF}V_n(f)\le \sigma$, then
    \begin{equation}
    \E \max_{f\in \sF}\big|\G_n(f)\big| \le c \cdot \min_{q\in \{1,...,n\}} \Big[\sigma \sqrt{H} + \sqrt{H}\cdot\ID_{n}^{\infty}\beta(q) +\frac{q M H}{\sqrt{n}}\Big]\label{proposition_rosenthal_bound_result1}
\end{equation}
and
\begin{equation}
    \E \max_{f\in \sF}\big|\G_n(f)\big| \le 2c\cdot \Big(\sigma \sqrt{H} +  q^{*}\big(\frac{M\sqrt{H}}{\sqrt{n}\ID_{n}^{\infty}}\big)\frac{MH}{\sqrt{n}}\Big).\label{proposition_rosenthal_bound_result2}
\end{equation}
\end{thm}

Clearly, the second bound \reff{proposition_rosenthal_bound_result2} is a corollary of \reff{proposition_rosenthal_bound_result1} which balances the two terms which involve $q$. Values of $q^{*}(\cdot)$ for the two prominent cases that $\Delta(\cdot)$ is polynomial or exponential decaying can be found in Table \ref{table_r_values}. The proof of Theorem \ref{proposition_rosenthal_bound} relies on a decomposition of $\G_n(f)$ in i.i.d. parts and a residual term of martingale structure. Similar decompositions are also the core of empirical process results for Harris-recurrent Markov chains (cf. \cite{tjost16}) and mixing sequences (cf. \cite{Dedeck02}).

    \begin{table}[h!]
    \centering
        \begin{tabular}{l|l|l}
         & \multicolumn{2}{c}{$\Delta(j)$ }  \\
         &  $C j^{-\alpha}$, $\alpha > 1$         & $C \rho^j$, $\rho \in (0,1)$    \\ 
         \hline
         \hline
        $q^*(x)$ & $\max\{x^{-\frac{1}{\alpha}},1\}$ & $\max\{\log(x^{-1}),1\}$ \\ \hline
        $r(\delta)$ & $\min\{\delta^{\frac{\alpha}{\alpha-1}},\delta\}$ & $\min\{\frac{\delta}{\log(\delta^{-1})},\delta\}$
        \end{tabular}
        
        \caption{Equivalent expressions of $q^{*}(\cdot)$ and $r(\cdot)$ taken from Lemma \ref{r_d_values} in Section \ref{sec_V_imply_supp}. We omitted the lower and upper bound constants which are only depending on $C, \rho, \alpha$.}
        \label{table_r_values}
    \end{table}

In the next subsections, we will prove asymptotic tightness for $\G_n(f)$ under the condition that $\ID_n^{\infty}$, $\ID_n$ do not depend on $n$. However, uniform convergence rates of $\G_n(f)$ for finite $\sF$ (growing with $n$) can be obtained without these conditions but with additional moment assumptions, which is done in the following Corollary \ref{proposition_rosenthal_bound_rates}. To incorporate the additional moment assumptions, we use a slightly stronger assumption than Assumption \ref{ass1_vorher}.

\begin{assum}\label{ass1}
    $\bar\sF$ is a $(L_{\sF},s,R,C)$-class. There exists $\nu \ge 2$, $p\in (1,\infty]$, $C_X > 0$ such that
    \begin{equation}
        \sup_{i,u}\| R(Z_{i},u)\|_{\nu p} \le C_R, \quad\quad \sup_{i,j}\|X_{ij}\|_{\frac{\nu sp}{p-1}} \le C_X.\label{ass1_eq1}
    \end{equation}
    Let $\ID_n \ge 0$, $\Delta(k) \ge 0$ be such that for all $k\in\N_0$,
    \[
        2d C_R \cdot \sum_{j = 0}^kL_{\sF,j} (\delta_{\frac{\nu sp}{p-1}}^X(k-j))^s \le \Delta(k), \quad\quad \sup_{f\in \sF} \Big(\frac{1}{n}\sum_{i=1}^{n}\big|D_{f,n}(\frac{i}{n})\big|^{2}\Big)^{1/2} \le \ID_{n}.
    \]
\end{assum}

Note that Assumption \ref{ass1_vorher} is obtained by taking $\nu=2$. For $\delta > 0$, let
\[
    r(\delta) := \max\{r > 0: q^{*}(r)r \le \delta\},
\]
cf. Table \ref{table_r_values} for values of $r(\cdot)$ in special cases. 

\begin{cor}[Uniform convergence rates]\label{proposition_rosenthal_bound_rates}
    Suppose that $\sF$ satisfies $|\sF| < \infty$ and Assumption \ref{ass1} for some $\nu \ge 2$. Let $C_{\Delta} := 4d\cdot |L_{\sF}|_1\cdot C_X^{s}C_R + C_{\bar{f}}$. Furthermore, suppose that
    \begin{equation}
        \sup_{n\in\N}\sup_{f\in\sF}V_n(f) < \infty, \quad\quad \sup_{n\in\N}\frac{\ID_{\nu,n}^{\infty}}{\ID_{n}^{\infty}} < \infty,\quad\quad\quad \sup_{n\in\N}\frac{C_{\Delta} H}{n^{1-\frac{2}{\nu}}r(\frac{\sigma}{\ID_n^{\infty}})^2} < \infty.\label{proposition_rosenthal_bound_rates_cond1}
    \end{equation}
    Then,
    \[
        \max_{f\in\sF}|\G_n(f)| = O_p(\sqrt{H}).
    \]
\end{cor}
The first condition in \reff{proposition_rosenthal_bound_rates_cond1} guarantees that $\G_n(f)$ is properly normalized. The second and third condition are needed to prove that the ``rare events'', where  $|f(Z_i,\frac{i}{n})|$ exceeds some threshold $M_n\in (0,\infty)$, are of the same order as $\sqrt{H}$. For this, we may need more than two moments of $f(Z_i,\frac{i}{n})$, that is, $\nu > 2$, depending on $\sqrt{H}$ and the behavior of $\ID_n^{\infty}$.

Corollary \ref{proposition_rosenthal_bound_rates} can be used to prove (optimal) convergence rates for kernel density and regression estimators as well as maximum likelihood estimators under dependence. We give an example in Section \ref{sec_examples}.

\subsection{Asymptotic tightness}
\label{sec_asymptotic_tightness}

In this section, we extend the maximal inequality from Theorem \ref{proposition_rosenthal_bound} to arbitrary (infinite) classes $\sF$. Since Assumption \ref{ass1_vorher} forces $f\in \sF$ to be H\"older-continuous with respect to its first argument $z$, classical chaining approaches which use indicator functions do not apply here. We provide a new chaining technique which preserves continuity in Section \ref{sec_maximal_chaining}. 

For $n\in\N$, $\delta > 0$ and $k\in\N$ define $H(k) = 1 \vee \log(k)$ and 
\begin{equation}
    m(n,\delta,k) := r(\frac{\delta}{\ID_n})\cdot \frac{\ID_n^{\infty}n^{1/2}}{H(k)^{1/2}}.\label{lemma_chaining_definition_m}
\end{equation}

Here, $m(n,\delta,k)$ represents the threshold for rare events in the chaining procedure. We have the following result.

\begin{thm}\label{thm_equicont}
    Let $\sF$ satisfy Assumption \ref{ass1_vorher} and let $F$ be some envelope function of $\sF$, that is, for each $f\in \sF$ it holds that $|f| \le F$. Let $\sigma > 0$ and assume that $\sup_{f\in \sF}V_n(f) \le \sigma$. Then there exists some universal constant $\tilde c > 0$ such that 
    \begin{eqnarray*}
        &&\E \sup_{f\in \sF}\big|\G_n(f)\big|\\
        &\le& \tilde c \Big[ (1 + \frac{\ID_n^{\infty}}{\ID_n} + \frac{\ID_n}{\ID_n^{\infty}})\int_0^{\sigma} \sqrt{\IH\big(\eps,\sF,V_n\big)} \, \mathrm{d}\eps + \sqrt{n}\big\|F\Ii_{\{F > \frac{1}{4}m(n,\sigma,\N(\frac{\sigma}{2},\sF,V_n))\}}\big\|_{1,n}\Big].
    \end{eqnarray*}
\end{thm}

As a corollary, we obtain asymptotic equicontinuity of $\G_n(f)$. Here, we use Assumptions \ref{ass_clt_process} and \ref{ass_clt_fcont} only to discuss the remainder term in Theorem \ref{thm_equicont} without imposing the existence of additional moments.

\begin{cor}\label{cor_equicont}
    Let $\sF$ satisfy Assumption \ref{ass1_vorher}, \ref{ass_clt_process} and \ref{ass_clt_fcont}. Suppose that
    \begin{equation}
        \sup_{n\in\N}\int_0^{1}\sqrt{\IH(\varepsilon,\sF,V_n)} d\varepsilon < \infty.\label{cor_equicont_cond0}
    \end{equation}
    Furthermore, assume that $\ID_n,\ID_n^{\infty} \in (0,\infty)$ are independent of $n$, and
    \begin{equation}
        \sup_{i=1,...,n}\frac{D_{n}^{\infty}(\frac{i}{n})}{\sqrt{n}} \to 0.\label{cor_equicont_cond1}
    \end{equation}
    Then, the process $\G_n(f)$ is equicontinuous with respect to $V_n$, that is, for every $\eta > 0$,
    \[
        \lim_{\sigma \to 0}\limsup_{n \to \infty} \IP\Big( \sup_{f,g \in \sF, V_n(f-g) \leq \sigma} |\G_n(f) - \G_n(g)| \geq \eta \Big) = 0.
    \]
\end{cor}

\section{Applications}
\label{sec_examples}

In this section, we provide some applications of the main results (Corollary \ref{proposition_rosenthal_bound_rates} and Corollary \ref{cor_functional_central_limit_theorem}). We will focus on locally stationary processes and therefore use localization in our functionals, but the results also hold for stationary processes, accordingly. 

Let $K:\R\to\R$ be some bounded kernel function which is Lipschitz continuous with Lipschitz constant $L_K$, $\int K(u) du = 1$, $\int K(u)^2 du \in (0,\infty)$ and support $\subset[-\frac{1}{2},\frac{1}{2}]$. For some bandwidth $h > 0$, put $K_h(\cdot) := \frac{1}{h}K(\frac{\cdot}{h})$.

In the first example we consider the nonparametric kernel estimator in the context of nonparametric regression with fixed design and locally stationary noise. We show that under conditions on the bandwidth $h$, which are common in the presence of dependence (cf. \cite{hansen2008} or \cite{vogt2012}), we obtain the optimal uniform convergence rate $\sqrt{\frac{\log(n)}{nh}}$. Write $a_n \gtrsim b_n$ for sequences $a_n,b_n$ if there exists some constant $c > 0$ such that $a_n \ge c b_n$ for all $n\in\N$.

\begin{ex}[Nonparametric Regression]\label{example_regeression}
   Let $X_i$ be some arbitrary process of the form \reff{representation_x} with $\sum_{k=0}^{\infty}\delta^{X}_2(k) < \infty$ which fulfills $\sup_{i=1,...,n}\|X_i\|_{\nu} \le C_X \in (0,\infty)$ for some $\nu > 2$. Suppose that we observe $Y_i$, $i=1,...,n$ given by
   \[
        Y_i = g(\frac{i}{n}) + X_i,
   \]
   where $g:[0,1]\to \R$ is some function. Estimation of $g$ is performed via
   \[
        \hat g_{n,h}(v) := \frac{1}{n} \sum_{i = 1}^nK_h(\frac{i}{n}-v)Y_i.
   \]
   Suppose that either
   \begin{itemize}
       \item $\delta_2^{X}(j) \le \kappa j^{-\alpha}$ with some $\kappa > 0, \alpha > 1$, and $h \gtrsim  (\frac{\log(n)}{n^{1-\frac{2}{\nu}}})^\frac{\alpha-1}{\alpha}$, or
       \item $\delta_2^{X}(j) \le \kappa \rho^{j}$ with some $\kappa > 0, \rho \in (0,1)$ and $h \gtrsim \frac{\log(n)^3}{n^{1-\frac{2}{\nu}}}$.
   \end{itemize}
   From \reff{example_regeression_eq1} and \reff{example_regeression_eq2} below it follows that
   \[
        \sup_{v\in[0,1]}|\hat g_{n,h}(v) - \IE \hat g_{n,h}(v)| = O_p\big(\sqrt{\frac{\log(n)}{nh}}\big).
   \]
   First note that due to Lipschitz continuity of $K$ with Lipschitz constant $L_K$, we have
   \begin{eqnarray}
        &&\sup_{|v-v'|\le n^{-3}}\big|(\hat g_{n,h}(v) - \IE \hat g_{n,h}(v)) - (\hat g_{n,h}(v') - \IE \hat g_{n,h}(v'))\big|\nonumber\\
        &\le& \cdot \frac{L_K n^{-3}}{nh^2}\sum_{i=1}^{n}\big(|X_i| + \IE|X_i|\big) = O_p( n^{-1} ).\label{example_regeression_eq1}
   \end{eqnarray}
   For the grid $V_n = \{in^{-3}, i = 1,...,n^{3} \}$, which discretizes $[0,1]$ up to distances $n^{-3}$, we obtain by Corollary \ref{proposition_rosenthal_bound_rates} that 
   \begin{equation}
        \sqrt{nh}\sup_{v\in V_n}|\hat g_{n,h}(v) - \IE \hat g_{n,h}(v)| = \sup_{f\in \sF}|\G_n(f)| = O_p\big(\sqrt{\log |V_n| }\big) = O_p\big(\log(n)^{1/2}\big),\label{example_regeression_eq2}
   \end{equation}
   where
   \[
        \sF = \{f_v(x,u) = \frac{1}{\sqrt{h}}K(\frac{u - v}{h})x: v \in V_n\}.
    \] 
    The conditions of Corollary \ref{proposition_rosenthal_bound_rates} are easily verified: It holds that $f_v(x,u) = D_{f,n}(u)\cdot \bar f_v(x,u)$ with $D_{f,n}(u) = \frac{1}{\sqrt{h}}K(\frac{u - v}{h})$ and $\bar f_v(x,u) = x$. Thus, Assumption \ref{ass1} is satisfied with $\Delta(k) = 2\delta_2^{X}(k)$, $p = \infty$, $R(\cdot) = C_R = 1$. Furthermore, $\ID_n = |K|_{\infty}, \ID_{\nu,n} = \frac{|K|_{\infty}}{\sqrt{h}}$, and
    \[
        \norm{f_v}_{2,n} \le \frac{1}{\sqrt{h}} \Big( \frac{1}{n} \sum_{i = 1}^n K(\frac{v - u}{h})^2\norm{X_i}_2^2 \Big)^{1/2} \le C_X |K|_\infty,
    \]
    which shows that $\sup_{f\in\sF}V_n(f) = O(1)$. The conditions on $h$ emerge from the last condition in \reff{proposition_rosenthal_bound_rates_cond1} and using the bounds for $r(\cdot)$ from Table \ref{table_r_values}.

\end{ex}

For the following two examples we suppose that the underlying process $X_i$ is locally stationary in the sense of Assumption \ref{ass_clt_process}. Similar assumptions are posed in \cite{dahlhaus2019} and are fulfilled for a large variety of locally stationary processes.

In the same spirit as in Example \ref{example_regeression}, it is possible to derive uniform rates of convergence for M-estimators of parameters $\theta$ in models of locally stationary processes. Furthermore, weak Bahadur representations can be obtained. The following results apply for instance to maximum likelihood estimation of parameters in tvARMA or tvGARCH processes. The main tool is to prove uniform convergence of the corresponding objective functions and its derivatives. Since the rest of the proof is  standard, the details are postponed to the Supplementary Material, Section \ref{sec_examples_supp}. Let $\nabla_{\theta}^j$ denote the $j$-th derivative with respect to $\theta$. To apply empirical process theory, we ask for the objective functions to be $(L_{\sF},1,R,C)$-classes in (A1) and Lipschitz with respect to $\theta$ in (A2).

\begin{lem}[M-estimation, uniform results]\label{example_maximumlikelihood}
    Let $\Theta \subset \R^{d_{\Theta}}$ be compact and  $\theta_0:[0,1] \to \text{interior}(\Theta)$. For each $\theta\in\Theta$, let $\ell_{\theta}:\R^{k} \to \R$ be some measurable function which is twice continuously differentiable. Let $Z_i = (X_i,...,X_{i-k+1})$, and define for $v\in [0,1]$,
    \[
        \hat \theta_{n,h}(v) := \arg\min_{\theta \in \Theta}L_{n,h}(v,\theta), \quad\quad L_{n,h}(v,\theta) := \frac{1}{n}\sum_{i=k}^{n}K_h\big(\frac{i}{n}-v\big)\cdot \ell_{\theta}(Z_i)
    \]
    Suppose that there exists $C_{\Theta} > 0$ such that for $j\in \{0,1,2\}$,
    \begin{itemize}
        \item[(A1)] $\bar\sF_{j} = \{\nabla_{\theta}^j\ell_{\theta}:\theta \in \Theta\}$ is an $(L_{\sF},1,R,C)$-class with $R(z) = 1+|z|_1^{M-1}$ for some $M \ge 1$ and Assumption \ref{ass_clt_process} for $\bar \sF_j$ is fulfilled with $s=1$, $p=\frac{M}{M-1}$.
        \item[(A2)] for all $z\in\R^{k}$, $\theta,\theta' \in \Theta$,
    \[
        \big|\nabla_{\theta}^j\ell_{\theta}(z) - \nabla_{\theta}^j\ell_{\theta'}(z)\big|_{\infty} \le C_{\Theta}(1+|z|_1^M)\cdot |\theta - \theta'|_2,
    \]
        \item[(A3)] $\theta \mapsto \E \ell_{\theta}(\tilde Z_0(v))$ attains its global minimum in $\theta_0(v)$ with positive definite $I(v) := \IE \nabla_{\theta}^2 \ell_{\theta}(\tilde Z_0(v))$.
    \end{itemize}
    Furthermore, suppose that either
   \begin{itemize}
       \item $\delta_{2M}^{X}(j) \le \kappa j^{-\alpha}$ with some $\kappa > 0, \alpha > 1$, and $h \gtrsim  (\frac{\log(n)}{n^{1-\frac{2}{\nu}}})^\frac{\alpha-1}{\alpha}$, or
       \item $\delta_{2M}^{X}(j) \le \kappa \rho^{j}$ with some $\kappa > 0, \rho \in (0,1)$ and $h \gtrsim \frac{\log(n)^3}{n^{1-\frac{2}{\nu}}}$.
   \end{itemize}
    Define $\tau_{n} := \sqrt{\frac{\log(n)}{nh}}$ and $B_h := \sup_{v\in[0,1]}|\IE \nabla_{\theta} L_{n,h}(v,\theta_0(v))|$ (the bias). Then, $B_h = O(h^{\varsigma})$, and as $nh\to\infty$,
    \[
        \sup_{v \in [\frac{h}{2},1-\frac{h}{2}]}\big|\hat \theta_{n,h}(v) - \theta_0(v)\big| = O_p\big(\tau_n + B_h\big)
    \]
    and
    \[
        \sup_{v\in[\frac{h}{2},1-\frac{h}{2}]}\big|\{\hat \theta_{n,h}(v) - \theta_0(v)\} - I(v)^{-1}\nabla_{\theta} L_{n,h}(v,\theta_0(v))\big| = O_p((\tau_n + h^{\varsigma})(\tau_n + B_h)).
    \]
\end{lem}
\begin{rem}
    \begin{itemize}
        \item In the tvAR(1) case $X_i = a(i/n) X_{i-1} + \varepsilon_i$, we can use for instance
    \[
        \ell_{\theta}(x_1,x_0) = (x_1 - a x_0)^2,
    \]
    which for $a\in (-1,1)$ is a $((1,a),1,|x_0|+|x_1|,(0,1))$-class.
        \item With more smoothness assumptions on $\nabla_{\theta}\ell$ or using a local linear estimation method for $\hat \theta_{n,h}$, the bias term $B_h$ can be shown to be of smaller order, for instance $O(h^2)$ (cf. \cite{dahlhaus2019}).
        \item The theory derived in this paper can also be used to prove asymptotic properties of M-estimators based on objective functions $\ell_{\theta}$ which are only almost everywhere differentiable in the Lebesgue sense by following the theory of chapter 5 in \cite{Vaart98}. This is of utmost interest for $\ell_{\theta}$ that have additional analytic properties, such as convexity. Since these properties are also needed in the proofs, we will not discuss this in detail.
    \end{itemize}
\end{rem}

We give an easy application of the functional central limit theorem from Theorem \ref{cor_functional_central_limit_theorem} by inspecting a local stationary version of Example 19.25 in \cite{Vaart98}.

\begin{ex}[Local mean absolute deviation]\label{example_mad}
    For fixed $v \in (0,1)$, put $ \overline{X_n}(v) := \frac{1}{n}K_h\big(\frac{i}{n}-v\big)X_i$ and define the mean absolute deviation
    \[
        \text{mad}_n(v) := \frac{1}{n}\sum_{i=1}^{n}K_h\big(\frac{i}{n}-v\big)|X_i - \overline{X_n}(v)|.
    \]
    Let Assumption \ref{ass_clt_process} hold with $s=1$, $p= \infty$. Suppose that $\IP(\tilde X_0(v) = \IE \tilde X_0(v))=0$ and that for some $\kappa > 0, \alpha > 1$, $\delta_2^{X}(j) \le \kappa j^{-\alpha}$. We show that if $nh\to \infty$ and $nh^{1+2\varsigma}\to 0$,
    \begin{equation}
        \sqrt{nh}\big(\text{mad}_n(v) - \IE|\tilde X_0(v) - \mu|\big) \dto N(0, \sigma^2),\label{example_mad_eq0}
    \end{equation}
    where $\mu = \IE \tilde X_0(v)$, $G$ denotes the distribution function of $\tilde X_0(v)$ and
    \begin{eqnarray*}
        \sigma^2 &=& \int K(u)^2 du \cdot \sum_{j=0}^{\infty}\Cov\big(|\tilde X_0(v) - \mu| + (2G(\mu)-1)\tilde X_0(v),\\
        &&\quad\quad\quad\quad\quad\quad\quad\quad\quad\quad\quad\quad |\tilde X_j(v) - \mu| + (2G(\mu)-1)\tilde X_j(v)\big).
    \end{eqnarray*}
    The result is obtained by using the decomposition
    \begin{eqnarray*}
        &&\sqrt{nh}\big(\text{mad}_n(v) - \IE|\tilde X_0(v) - \mu|\big) = \G_n(f_{\overline{X_n}(v)} - f_{\mu}) + \G_n(f_{\mu}) + A_n,\\
        &&\quad\quad\quad\quad A_n = \frac{\sqrt{nh}}{n}\sum_{i=1}^{n}K_h\big(\frac{i}{n}-v\big)\{\IE|X_i - \theta| - \IE|\tilde X_0(v) - \mu|\big\}\Big|_{\theta = \overline{X_n}(v)},
    \end{eqnarray*}
    where $\Theta = \{\theta \in \R: |\theta - \mu| \le 1\}$ and 
    \[
        \sF = \{f_{\theta}(x,u) = \sqrt{h}K_h(u-v)|x-\theta|: \theta \in \Theta\}.
    \]
    By the triangle inequality, $\sF$ satisfies Assumption \ref{ass1_vorher} with $\bar f_{\theta}(x,u) = |x-\theta|$, $R(\cdot) = C_R = 1$, $p=\infty$, $s=1$ and $\Delta(k) = 2\delta_{2}^{X}(k)$. Assumption \ref{ass_clt_fcont} is trivially fulfilled since $\bar f$ does not depend on $u$. Since $\sF$ is a one-dimensional Lipschitz class, $\sup_{n\in\N}\IH(\varepsilon,\sF, \|\cdot\|_{2,n}) = O(\log(\varepsilon^{-1}\vee 1))$. By Corollary \ref{cor_functional_central_limit_theorem}, we obtain that there exists some process $[\G(f_{\theta})]_{\theta \in \Theta}$ such that for $h \to 0$, $nh \to \infty$,
    \begin{equation}
        \big[\G_n(f_{\theta})\big]_{\theta \in \Theta} \dto \big[\G(f_{\theta})\big]_{\theta\in \Theta}\quad\quad \text{ in }\ell^{\infty}(\Theta).\label{example_mad_eq01}
    \end{equation}
    Furthermore, by Assumption \ref{ass_clt_process},
    \begin{eqnarray}
        &&\|f_{\overline{X_n}(v)}(X_i) - f_{\mu}(X_i)\|_2\nonumber\\
        &\le& \|\overline{X_n}(v) - \mu\|_2 \le \|\overline{X_n}(v) - \IE \overline{X_n}(v)\|_2 + \|\IE \overline{X_n}(v) - \mu\|_2\nonumber\\
        &\le& \frac{1}{\sqrt{nh}}\Big(\frac{1}{nh}\sum_{i=1}^{n}K\big(\frac{\frac{i}{n}-v}{h}\big)^2\Big)^{1/2}\sum_{j=0}^{\infty}\delta_2^{X}(j) + \frac{1}{n}\sum_{i=1}^{n}K_h(\frac{i}{n}-v)\big|\IE X_i - \IE \tilde X_0(v)|\nonumber\\
        &=& O((nh)^{-1/2} + h^{\varsigma}).\label{example_mad_eq02}
    \end{eqnarray}
    By Lemma 19.24 in \cite{Vaart98}, we conclude from \reff{example_mad_eq01} and \reff{example_mad_eq02} that
    \begin{equation}
        \G_n(f_{\overline{X_n}(v)} - f_{\mu}) \pto 0.\label{example_mad_eq1}
    \end{equation}
    
    By Assumption \ref{ass_clt_process} and bounded variation of $K$,
    \begin{equation}
        A_n = \sqrt{nh}\big\{\IE|\tilde X_0(v) - \theta|\big|_{\theta = \overline{X_n}(v)} - \IE|\tilde X_0(v) - \mu|\big\} + O_p( (nh)^{-1/2} + (nh)^{1/2}h^{-\varsigma}).\label{example_mad_eq2}
    \end{equation}
    Due to $\IP(\tilde X_0(v) = \mu) = 0$, $g(\theta) = \IE|\tilde X_0(v)-\theta|$ is differentiable in $\theta = \mu$ with derivative $2G(\mu)-1$. The Delta method delivers
    \begin{eqnarray}
        &&\sqrt{nh}\big\{\IE|\tilde X_0(v) - \theta|\big|_{\theta = \overline{X_n}(v)} - \IE|\tilde X_0(v) - \mu|\big\}\nonumber\\
        &=& (2G(\mu)-1)\sqrt{nh}(\overline{X_n}(v)-\mu) + o_p(1).\label{example_mad_eq3}
    \end{eqnarray}
    From \reff{example_mad_eq1}, \reff{example_mad_eq2} and \reff{example_mad_eq3} we obtain
    \[
        \sqrt{nh}\big(\text{mad}_n(v) - \IE|\tilde X_0(v) - \mu|\big) = \G_n(f_{\mu} + (2G(\mu)-1)\text{id}) + o_p(1).
    \]
    Theorem \ref{theorem_clt_mult} now yields \reff{example_mad_eq0}.
\end{ex}

\section{Conclusion}
\label{sec_conclusion} In this paper, we have developed a new empirical process theory for locally stationary processes with the functional dependence measure. We have proven a functional central limit theorem and  maximal inequalities. A general empirical process theory for locally stationary processes is a key step to derive asymptotic and nonasymptotic results for M-estimates or testing based on $L^{2}$- or $L^{\infty}$-statistics. We have given an example in nonparametric estimation where our theory is applicable. Due to the possibility to analyze the size of the function class and the stochastic properties of the underlying process separately, we conjecture that our theory also permits an extension of various results from i.i.d. to dependent data, such as empirical risk minimization.

From a technical point of view, the linear and moment-based nature of the functional dependence measure has forced us to modify several approaches from empirical process theory for i.i.d. or mixing variables. A main issue was given by the fact that the dependence measure only transfers decay rates for continuous functions. We therefore have provided a new chaining technique which preserves continuity of the arguments of the empirical process.

In principle, a similar empirical process theory for locally stationary processes can be established under mixing conditions such as absolute regularity. This would be a generalization of the results found in \cite{rio1995} and \cite{Dedeck02}. As we have seen in Section \ref{sec_discussion_mixing}, such a theory would pose additional moment conditions on $f(Z_i,\frac{i}{n})$. Contrary to that, our framework only requires second moments of $f(Z_i,\frac{i}{n})$, but the entropy integral is enlarged by some factor which increases with stronger dependence. Moreover, in nearly all models the derivation of a bound for mixing coefficients needs continuity of the innovation process which may not be suitable in several examples. Therefore, we consider our theory as a valuable addition to this existing theory even in the stationary case.

One could also think of an extension of our empirical process theory to functions $f(z,u)$ which are noncontinuous with respect to $z$. This can in principle be done by using a martingale decomposition and assuming continuity of $z \mapsto \IE[f(Z_i,\frac{i}{n})|\sG_{i-1}=z]$, instead. However, one typically can only expect continuity of this functional if either already $z \mapsto f(z,u)$ was continuous or $\varepsilon_i$ has a continuous density. In the latter case, the sequence might also be $\beta$-mixing, and a more detailed discussion about advantages of our formulation is necessary.

\section*{Acknowledgements}

The authors would like to thank the associate editor and two anonymous referees for their helpful remarks which helped to provide a much more concise version of the paper.

\section{Appendix}

In the appendix, we present the basic ideas used to prove the maximal inequalities of Section \ref{sec_max_asymptotic_tightness}. We first consider the finite version in Section \ref{sec_max_finite} and then present a chaining approach which preserves continuity in Section \ref{sec_maximal_chaining}.

\subsection{Proof idea: A maximal inequality for finite $\sF$, Theorem \ref{proposition_rosenthal_bound}}
\label{sec_max_finite}

We provide an approach to obtain maximal inequalities for sums of random variables $W_i(f)$, $i = 1,...,n$, indexed by $f\in\sF$, by using a decomposition into independent random variables. An approach with similar intentions is presented in \cite{Dedeck02} (Section 4.3 therein) for absolutely regular sequences and in \cite{tjost16} for Harris-recurrent Markov chains. For convenience, we abbreviate
\[
    W_i(f) = f(Z_i,\frac{i}{n})
\]
and put $S_n(f) := \sum_{i=1}^{n}W_i(f)$.

To approximate $W_i(f)$ by independent variables, we use a technique from \cite{Wu13} which was refined in \cite{wuzhang2017}. This decomposition is much more involved then the ones for Harris-recurrent Markov chains or mixing sequences since no direct coupling method is available. Define 
\[
    W_{i,j}(f) := \E[W_i(f)|\varepsilon_{i-j},\varepsilon_{i-j+1},...,\varepsilon_i],\quad\quad j\in\N,
\]
and
\[
    S_n(f):= \sum_{i = 1}^n \{W_i(f) - \IE W_i(f)\}, \quad\quad S_{n,j}(f) := \sum_{i=1}^{n}\{W_{i,j}(f) - \IE W_{i,j}(f)\}.
\]
Let $q \in\{1,...,n\}$ be arbitrary. Put $L := \lfloor\frac{\log(q)}{\log(2)}\rfloor$ and $\tau_l := 2^l$ ($l=0,...,L-1$), $\tau_L := q$. Then we have
\[
    W_i(f) = W_i(f) - W_{i,q}(f) + \sum_{l=1}^{L}(W_{i,\tau_{l}}(f)-W_{i,\tau_{l-1}}(f)) + W_{i,1}(f)
\]
(in the case $q = 1$, the sum in the middle does not appear) and thus
\[
    S_n(f) = \big[S_n(f) - S_{n,q}(f)\big] + \sum_{l=1}^{L}\big[S_{n,\tau_l}(f) - S_{n,\tau_{l-1}}(f)\big] + S_{n,1}(f).
\]
We write
\[
    S_{n,\tau_l}(f) - S_{n,\tau_{l-1}}(f) = \sum_{i=1}^{\lfloor \frac{n}{\tau_l}\rfloor+1}T_{i,l}(f),\quad\quad T_{i,l}(f) := \sum_{k=(i-1)\tau_l+1}^{(i\tau_l)\wedge n}\big[W_{k,\tau_l}(f) - W_{k,\tau_{l-1}}(f)\big].
\]
The random variables $T_{i,l}(f), T_{i',l}(f)$ are independent if $|i-i'| > 1$. This leads to the decomposition
\begin{eqnarray}
    \max_{f\in \sF}\big|\G_n(f)\big| &\le& \max_{f\in \sF}\frac{1}{\sqrt{n}}\big|S_n(f) - S_{n,q}(f)\big|\nonumber\\
    && + \sum_{l=1}^{L}\Big[\max_{f\in \sF}\Big|\frac{1}{\sqrt{\frac{n}{\tau_l}}}\underset{i \text{ even}}{\sum_{i=1}^{\lfloor \frac{n}{\tau_l}\rfloor + 1}}\frac{1}{\sqrt{\tau_l}}T_{i,l}(f)\Big| + \max_{f\in \sF}\Big|\frac{1}{\sqrt{\frac{n}{\tau_l}}}\underset{i \text{ odd}}{\sum_{i=1}^{\lfloor \frac{n}{\tau_l}\rfloor + 1}}\frac{1}{\sqrt{\tau_l}}T_{i,l}(f)\Big|\Big]\nonumber\\
    && + \max_{f\in \sF}\frac{1}{\sqrt{n}}\big|S_{n,1}^{W}(f)\big|.\label{rosenthal_main_decomposition}
\end{eqnarray}

While the first term in \reff{rosenthal_main_decomposition} can be made small by assumptions on the dependence of $W_i(f)$ and by the use of a large deviation inequality for martingales in Banach spaces from \cite{pinelis1994}, the second and third term allow the application of Rosenthal-type bounds due to the independency of the summands $T_{i,l}(f)$ and $W_{i,1}(f)$, respectively. Since the first term in \reff{rosenthal_main_decomposition} allows for a stronger bound in terms of $n$ than it is the case for mixing, we can obtain a theory which only needs second moments of $W_i(f) = f(X_i,\frac{i}{n})$. By Assumption \ref{ass1_vorher}, we can show the following results (cf. Lemma \ref{depend_trans} in the Supplementary Material and recall \reff{definition_dinf} for the definition of $D_n^{\infty}$). For each $i=1,...,n$, $j\in\N$, $s\in\N \cup\{\infty\}$, $f\in\sF$,
\begin{eqnarray}
    \Big\| \sup_{f\in \sF}\big|W_i(f) - W_i(f)^{*(i-j)}\big|\,\Big\|_2 &\le&  D_{n}^{\infty}(\frac{i}{n})\Delta(j),\label{wproof_eq1}\\
    \big\|W_i(f) - W_i(f)^{*(i-j)}\big\|_2 &\le& |D_{f,n}(\frac{i}{n})|\cdot \Delta(j),\label{wproof_eq2}\\
    \big\|W_i(f)\|_s &\le& \big\|f(Z_i,\frac{i}{n})\big\|_s.\label{wproof_eq3}
\end{eqnarray}
We now summarize the proof of Theorem \ref{proposition_rosenthal_bound}. The detailed proof is found in the Supplementary material. Denote
\[
    A_1 := \max_{f\in \sF}\frac{1}{\sqrt{n}}\big|S_n^{W}(f) - S_{n,q}^{W}(f)\big|, \quad\quad A_2 := \sum_{l=1}^{L}\Big[\max_{f\in \sF}\Big|\frac{1}{\sqrt{\frac{n}{\tau_l}}}\underset{i \text{ even}}{\sum_{i=1}^{\lfloor \frac{n}{\tau_l}\rfloor + 1}}\frac{1}{\sqrt{\tau_l}}T_{i,l}(f)\Big|.
\]
The remaining terms in \reff{rosenthal_main_decomposition} are discussed similarly or are special cases. We first have
\[
        \E A_1 \le \sum_{j=q}^{\infty}\frac{1}{\sqrt{n}}\E \max_{f\in \sF}\Big|\sum_{i=1}^n E_{i,j}(f)\Big|,
    \]
    where $E_{i,j}(f) = \E[W_i(f)|\varepsilon_{i-j},...,\varepsilon_i] - \E[W_i(f)|\varepsilon_{i-j+1},...,\varepsilon_i]$ is a martingale difference sequence with respect to $\sG^{i} = \sigma(\varepsilon_{i-j},\varepsilon_{i-j+1},...)$. For $x = (x_f)_{f\in \sF}$ and $s := 2 \vee \log|\sF|$, define $|x|_s := (\sum_{f\in \sF}|x_f|^s)^{1/s}$. Then
    \[
        \E \max_{f\in \sF}\Big|\sum_{i=1}^n E_{i,j}\Big| \le \IE \Big|\sum_{i=1}^n E_{i,j}\Big|_{s} \le \Big\| \Big|\sum_{i=1}^n E_{i,j}\Big|_{s}\Big\|_2.
    \]
    By Theorem 4.1 in \cite{pinelis1994} there exists an absolute constant $c_1 > 0$ such that for $s > 1$,
    \[
        \Big\| \Big|\sum_{i=1}^n E_{i,j}(f)\Big|_{s}\Big\|_2 \le c_1 \sqrt{s}\Big(\sum_{i=1}^{n}\big\| |E_{i,j}|_s \big\|_2^2 \Big)^{1/2}.
    \]
    By using \reff{wproof_eq1} and standard techniques for the functional dependence measure, we conclude that
    \[
        \big\| |E_{i,j}|_s \|_2 \le s^{1/s}\big\| \sup_{f\in \sF}|E_{i,j}(f)| \big\|_2 \le e D_{n}^{\infty}(\frac{i}{n})\Delta(j),
    \]
    Summarizing the results, we obtain
    \begin{equation}
        \IE A_1 \le 2ec_1 \cdot \sqrt{H}\cdot \ID_n^{\infty}\beta(q).\label{wproof_eq4}
    \end{equation}
    Regarding $A_2$, we have
    \[
        \IE A_2 \le \sum_{l=1}^{L}\Big[\IE\max_{f\in \sF}\Big|\frac{1}{\sqrt{\frac{n}{\tau_l}}}\underset{i \text{ even}}{\sum_{i=1}^{\lfloor \frac{n}{\tau_l}\rfloor + 1}}\frac{1}{\sqrt{\tau_l}}T_{i,l}(f)\Big|\Big].
    \]
    Using the fact that martingale sequences are uncorrelated, \reff{wproof_eq2} and the simple bound $|W_i(f)| \le M$, one can show that 
    \[
        \|T_{i,l}\|_2 \le \sum_{j=\tau_{l-1} + 1}^{\tau_l}\min\big\{\|f\|_{2,n}, \ID_n \Delta(\lfloor \frac{j}{2}\rfloor)\big\}, \quad\quad \frac{1}{\sqrt{\tau_l}}|T_{i,l}(f)| \le 2\sqrt{\tau_l} M.
    \]
    A maximal inequality for independent random variables based on Bernstein's inequality yields that there exists some universal constant $c_2 > 0$ such that
    \begin{eqnarray*}
        &&\IE\max_{f\in \sF}\Big|\frac{1}{\sqrt{\frac{n}{\tau_l}}}\underset{i \text{ even}}{\sum_{i=1}^{\lfloor \frac{n}{\tau_l}\rfloor + 1}}\frac{1}{\sqrt{\tau_l}}T_{i,l}(f)\Big|\\
        &\le& c_2\cdot \Big[\max_{f}\Big(\frac{1}{\frac{n}{\tau_l}} \underset{i \text{ even}}{\sum_{i=1}^{\lfloor \frac{n}{\tau_l}\rfloor + 1}}\norm{\frac{1}{\sqrt{\tau_l}}T_{i,l}(f)}_{2}^2 \Big)^{1/2} \sqrt{H} + \frac{\sup_{f\in \sF}\big\|\frac{1}{\sqrt{\tau_l}}|T_{i,l}(f)|\big\|_{\infty} H}{\sqrt{\frac{n}{\tau_l}}}\Big]\\
        &\le& 4c_2 \Big[\Big(\sum_{j=\tau_{l-1}+1}^{\tau_l} \min\{\max_{f\in\sF}\|f\|_{2,n},\ID_{n}\Delta(\lfloor\frac{j}{2}\rfloor )\}\Big)\cdot \sqrt{H} + \frac{\tau_l M H}{\sqrt{n}}\Big].
    \end{eqnarray*}
    By monotonicity of the first term with respect to $\|f\|_{2,n}$ and
    \[
        \sum_{l=1}^{L}\sum_{j=\tau_{l-1}+1}^{\tau_l} \min\{\|f\|_{2,n},\ID_{n}\Delta(\lfloor\frac{j}{2}\rfloor )\} \le 2 V_n(f), \quad\quad \sum_{l=1}^{L}\tau_l \le 2q,
    \]
    we obtain with some universal constant $c_3 > 0$,
    \[
        \IE A_2 \le c_3\Big(\sup_{f\in \sF}V_n(f) \sqrt{H} + \frac{q MH}{\sqrt{n}}\Big) \le c_3\Big(\sigma \sqrt{H} + \frac{q MH}{\sqrt{n}}\Big),
    \]
    which together with \reff{wproof_eq4} provides the result of Theorem \ref{proposition_rosenthal_bound}.

\subsection{An elementary chaining approach which preserves continuity}
\label{sec_maximal_chaining}

In this section, we provide a chaining approach which preserves continuity of the functions inside the empirical process. Typical chaining approaches work with indicator functions which is not suitable for application of Theorem \ref{proposition_rosenthal_bound}. We replace the indicator functions by suitably chosen truncations. For $m > 0$, define $\varphi_m^{\wedge}:\R \to \R$ and the corresponding ``peaky'' residual $\varphi_m^{\vee}:\R \to \R$ via  
\[
    \varphi_{m}^{\wedge}(x) := (x\vee (-m))\wedge m, \quad\quad 
    \varphi_m^{\vee}(x) := x - \varphi_m^{\wedge}(x).
\]

In the following, assume that for each $j\in\N_0$ there exists a decomposition $\sF = \bigcup_{k=1}^{N_j}\sF_{jk}$, where $(\sF_{jk})_{k=1,...,N_j}$, $j\in\N_0$ is a sequence of nested partitions. For each $j\in\N_0$ and $k\in \{1,...,N_j\}$, choose a fixed element $f_{jk} \in \sF_{jk}$. For $j\in\N_0$, define $\pi_j f := f_{jk}$ if $f \in \sF_{jk}$.

Assume furthermore that there exists a sequence $(\Delta_{j}f)_{j\in\N}$ such that for all $j\in\N_0$, $\sup_{f,g\in \sF_{jk}}|f-g| \le \Delta_j f$. Finally, let  $(m_j)_{j\in\N_0}$ be a decreasing sequence which will serve as a truncation sequence.

For $j\in\N_0$, we use the decomposition
\begin{eqnarray*}
    f - \pi_j f= \varphi_{m_j}^{\wedge}(f-\pi_j f) + \varphi_{m_j}^{\vee}(f-\pi_j f)
\end{eqnarray*}
Since
\begin{eqnarray}
    f-\pi_jf &=& f-\pi_{j+1}f + \pi_{j+1}f-\pi_j f\nonumber\\
    &=& \varphi_{m_{j+1}}^{\wedge}(f-\pi_{j+1}f) + \varphi_{m_{j+1}}^{\vee}(f-\pi_{j+1}f)\nonumber\\
    &&\quad\quad+ \varphi_{m_j - m_{j+1}}^{\wedge}(\pi_{j+1}f-\pi_j f) + \varphi_{m_j - m_{j+1}}^{\vee}(\pi_{j+1}f-\pi_j f),\label{truncation_decomposition_eq0}
\end{eqnarray}
we can write
\begin{equation}
    \varphi_{m_j}^{\wedge}(f-\pi_j f) = \varphi_{m_{j+1}}^{\wedge}(f-\pi_{j+1}f) + \varphi_{m_j-m_{j+1}}^{\wedge}(\pi_{j+1}f-\pi_j f) + R(j),\label{truncation_decomposition_eq1}
\end{equation}
where
\[
    R(j) := \varphi_{m_j}^{\wedge}(f-\pi_j f) - \varphi_{m_j}^{\wedge}(\varphi_{m_{j+1}}^{\wedge}(f-\pi_{j+1}f)) - \varphi_{m_j}^{\wedge}(\varphi_{m_j-m_{j+1}}^{\wedge}(\pi_{j+1}f-\pi_j f)).
\]

To bound $R(j)$, we use (i) of the following elementary Lemma \ref{lemma_properties_truncationfunction} which is proved in Section \ref{sec_maximal_chaining_supp} included in the Supplementary Material.

\begin{lem}\label{lemma_properties_truncationfunction}
    Let $y,x,x_1,x_2,x_3$ and $m,m' > 0$ be real numbers. Then the following assertions hold:
    \begin{enumerate}
        \item[(i)] If $|x_1| + |x_2| \le m$, then
        \[
            \big|\varphi_m^{\wedge}(x_1+x_2+x_3) - \varphi_m^{\wedge}(x_1) - \varphi_m^{\wedge}(x_2)\big| \le \min\{|x_3|,2m\}.
        \]
        \item[(ii)] $|\varphi_m^{\wedge}(x)| \le \min\{|x|,m\}$ and if $|x| < y$,
        \[
            |\varphi_m^{\vee}(x)| \le \varphi_m^{\vee}(y) \le  y\Ii_{\{y>m\}}.
        \]
        \item[(iii)] If $\sF$ fulfills Assumption \ref{ass1}, then Assumption \ref{ass1} also holds for $\{\varphi_{m}^{\wedge}(f): f\in \sF\}$ and $\{\varphi_{m}^{\vee}(f): f\in \sF\}$.
    \end{enumerate}
\end{lem}

Because the partitions are nested, we have $|\pi_{j+1} f - \pi_j f|  \le \Delta_{j}f$. By Lemma \ref{lemma_properties_truncationfunction} and \reff{truncation_decomposition_eq0}, we have
\begin{eqnarray}
    |R(j)| &\le& \min\big\{\big|\varphi_{m_{j+1}}^{\vee}(f-\pi_{j+1}f) + \varphi_{m_j-m_{j+1}}^{\vee}(\pi_{j+1}f-\pi_{j}f)\big|,2m_j\big\}\nonumber\\
    &\le& \min\big\{\big|\varphi_{m_{j+1}}^{\vee}(\Delta_{j+1}f)\big|,2m_j\big\}+\min\big\{\big|\varphi_{m_j-m_{j+1}}^{\vee}(\Delta_{j}f)\big|,2m_j\big\}.\label{truncation_decomposition_eq2}
\end{eqnarray}
Let $\tau \in\N$. We then have with iterated application of \reff{truncation_decomposition_eq1} and linearity of $f \mapsto W_i(f)$,
\begin{eqnarray}
    && \G_n( \varphi_{m_0}^{\wedge}(f - \pi_0 f) )\nonumber\\
    &=& \G_n( \varphi_{m_1}^{\wedge}(f-\pi_{1} f) ) + \G_n( \varphi_{m_0 - m_{1}}^{\wedge}(\pi_1f - \pi_0 f)) + \G_n(R(0))\nonumber\\
    &=& \G_n(\varphi_{m_\tau}^{\wedge}(f-\pi_\tau f)) + \sum_{j=0}^{\tau-1}\G_n(\varphi_{m_j-m_{j+1}}^{\wedge}(\pi_{j+1}f - \pi_j f)) + \sum_{j=0}^{\tau-1}\G_n(R(j)),\label{truncation_decomposition_eq3}
\end{eqnarray}
which in combination with \reff{truncation_decomposition_eq2} can now be used for chaining. The following lemma provides the necessary balancing between the truncated versions of $\G_n(f)$ and the rare events excluded. Recall that $H(k) = 1 \vee \log(k)$ as in \reff{H_var}.

\begin{lem}[Compatibility lemma]\label{lemma_chaining} If $\sF$ fulfills $|\sF| \le k$ and Assumption \ref{ass1}, then  $\sup_{f\in \sF}V_n(f) \le \delta$,  $\sup_{f\in \sF}\|f\|_{\infty} \le m(n,\delta,k)$ imply
\begin{equation}
    \E \max_{f\in\sF}\big|\G_n(f)\big| \le c(1+\frac{\ID_n^{\infty}}{\ID_n}) \delta \sqrt{H(k)},\label{proposition_rosenthal_bound_implication2}
   \end{equation}
   and $\sup_{f\in \sF}V_n(f) \le \delta$ implies that for each $\gamma > 0$,
   \begin{equation}
    \sqrt{n}\|f \Ii_{\{f > \gamma\cdot m(n,\delta,k)\}}\|_{1,n} \le \frac{1}{\gamma}\frac{\ID_n}{\ID_n^{\infty}}\delta \sqrt{H(k)}.\label{proposition_rosenthal_bound_implication3}
\end{equation}
\end{lem}
\begin{proof}[Proof of Lemma \ref{lemma_chaining}]
For $q\in \N$, put $\beta_{norm}(q) := \frac{\beta(q)}{q}$. By Theorem \ref{proposition_rosenthal_bound} and the definition of $r(\cdot)$, 
    \begin{eqnarray*}
        \E \max_{f\in\sF}\big|\G_n(f)\big| &\le& c\Big( \delta\sqrt{H(k)} + q^{*}\Big(\frac{m(n,\delta,k)\sqrt{H(k)}}{\sqrt{n}\ID_n^{\infty}}\Big)\frac{m(n,\delta,k) H(k)}{\sqrt{n}}\Big)\\
        &=& c\Big( \delta\sqrt{H(k)} + \ID_n^{\infty} q^{*}(r(\frac{\delta}{\ID_n}))r(\frac{\delta}{\ID_n}) \sqrt{H(k)}\Big)\\
        &=& c(1+\frac{\ID_n^{\infty}}{\ID_n}) \delta \sqrt{H(k)}.
    \end{eqnarray*}
    which shows \reff{proposition_rosenthal_bound_implication2}. Since
    \[
        \|f(Z_i,\frac{i}{n})\Ii_{\{f(Z_i,\frac{i}{n}) > \gamma m(n,\delta,k)\}}\|_1 \le \frac{1}{\gamma m(n,\delta,k)}\|f(Z_i,\frac{i}{n})^2\|_1 = \frac{1}{\gamma m(n,\delta,k)}\|f(Z_i,\frac{i}{n})\|_2^2,
    \]
    for all $f\in \sF$ with $V_n(f) \le \delta$, it holds that
    \begin{eqnarray}
        \sqrt{n}\| f \Ii_{\{f > \gamma m(n,\delta,k)}\|_{1,n}  \le \frac{\sqrt{n}}{\gamma m(n,\delta,k)}\|f\|_{2,n}^2 \le \frac{1}{\gamma}\frac{\|f\|_{2,n}^2}{\ID_n^{\infty}r(\frac{\delta}{\ID_n})}\sqrt{H(k)}. \label{proposition_rosenthal_bound_implication3_part0}
    \end{eqnarray}
    
    If $\|f\|_{2,n} \ge \ID_n\Delta(1)$, we have
    \begin{equation}
        V_n(f) = \|f\|_{2,n} + \ID_n\sum_{j=1}^{\infty}\Delta(j) \ge \|f\|_{2,n} + \ID_n\beta(1).\label{proposition_rosenthal_bound_implication3_part1}
    \end{equation}
    In the case $\|f\|_{2,n} < \ID_n\Delta(1)$, the fact that $\Delta(\cdot)$ is decreasing implies that $a^{*} = \max\{j \in\N: \|f\|_{2,n} < \ID_n\Delta(j)\}$ is well-defined. We conclude that 
    \begin{eqnarray}
        V_n(f) &=& \|f\|_{2,n} + \sum_{j=0}^{\infty}\|f\|_{2,n} \wedge (\ID_n\Delta(j)) = \|f\|_{2,n} + \sum_{j=1}^{a^{*}}\|f\|_{2,n} + \ID_n\sum_{j=a^{*}+1}^{\infty}\Delta(j)\nonumber\\
        &=& \|f\|_{2,n}(a^{*}+1) + \ID_n\beta(a^{*}) \ge \|f\|_{2,n} a^{*} + \beta(a^{*}).\label{proposition_rosenthal_bound_implication3_part2}
    \end{eqnarray}
    Summarizing the results \reff{proposition_rosenthal_bound_implication3_part1} and \reff{proposition_rosenthal_bound_implication3_part2}, we have
    \[
        V_n(f) \ge \|f\|_{2,n} (a^{*} \vee 1) + \ID_n\beta(a^{*}\vee 1).
    \]
    We conclude that
    \[
        V_n(f) \ge \min_{a\in\N}\big[\|f\|_{2,n} a + \ID_n\beta(a)\big] \ge \|f\|_{2,n} \hat a + \ID_n\beta(\hat a),
    \]
    where $\hat a = \arg\min_{j\in\N}\big\{\|f\|_{2,n} \cdot j + \ID_n\beta(j)\big\}$.
    
    Since $\delta \ge V_n(f)$, we have $\delta \ge \ID_n\beta(\hat a) = \ID_n \beta_{norm}(\hat a) \hat a$. Thus $\beta_{norm}(\hat a) \le \frac{\delta}{\ID_n\hat a}$. By definition of $q^{*}$, $q^{*}(\frac{\delta}{\ID_n\hat a}) \le \hat a$. Thus $q^{*}(\frac{\delta}{\ID_n\hat a})\frac{\delta}{\ID_n\hat a} \le \frac{\delta}{\ID_n}$. By definition of $r(\cdot)$, $r(\frac{\delta}{\ID_n}) \ge \frac{\delta}{\ID_n\hat a}$. We conclude with $\|f\|_{2,n} \le V_n(f) \le \delta$ that
    \begin{equation}
        \frac{\|f\|_{2,n}^2}{\ID_n^{\infty} r(\frac{\delta}{\ID_n})} \le \frac{\ID_n\hat a \|f\|_{2,n}^2}{\ID_n^{\infty}\delta} \le \frac{\ID_n V_n(f) \|f\|_{2,n}}{\ID_n^{\infty}\delta} \le \frac{\ID_n}{\ID_n^{\infty}}\|f\|_{2,n} \le \frac{\ID_n}{\ID_n^{\infty}}\delta.\label{proposition_rosenthal_bound_implication3_bound2norm}
    \end{equation}
    
    Inserting the result into \reff{proposition_rosenthal_bound_implication3_part0}, we finally obtain that for all $f\in \sF$ with $V_n(f) \le \delta$ it holds that
    \[
        \sqrt{n}\| f \Ii_{\{f > \gamma m(n,\delta,k)}\|_{1,n} \le \frac{\sqrt{n}}{\gamma m(n,\delta,k)}\|f\|_{2,n}^2 \le \frac{1}{\gamma}\frac{\|f\|_{2,n}^2}{\ID_n^{\infty}r(\frac{\delta}{\ID_n})}\sqrt{H(k)} \le \frac{1}{\gamma}\frac{\ID_n}{\ID_n^{\infty}}\delta \sqrt{H(k)}.
    \]
    which shows \reff{proposition_rosenthal_bound_implication3}.
\end{proof}

\subsection{Proof idea: A maximal inequality for infinite $\sF$, Theorem \ref{thm_equicont}}

The details of the proof are given in Section \ref{sec_asymptotic_tightness_supp} in the Supplementary material. In the following, we abbreviate $\IH(\delta) = \IH(\delta,\sF, V_n)$ and $\N(\delta) = \N(\delta,\sF, V_n)$. Choose $\delta_0 = \sigma$ and $\delta_j = 2^{-j}\delta_0$.
    
For each $j\in\N_0$, we choose a covering by brackets $\sF_{jk} := [l_{jk},u_{jk}] \cap \sF$, $k = 1,...,N_j := \N(\delta_j)$ such that $V_n(u_{jk} - l_{jk}) \le \delta_j$ and $\sup_{f,g \in \sF_{jk}}|f-g| \le u_{jk} - l_{jk} =: \Delta_{jk}$. We may assume w.l.o.g. that $l_{jk}, u_{jk}, \Delta_{jk} \in \sF$ and that $(\sF_{jk})_k$ are nested.

In each $\sF_{jk}$, fix some $f_{jk} \in \sF$, and define $\pi_{j}f := f_{jk}$ and $\Delta_{j}f := \Delta_{jk}$. Put
    \[
        I(\sigma) := \int_0^{\sigma}\sqrt{\IH(\varepsilon,\sF, V_n)} d \varepsilon,\quad\quad  \tau := \min\Big\{j \ge 0: \delta_j \le \frac{I(\sigma)}{\sqrt{n}}\Big\} \vee 1.
    \]
    The chaining procedure is now applied with $m_j := \frac{1}{2}m(n,\delta_j, N_{j+1})$ ($m(\cdot)$ from \reff{lemma_chaining_definition_m}). Choose $M_n = \frac{1}{2}m_0$. We then have
    \begin{equation}
        \E \sup_{f\in \sF}\big|\G_n(f)\big| \le \E \sup_{f\in \sF(M_n)}\big|\G_n(f)\big| + \frac{1}{\sqrt{n}}\sum_{i=1}^{n}\E\big[W_i(F\Ii_{\{F > M_n\}})\big],\label{trunction_bound_paper}
    \end{equation}
    where $\sF(M_n) := \{\varphi_{M_n}^{\wedge}(f):f\in \sF\}$.
    
 By \reff{truncation_decomposition_eq2},  \reff{truncation_decomposition_eq3}, $|\G_n(f)| \le \G_n(g) + \frac{2}{\sqrt{n}}\sum_{i=1}^{n}\|W_i(g)\|_1 \le \G_n(g) + 2 \sqrt{n} \|g\|_{1,n}$ for $|f| \le g$, we obtain the decomposition
\begin{eqnarray}
    \sup_{f\in\sF(M_n)}|\G_n(f)| &\le&  \sup_{f\in\sF}|\G_n(\pi_0 f)|\nonumber\\
    &&\quad\quad + \Big\{\sup_{f\in\sF}|\G_n(\varphi_{m_\tau}^{\wedge}(\Delta_\tau f))| + 2\sqrt{n}\sup_{f\in \sF}\|\Delta_\tau f\|_{1,n}\Big\}\nonumber\\
    &&\quad\quad + \sum_{j=0}^{\tau-1}\sup_{f\in\sF}\Big|\G_n(\varphi_{m_j-m_{j+1}}^{\wedge}(\pi_{j+1}f - \pi_j f))\Big|\nonumber\\
    &&\quad\quad + \sum_{j=0}^{\tau-1}\Big\{\sup_{f\in\sF}\Big|\G_n(\min\big\{\big|\varphi_{m_{j+1}}^{\vee}(\Delta_{j+1}f)\big|,2m_j\big\})\Big|\nonumber\\
    &&\quad\quad\quad\quad\quad\quad\quad\quad\quad\quad+ 2\sqrt{n}\sup_{f\in\sF}\| \Delta_{j+1}f \Ii_{\{\Delta_{j+1}f > m_{j+1}\}}\|_{1,n}\Big\}\nonumber\\
    &&\quad\quad + \sum_{j=0}^{\tau-1}\Big\{\sup_{f\in\sF}\Big|\G_n(\min\big\{\big|\varphi_{m_j-m_{j+1}}^{\vee}(\Delta_{j}f)\big|,2m_j\big\})\Big|\nonumber\\
    &&\quad\quad\quad\quad\quad\quad\quad\quad\quad\quad + 2\sqrt{n}\sup_{f\in\sF}\| \Delta_{j}f \Ii_{\{\Delta_{j}f > m_j - m_{j+1}\}}\|_{1,n}\Big\}\nonumber\\
    &=:& R_1 + R_2 + R_3 + R_4 + R_5.\label{truncation_decomposition_final_paper}
\end{eqnarray}
The terms $R_i$, $i=1,...,5$ can now be discussed separately with Lemma \ref{lemma_chaining} and give the upper bounds (with constants $C > 0$): 
\[
    \IE R_1 \le C \delta_0 \sqrt{\IH(\delta_1)}, \quad \IE R_2 \le 2 I(\sigma) + C\delta_{\tau}\sqrt{\IH(\delta_{\tau+1})}, \quad \IE R_3, \IE R_4, \IE R_5 \le C \sum_{j=0}^{\tau}\delta_j \sqrt{\IH(\delta_{j+1})},
\]
and thus $\IE \sup_{f\in \sF(M_n)}|\G_n(f)| \le C'  I(\sigma)$. Together with \reff{trunction_bound_paper}, we result follows.

\bibliographystyle{plain}
\bibliography{reference}

\newpage

\begin{center}
{\Large \bf Supplementary Material}\\
\end{center}

This material contains some details of the proofs in the paper as well as the proofs of the examples.

\subsection{Proofs of Section \ref{sec_maximal}}
\label{sec_maximal_supp}

\begin{lem} \label{depend_trans}
    Let Assumption \ref{ass1} hold for some $\nu\ge 2$. Then, 
    \begin{eqnarray*}
        \delta_{\nu}^{f(Z,u)}(k) &\le& |D_{f,n}(u)|\cdot \Delta(k),\\
        \sup_{i}\Big\|\sup_{f\in \sF}\big|f(Z_i,u) - f(Z_i^{*(i-j)},u)\big|\,\Big\|_{\nu} &\le& D_{n}^{\infty}(u)\cdot \Delta(k),\\
        \sup_{i}\|f(Z_i,u)\|_{\nu} &\le& |D_{f,n}(u)|\cdot C_{\Delta},
    \end{eqnarray*}
    where $C_{\Delta} := 4d\cdot|L_{\sF}|_1\cdot C_X^s C_R + C_{\bar f}$.
\end{lem}
\begin{proof}[Proof of Lemma \ref{depend_trans}]
We have for each $f\in \sF$ and $\nu \ge 2$ that
\begin{eqnarray*}
    && \sup_{i}\norm{\bar f(Z_i,u)-\bar f(Z_i^{*(i-k)},u)}_{\nu}\\
    &\leq& \sup_{i} \norm{|Z_i-Z_i^{*(i-k)}|_{L_{\sF,s}}^{s}\big(R(Z_i,u) + R(Z_i^{*(i-k)},u)\big)}_{\nu} \\
    &\le& \sup_{i} \norm{\big| Z_i-Z_i^{*(i-k)}\big|_{L_{\sF,s}}^{s}}_{\frac{p}{p-1}\nu}\norm{R(Z_i,u) + R(Z_i^{*(i-k)},u)}_{p\nu} \\
    &\le& \sup_{i}\norm{\sum_{j = 0}^\infty L_{\sF,j}\big|X_{i-j}-X_{i-j}^{*(i-k)}\big|_{\infty}^s }_{\frac{p}{p-1}\nu} \left(\norm{R(Z_i,u)}_{p\nu} + \norm{R(Z_i^{*(i-k)},u)}_{p\nu}\right) \\
    &\le& 2d C_R \sum_{j = 0}^kL_{\sF,j} (\delta_{\frac{p}{p-1}\nu s}^X(k-j))^s.
\end{eqnarray*}
This shows the first assertion. Due to
\[
    \sup_{f\in \sF}\big|\bar f(Z_i,u)-\bar f(Z_i^{*(i-k)},u)\big| \le |Z_i-Z_i^{*(i-k)}|_{L_{\sF,s}}^{s}\big(R(Z_i,u) + R(Z_i^{*(i-k)},u)\big),
\]
the second assertion follows similarly.
The last assertion follows from
\[
    |\bar f(z,u)| \le |\bar f(z,u) - \bar f(0,u)| + |\bar f(0,u)| \le |z|_{L_{\sF},s}^s\cdot (R(z,u) + R(0,u)) + |\bar f(0,u)|
\]
which implies
\begin{eqnarray*}
    \|\bar f(Z_i,u)\|_{\nu} &\le& \Big\|\sum_{j=0}^{\infty}L_{\sF,j}|Z_{i-j}|_{\infty}^s\Big\|_{\frac{p}{p-1}\nu}\big(\big\|R(Z_i,u)\big\|_{pq} + R(0,u)\big) + |\bar f(0,u)|\\
    &\le& 2d\cdot|L_{\sF}|_1\cdot C_X^s\cdot (C_R +  |R(0,u)|) + |\bar f(0,u)|\\
    &\le& 4d\cdot|L_{\sF}|_1\cdot C_X^s\cdot C_R + C_{\bar f}.
\end{eqnarray*}
\end{proof}

\begin{proof}[Proof of Theorem \ref{proposition_rosenthal_bound}]
Denote the three terms on the right hand side of \reff{rosenthal_main_decomposition} by $A_1,A_2,A_3$. We now discuss the three terms separately. First, we have
\[
        \E A_1 \le \sum_{j=q}^{\infty}\frac{1}{\sqrt{n}}\E \max_{f\in \sF}\Big|\sum_{i=1}^{n}(W_{i,j+1}(f) - W_{i,j}(f))\Big|.
    \]
    For fixed $j$, the sequence
    \begin{eqnarray*}
        E_{i,j} := (E_{i,j}(f))_{f\in \sF} &=& \big((W_{i,j+1}(f) - W_{i,j}(f))\big)_{f\in \sF}\\
        &=& (\E[W_i(f)|\varepsilon_{i-j},...,\varepsilon_i] - \E[W_i(f)|\varepsilon_{i-j+1},...,\varepsilon_i])_{f\in \sF}
    \end{eqnarray*}
    is a  $|\sF|$-dimensional martingale difference vector with respect to $\sG^{i} = \sigma(\varepsilon_{i-j},\varepsilon_{i-j+1},...)$. For a vector $x = (x_f)_{f\in \sF}$ and $s \ge 1$, write $|x|_s := (\sum_{f\in \sF}|x_f|^s)^{1/s}$. By Theorem 4.1 in \cite{pinelis1994} there exists an absolute constant $c_1 > 0$ such that for $s > 1$, 
    \begin{equation}
        \Big\| \Big|\sum_{i=1}^{n}E_{i,j}\Big|_s \Big\|_2 \le c_1\Big\{2\Big\|\sup_{i=1,...,n}|E_{i,j}|_s\Big\|_2 + \sqrt{2(s-1)}\Big\|\Big(\sum_{i=1}^{n}\E[|E_{i,j}|_s^2|\sG^{i-1}]\Big)^{1/2}\Big\|_2\Big\}.\label{prop_alternativebound_eq1}
    \end{equation}
    We have
    \[
        \Big\|\sup_{i=1,...,n}|E_{i,j}|_s\Big\|_2 = \Big\|\big(\sup_{i=1,...,n}|E_{i,j}|_s^2\big)^{1/2}\Big\|_2 \le \Big\|\big(\sum_{i=1}^{n}|E_{i,j}|_s^2\big)^{1/2}\Big\|_2,
    \]
    therefore both terms in \reff{prop_alternativebound_eq1} are of the same order and it is enough to bound the second term in \reff{prop_alternativebound_eq1}. We have
    \begin{eqnarray}
        \Big\|\Big(\sum_{i=1}^{n}\E[|E_{i,j}|_s^2|\sG^{i-1}]\Big)^{1/2}\Big\|_2 &=& \Big\|\sum_{i=1}^{n}\E[|E_{i,j}|_s^2|\sG^{i-1}]\Big\|_{1}^{1/2}\nonumber\\
        &\le& \Big(\sum_{i=1}^{n}\big\|\E[|E_{i,j}|_s^2|\sG^{i-1}]\big\|_{1}\Big)^{1/2}\nonumber\\
        &\le& \Big(\sum_{i=1}^{n}\big\||E_{i,j}|_s\big\|_{2}^2\Big)^{1/2}.\label{prop_alternativebound_eq2}
    \end{eqnarray}
    Note that
    \begin{eqnarray}
        E_{i,j}(f) &=& W_{i,j+1}(f) - W_{i,j}(f) = \E[W_i(f)|\varepsilon_{i-j},...,\varepsilon_i] - \E[W_i(f)|\varepsilon_{i-j+1},...,\varepsilon_i]\nonumber\\
        &=& \E[W_i(f)^{**(i-j)} - W_i(f)^{**(i-j+1)}|\sG_i],\label{prop_alternativebound_eq2_definitionstarstar}
    \end{eqnarray}
    where $H(\sF_i)^{**(i-j)} := H(\sF_i^{**(i-j)})$ and $\sF_i^{**(i-j)} = (\varepsilon_i,\varepsilon_{i-1},...,\varepsilon_{i-j},\varepsilon_{i-j-1}^{*},\varepsilon_{i-j-2}^{*},...)$.
    
    By Jensen's inequality, Lemma \ref{depend_trans} and the fact that $(W_i(f)^{**(i-j)},W_i(f)^{**(i-j+1)})$ has the same distribution as $(W_i(f),W_i(f)^{*(i-j)})$, 
    \begin{eqnarray}
    	\| |E_{i,j}|_s\big\|_2 &=& |\Big\|\Big(\sum_{f\in \sF}|E_{i,j}(f)|^s\Big)^{1/s}\Big\|_2\nonumber\\
    	&\le& s^{1/s} \Big\|\sup_{f\in \sF}\big|\E[W_i(f)^{**(i-j)} - W_i(f)^{**(i-j+1)}|\sG_i]\big|\Big\|_2\nonumber\\
    	&\le& e \cdot \Big\|\E\big[ \sup_{f\in \sF}\big|W_i(f)^{**(i-j)} - W_i(f)^{**(i-j+1)}\big| \, \big|\sG_i\big]\Big\|_2\nonumber\\
    	&\le& e \cdot \Big\| \sup_{f\in \sF}\big|W_i(f)^{**(i-j)} - W_i(f)^{**(i-j+1)}\big|\,\Big\|_2\nonumber\\
    	&=& e \cdot \Big\| \sup_{f\in \sF}\big|W_i(f) - W_i(f)^{*(i-j)}\big|\,\Big\|_2\nonumber\\
	    &\le& e\cdot D_{n}^{\infty}(\frac{i}{n})\Delta(j).\label{prop_alternativebound_eq3}
    \end{eqnarray}
    
    Inserting \reff{prop_alternativebound_eq3} into \reff{prop_alternativebound_eq2} delivers
    \[
        \Big(\sum_{i=1}^{n}\big\||E_{i,j}|_s\big\|_{2}^2\Big)^{1/2} \le  e\Big(\sum_{i = 1}^n D_{n}^{\infty}(\frac{i}{n})^2\Big)^{1/2} \Delta(j),
    \]
    Inserting this bound into \reff{prop_alternativebound_eq1}, we obtain
    \[
    	\Big\| \Big|\sum_{i=1}^{n}E_{i,j}\Big|_s \Big\|_2 \le 4e c_1s^{1/2}n^{1/2}\Big(\frac{1}{n}\sum_{i = 1}^n D_{n}^{\infty}(\frac{i}{n})^2\Big)^{1/2}\Delta(j).
    \]
    We conclude with $s := 2 \vee \log|\sF|$ that
    \begin{eqnarray}
    	\E A_1 &\le& \frac{1}{\sqrt{n}}\sum_{k=q}^{\infty}\Big\| \Big| \sum_{i=1}^{n}E_{i,j}\Big|_s \Big\|_2\nonumber\\
    	&\le& 4e c_1  \cdot \sqrt{2 \vee \log|\sF|}\cdot \Big(\frac{1}{n}\sum_{i = 1}^n D_{n}^{\infty}(\frac{i}{n})^2\Big)^{1/2}\sum_{j=q}^{\infty}\Delta_p(j)\nonumber\\
	&\le& 8e c_1 \cdot \sqrt{H}\cdot  \ID_{n}^{\infty}\beta(q).\label{proof_rosenthal_a1_final}
    \end{eqnarray}
    
We now discuss $\E A_2$. If $M_Q,\sigma_Q > 0$ are constants and $Q_i(f)$, $i=1,...,m$ mean-zero independent variables (depending on $f\in \sF$) with $|Q_i(f)| \le M_Q$, $(\frac{1}{m}\sum_{i=1}^{m}\|Q_i(f)\|_2^2)^{1/2} \le \sigma_Q$, then there exists some universal constant $c_2 > 0$ such that
\begin{equation}
    \E \max_{f\in \sF}\frac{1}{\sqrt{m}}\Big|\sum_{i=1}^{m}\big[Q_i(f) - \E Q_i(f)\big]\Big| \le c_2\cdot \Big(\sigma_Q\sqrt{H} + \frac{M_Q H}{\sqrt{m}}\Big),\label{dedecker_louichi_rosenthal_eq}
\end{equation}
(see e.g. \cite{Dedeck02} equation (4.3) in Section 4.1 therein).

Note that $(W_{k,j} - W_{k,j-1})_k$ is a martingale difference sequence and $W_{k,\tau_l} - W_{k,\tau_{l-1}} = \sum_{j=\tau_{l-1}+1}^{\tau_l}(W_{k,j} - W_{k,j-1})$. Furthermore, we have
\[
    \|W_{k,j} - W_{k,j-1}\|_2 \le \|W_k - \IE[W_k|\varepsilon_{k-j+1}]\|_2 \le \|W_k\|_2
\]
and
\begin{eqnarray*}
    \|W_{k,j} - W_{k,j-1}\|_2 &=& \|\IE[W_k^{**(k-j+1)} - W_k^{**(k-j+2)}|\sG_k]\|_2\\
    &\le& \|W_k^{**(k-j+1)} - W_k^{**(k-j+2)}\|_2\\
    &=& \|W_k - W_k^{*(k-j+1)}\|_2 = \delta_2^{W_k}(j-1),
\end{eqnarray*}
thus
\begin{eqnarray*}
    \|W_{k,j} - W_{k,j-1}\|_2 \le \min\{\|W_k\|_2,\delta^{W_k}_2(j-1)\}.
\end{eqnarray*}
We conclude with the elementary inequality $\min\{a_1,b_1\} + \min\{a_2,b_2\} \le \min\{a_1 + a_2, b_1 + b_2\}$ that
\begin{eqnarray*}
    \|T_{i,l}\|_2 &=& \Big\|\sum_{k=(i-1)\tau_l+1}^{(i\tau_l)\wedge n}(W_{k,\tau_l} - W_{k,\tau_{l-1}})\Big\|_2\\
    &=& \Big\|\sum_{j=\tau_{l-1}+1}^{\tau_l} \sum_{k=(i-1)\tau_l+1}^{(i\tau_l)\wedge n}(W_{k,j} - W_{k,j-1})\Big\|_2\\
    &\le& \sum_{j=\tau_{l-1}+1}^{\tau_l} \Big\|\sum_{k=(i-1)\tau_l+1}^{(i\tau_l)\wedge n}(W_{k,j} - W_{k,j-1})\Big\|_2\\
    &\le& \sum_{j=\tau_{l-1}+1}^{\tau_l} \Big(\sum_{k=(i-1)\tau_l+1}^{(i\tau_l)\wedge n}\|W_{k,j} - W_{k,j-1}\|_2^2\Big)^{1/2} \\ 
    &\le& \sum_{j=\tau_{l-1}+1}^{\tau_l} \min \Big\{  \Big(\sum_{k=(i-1)\tau_l+1}^{(i\tau_l)\wedge n}\|W_k\|_2^2\Big)^{1/2},\Big(\sum_{k=(i-1)\tau_l+1}^{(i\tau_l)}(\delta_2^{W_k}(j-1))^2 \Big)^{1/2}\Big\}.
\end{eqnarray*}
Put
\[
    \sigma_{i,l} := \Big(\frac{1}{\tau_l}\sum_{k=(i-1)\tau_l+1}^{(i\tau_l)\wedge n} \|W_k\|_2^2\Big)^{1/2}, \quad\quad \Delta_{i,j,l} := \Big(\frac{1}{\tau_l}\sum_{k=(i-1)\tau_l+1}^{(i\tau_l)\wedge n} \delta_2^{W_k}(j-1)^2\Big)^{1/2}.
\]
Then
\begin{eqnarray}
    && \Big(\frac{1}{\frac{n}{\tau_l}} \underset{i \text{ even}}{\sum_{i=1}^{\lfloor \frac{n}{\tau_l}\rfloor + 1}}\frac{1}{\tau_l}\norm{T_{i,l}(f)}_{2}^2 \Big)^{1/2}\nonumber\\
    &\le& \Big(\frac{1}{\frac{n}{\tau_l}}  \sum_{i=1}^{\lfloor \frac{n}{\tau_l}\rfloor + 1}   \Big(\sum_{j=\tau_{l-1}+1}^{\tau_l} \min \Big\{  \Big(\frac{1}{\tau_l}\sum_{k=(i-1)\tau_l+1}^{(i\tau_l)\wedge n} \|W_k\|_2^2\Big)^{1/2},\nonumber\\
    &&\quad\quad\quad\quad\quad\quad\quad\quad\quad\quad\quad\quad\quad\quad\quad\Big(\frac{1}{\tau_l}\sum_{k=(i-1)\tau_l+1}^{(i\tau_l)\wedge n} \delta_2^{W_k}(j-1)^2\Big)^{1/2}\Big\} \Big)^2 \Big)^{1/2} \nonumber\\
    &=& \Big(\frac{1}{\frac{n}{\tau_l}}  \sum_{i=1}^{\lfloor \frac{n}{\tau_l}\rfloor + 1}  \Big( \big(\tau_l - \tau_{l-1}\big)^2 \min\{  \sigma_i^2,\Delta_{i,\tau_{l-1}+1,l}^2\} \Big)^{1/2} \nonumber\\
    &=& \Big(\frac{1}{\frac{n}{\tau_l}}  \sum_{i=1}^{\lfloor \frac{n}{\tau_l}\rfloor + 1}  \big(\tau_l - \tau_{l-1}\big)^2 \min\{\sigma_{i,l}^2,\Delta_{i,\tau_{l-1}+1,l}^2\}\Big)^{1/2} \nonumber\\
    &\le& \big(\tau_l - \tau_{l-1}\big)\cdot \Big(\min\{\frac{1}{\frac{n}{\tau_l}}\sum_{i=1}^{\lfloor \frac{n}{\tau_l}\rfloor + 1}\sigma_{i,l}^2,\frac{1}{\frac{n}{\tau_l}}\sum_{i=1}^{\lfloor \frac{n}{\tau_l}\rfloor + 1}\Delta_{i,\tau_{l-1}+1,l}^2\}\Big)^{1/2} \nonumber\\
    &\le& \sum_{j=\tau_{l-1}+1}^{\tau_l} \min\{\Big(\frac{1}{\frac{n}{\tau_l}}\sum_{i=1}^{\lfloor \frac{n}{\tau_l}\rfloor + 1}\sigma_{i,l}^2\Big)^{1/2},\Big(\frac{1}{\frac{n}{\tau_l}}\sum_{i=1}^{\lfloor \frac{n}{\tau_l}\rfloor + 1}\Delta_{i,\tau_{l-1}+1,l}^2\Big)^{1/2}\} \nonumber\\ 
    &\le& \sum_{j=\tau_{l-1}+1}^{\tau_l} \min\{\|f\|_{2,n},\Big(\frac{1}{n}\sum_{i=1}^{n}\delta_2^{W_i}(\tau_{l-1})^2\Big)^{1/2}\}\nonumber\\
    &\le& \sum_{j=\tau_{l-1}+1}^{\tau_l} \min\{\|f\|_{2,n},\ID_{n}\Delta(\lfloor\frac{j}{2}\rfloor )\}\label{proof_rosenthal_a2_variancebound}
\end{eqnarray}
With $\frac{1}{\sqrt{\tau_l}}\big|T_{i,l}(f) \big| \leq 2\sqrt{\tau_l}\|f\|_\infty \le 2\sqrt{\tau_l}M$ and \reff{dedecker_louichi_rosenthal_eq}, we obtain
\begin{eqnarray*}
    \sum_{l=1}^{L}\Big[\IE\max_{f\in \sF}\Big|\frac{1}{\sqrt{\frac{n}{\tau_l}}}\underset{i \text{ even}}{\sum_{i=1}^{\lfloor \frac{n}{\tau_l}\rfloor + 1}}\frac{1}{\sqrt{\tau_l}}T_{i,l}(f)\Big|\Big] &\le& c_2\sum_{l=1}^{L} \Big[\sup_{f}\Big(\frac{1}{\frac{n}{\tau_l}} \underset{i \text{ even}}{\sum_{i=1}^{\lfloor \frac{n}{\tau_l}\rfloor + 1}}\norm{\frac{1}{\sqrt{\tau_l}}T_{i,l}(f)}_{2}^2 \Big)^{1/2} \sqrt{H} + \frac{2\sqrt{\tau_l}MH}{\sqrt{\frac{n}{\tau_l}}}\Big],
\end{eqnarray*}
and a similar assertion for the second term ($i$ odd) in $A_2$. With \reff{proof_rosenthal_a2_variancebound}, we conclude that
\begin{eqnarray}
    \E A_2 &\le& \sum_{l=1}^{L}\Big[\E \max_{f\in \sF}\frac{1}{\sqrt{\frac{n}{\tau_l}}}\Big|\sum_{1\le i \le \lfloor  \frac{n}{\tau_l}\rfloor+1, i \text{ odd}}\frac{1}{\sqrt{\tau_l}}T_{i,l}(f)\Big|\nonumber\\
    &&\quad\quad\quad\quad\quad\quad + \E \max_{f\in \sF}\frac{1}{\sqrt{\frac{n}{\tau_l}}}\Big|\sum_{1\le i \le \lfloor \frac{n}{\tau_l}\rfloor+1, i \text{ even}}\frac{1}{\sqrt{\tau_l}}T_{i,l}(f)\Big|\Big]\nonumber\\
    &\le& 4c_2\sum_{l=1}^{L}\Big[\Big(\sum_{j=\tau_{l-1}+1}^{\tau_l} \min\{\max_{f\in\sF}\|f\|_{2,n},\ID_{n}\Delta(\lfloor\frac{j}{2}\rfloor )\}\Big)\cdot \sqrt{H} + \frac{\sqrt{\tau_l} M H}{\sqrt{\lfloor \frac{n}{\tau_l}\rfloor+1}}\Big].\label{proof_rosenthal_a2_eq1}
\end{eqnarray}
Note that
\begin{equation}
    \sum_{l=1}^{L}\frac{\sqrt{\tau_l}}{\sqrt{\lfloor \frac{n}{\tau_l}\rfloor+1}} \le \sum_{l=1}^{L}\frac{\sqrt{\tau_l}}{\sqrt{\frac{n}{\tau_l}}} = \frac{1}{\sqrt{n}}\sum_{l=0}^{L}\tau_l = \frac{1}{\sqrt{n}}\sum_{l=1}^{L-1}2^l \le \frac{1}{\sqrt{n}}(2^L + q) \le \frac{2q}{\sqrt{n}}.\label{proof_rosenthal_a2_eq2}
\end{equation}
Furthermore, we have by Lemma \ref{lemma_vfactor_negligible} that
\begin{eqnarray}
    \sum_{l=1}^{L}\sum_{j=\tau_{l-1}+1}^{\tau_l} \min\{\max_{f\in\sF}\|f\|_{2,n},\ID_{n}\Delta(\lfloor\frac{j}{2}\rfloor )\} &\le& \sum_{j = 2}^\infty \min\{\max_{f\in\sF}\|f\|_{2,n},\ID_{n}\Delta(\lfloor\frac{j}{2}\rfloor )\}\nonumber\\
    &\le& 2\bar V_n(\max_{f\in\sF}\|f\|_{2,n})\nonumber\\
    &=& 2\max_{f\in\sF}\bar V_n(\|f\|_{2,n}) = 2\max_{f\in\sF}V_n(f),\label{proof_rosenthal_a2_eq3}
\end{eqnarray}
where
\begin{equation}
    \bar V_n(x) = x +  \sum_{j=1}^{\infty}\min\{x,\ID_n \Delta(j)\}\label{definition_vbar}
\end{equation}
and the second to last equality holds since $x \mapsto \bar V_n(x)$ is increasing.

Inserting \reff{proof_rosenthal_a2_eq2} and \reff{proof_rosenthal_a2_eq3} into \reff{proof_rosenthal_a2_eq1}, we conclude that with some universal $c_3 > 0$, 
\begin{equation}
    \E A_2 \le c_3\Big( \sup_{f\in \sF}V_n(f) \sqrt{H} + \frac{q MH}{\sqrt{n}}\Big) \le c_2\Big(\sigma \sqrt{H} + \frac{q M H}{\sqrt{n}}\Big).\label{proof_rosenthal_a2_final}
\end{equation}

Since $S_{n,1}^{W} = \sum_{i=1}^{n}W_{i,1}(f)$ is a sum of independent variables with $|W_{i,1}(f)| \le \|f\|_{\infty} \le M$ and $\|W_{i,0}(f)\|_2 \le 2\|f\|_2 \le 2V_n(f) \le 2\sigma$,
we obtain from \reff{dedecker_louichi_rosenthal_eq} again
\begin{equation}
    \E A_3 \le c_2\Big(\sigma \sqrt{H} + \frac{M H}{\sqrt{n}}\Big).\label{proof_rosenthal_a3_final}
\end{equation}
If we insert the bounds \reff{proof_rosenthal_a1_final}, \reff{proof_rosenthal_a2_final} and \reff{proof_rosenthal_a3_final}  into \reff{rosenthal_main_decomposition}, we obtain the result \reff{proposition_rosenthal_bound_result1}.

We now show \reff{proposition_rosenthal_bound_result2}. If  $q^{*}(\frac{M\sqrt{H}}{\sqrt{n}\ID_n^{\infty}})\frac{H}{n} \le 1$, we have $q^{*}(\frac{M\sqrt{H}}{\sqrt{n}\ID_n^{\infty}}) \in \{1,...,n\}$ and thus by \reff{proposition_rosenthal_bound_result1}:
    \begin{eqnarray}
        \E \max_{f\in \sF}\Big|\frac{1}{\sqrt{n}}S_n(f)\Big| &\le& c\Big(\sqrt{H}\ID_n^{\infty} \beta\Big(q^{*}\Big(\frac{M\sqrt{H}}{\sqrt{n}\ID_n^{\infty}}\Big)\Big) + q^{*}\Big(\frac{M\sqrt{H}}{\sqrt{n}\ID_n^{\infty}}\Big)\frac{MH}{\sqrt{n}} + \sigma \sqrt{H}\Big)\nonumber\\
        &\le& 2c\Big(q^{*}\Big(\frac{M\sqrt{H}}{\sqrt{n}\ID_n^{\infty}}\Big)\frac{MH}{\sqrt{n}} + \sigma\sqrt{H}\Big)\nonumber\\
        &=& 2c\Big(\sqrt{n}M\cdot \min\Big\{q^{*}\Big(\frac{M\sqrt{H}}{\sqrt{n}\ID_n^{\infty}}\Big)\frac{H}{n},1\Big\} + \sigma \sqrt{H}\Big).\label{proposition_rosenthal_bound_result2_eq1}
    \end{eqnarray}
    If $q^{*}(\frac{M\sqrt{H}}{\sqrt{n}\ID_n^{\infty}})\frac{H}{n} \ge 1$, we note that the simple bound
    \begin{eqnarray}
        \E \max_{f\in \sF}\Big|\frac{1}{\sqrt{n}}S_n(f)\Big| &\le& 2\sqrt{n}M\nonumber\\
        &\le& 2c\Big( \sqrt{n}M \min\Big\{q^{*}\Big(\frac{M\sqrt{H}}{\sqrt{n}\ID_n^{\infty}}\Big)\frac{H}{n},1\Big\} + \sigma\sqrt{H}\Big)\label{proposition_rosenthal_bound_result2_eq2}
    \end{eqnarray}
    holds. Putting the two bounds \reff{proposition_rosenthal_bound_result2_eq1} and \reff{proposition_rosenthal_bound_result2_eq2} together, we obtain the result \reff{proposition_rosenthal_bound_result2}.

\end{proof}

\begin{lem}\label{lemma_vfactor_negligible}
    Let $\omega(k)$ be an increasing sequence in $k$. Then, for any $x > 0$,
    \[
        \sum_{j = 2}^\infty \min\{x,\ID_{n} \Delta(\lfloor \frac{j}{2}\rfloor)\}\omega(j) \le 2 \sum_{j = 1}^\infty \min\{x,\ID_{n} \Delta(j)\}\omega(2j+1).
    \]
    Especially in the case $\omega(k) = 1$,
    \[
        \sum_{j = 2}^\infty \min\{x,\ID_{n} \Delta(\lfloor \frac{j}{2}\rfloor)\} \le 2 \sum_{j = 1}^\infty \min\{x,\ID_{n} \Delta(j)\}.
    \]
\end{lem}
\begin{proof}[Proof of Lemma \ref{lemma_vfactor_negligible}]
    It holds that
    \begin{eqnarray*}
        &&\sum_{j = 2}^\infty \min\{x,\ID_{n} \Delta(\lfloor \frac{j}{2}\rfloor)\}\omega(j)\\
        &=& \sum_{k = 1}^\infty \min\{x,\ID_{n} \Delta(\lfloor \frac{2k}{2}\rfloor)\}\omega(2k) + \sum_{k = 1}^\infty \min\{x,\ID_{n} \Delta(\lfloor \frac{2k+1}{2}\rfloor)\}\omega(2k+1)\\
        &=& \sum_{k = 1}^\infty \min\{x,\ID_{n} \Delta(k)\}\cdot \{\omega(2k) + \omega(2k+1)\}\\
        &\le& 2\sum_{k = 1}^\infty \min\{x,\ID_{n} \Delta(k)\}\cdot \omega(2k+1).
    \end{eqnarray*}
\end{proof}

\begin{proof}[Proof of Corollary \ref{proposition_rosenthal_bound_rates}]
Let $\sigma := \sup_{n\in\N}\sup_{f\in\sF}V_n(f) < \infty$. For $Q \ge 1$, define
\[
    M_n = \frac{\sqrt{n}}{\sqrt{H}}r(\frac{\sigma Q^{1/2}}{\ID_n^{\infty}})\ID_n^{\infty}.
\]
Let $\bar F = \sup_{f\in\sF}\bar f$, and $F(z,u) = D_{n}^{\infty}(u)\cdot \bar F(z,u)$. Then $F$ is an envelope function of $\sF$. We furthermore have
\begin{equation}
    \IP(\sup_{i=1,...,n} F(Z_i,\frac{i}{n}) > M_n) \le \IP\Big( \big(\frac{1}{n}\sum_{i=1}^{n}F(Z_i,\frac{i}{n})^{\nu}\big)^{1/\nu} > \frac{M_n}{n^{1/\nu}}\Big) \le  \frac{n}{M_n^{\nu}}\cdot \|F\|_{\nu,n}^{\nu}.\label{proposition_rosenthal_bound_prob1}
\end{equation}
Inserting the bound
\[
    \|F\|_{\nu,n}^{\nu} = \frac{1}{n}\sum_{i=1}^{n}D_{n}^{\infty}(\frac{i}{n})^{\nu}\|\bar F(Z_i,\frac{i}{n})\|_{\nu}^{\nu} \le C_{\Delta}^{\nu}\cdot \frac{1}{n}\sum_{i=1}^{n}D_{n}^{\infty}(\frac{i}{n})^{\nu} \le C_{\Delta}^{\nu}\cdot (\ID_{\nu,n}^{\infty})^{\nu}
\]
into \reff{proposition_rosenthal_bound_prob1} and using $r(\gamma a) \ge \gamma r(a)$ for $\gamma \ge 1, a > 0$ (this is similarly proven as for Lemma \ref{lemma_chaining_i}), we obtain
\begin{eqnarray}
    \IP(\sup_{i=1,...,n} F(Z_i,\frac{i}{n}) > M_n) &\le& \Big(\frac{H}{n^{1-\frac{2}{\nu}} r(\frac{\sigma Q^{1/2}}{\ID_n^{\infty}})^2}\Big)^{\nu/2}\cdot \Big(\frac{C_{\Delta} \ID_{\nu,n}^{\infty}}{\ID_n^{\infty}}\Big)^{\nu}\nonumber\\
    &\le& \frac{1}{Q^{\nu/2}}\Big(\frac{H}{n^{1-\frac{2}{\nu}} r(\frac{\sigma}{\ID_n^{\infty}})^2}\Big)^{\nu/2}\cdot \Big(\frac{C_{\Delta} \ID_{\nu,n}^{\infty}}{\ID_n^{\infty}}\Big)^{\nu}.\label{proposition_rosenthal_bound_prob1final}
\end{eqnarray}

Using the rough bound $\|f\|_{\nu,n} \le \|F\|_{\nu,n}$ and $r(a) \le a$ for $a > 0$ from Lemma \ref{lemma_chaining_i}, we obtain
\begin{eqnarray}
    \max_{f\in \sF}\frac{1}{\sqrt{n}}\sum_{i=1}^{n}\IE[f(Z_i,\frac{i}{n})\Ii_{\{|f(Z_i,\frac{i}{n})| > M_n\}}] &\le& \frac{1}{\sqrt{n}M_n^{\nu-1}}\max_{f\in \sF}\sum_{i=1}^{n}\IE[ |f(Z_i,\frac{i}{n})|^{\nu}]\nonumber\\
    &\le& \frac{n}{M_n^{\nu}}\cdot \frac{M_n}{\sqrt{n}}\max_{f\in\sF}\|f\|_{\nu,n}^{\nu}\nonumber\\
    &\le& \Big(\frac{C_{\Delta}^2H}{n^{1-\frac{2}{\nu}}r(\frac{\sigma Q^{1/2}}{\ID_n^{\infty}})^2}\Big)^{\nu/2}\cdot \frac{\sigma Q^{1/2}}{\sqrt{H}}\cdot \Big(\frac{\ID_{\nu,n}^{\infty}}{\ID_n^{\infty}}\Big)^{\nu}\nonumber\\
    &\le& \frac{\sigma}{Q^{\frac{\nu-2}{2}}\sqrt{H}}\Big(\frac{C_{\Delta}^2 H}{n^{1-\frac{2}{\nu}}r(\frac{\sigma}{\ID_n^{\infty}})^2}\Big)^{\nu/2}\cdot \Big(\frac{\ID_{\nu,n}^{\infty}}{\ID_n^{\infty}}\Big)^{\nu}.\label{proposition_rosenthal_bound_prob2final}
\end{eqnarray}
Abbreviate
\[
    C_n := \Big(\frac{C_{\Delta}^2 H}{n^{1-\frac{2}{\nu}}r(\frac{\sigma}{\ID_n^{\infty}})^2}\Big)^{\nu/2}\cdot \Big(\frac{\ID_{\nu,n}^{\infty}}{\ID_n^{\infty}}\Big)^{\nu}.
\]
By assumption, $\sup_{n\in\N}C_n < \infty$. By Theorem \ref{proposition_rosenthal_bound}, \reff{proposition_rosenthal_bound_prob1final} and \reff{proposition_rosenthal_bound_prob2final},
\begin{eqnarray*}
    &&\IP\Big(\max_{f\in \sF}\big|\G_n(f)\big| > Q\sqrt{H}\Big)\\
    &\le& \IP\Big(\max_{f\in \sF}\big|\G_n(f)\big| > Q\sqrt{H}, \sup_{i=1,...,n}\bar F(Z_i,\frac{i}{n}) \le M\Big)\\
    &&\quad\quad\quad + \IP(\sup_{i=1,...,n} F(Z_i,\frac{i}{n}) > M)\\
    &\le& \IP\Big(\max_{f\in \sF}\big|\G_n(\max\{\min\{f,M\},-M\})\big| > Q\sqrt{H}/2\Big)\\
    &&\quad\quad + \IP\Big(\max_{f\in \sF}\big|\frac{1}{\sqrt{n}}\sum_{i=1}^{n}\IE[f(Z_i,\frac{i}{n})\Ii_{\{|f(Z_i,\frac{i}{n})| > M\}}] > Q\sqrt{H}/2\Big)\\
    &&\quad\quad\quad + \IP(\sup_{i=1,...,n}F(Z_i,\frac{i}{n}) > M)\\
    &\le& \frac{2c}{Q\sqrt{H}}\Big[ \sigma\sqrt{H} + q^{*}\Big(r(\frac{\sigma Q^{1/2}}{\ID_n^{\infty}})\Big)r(\frac{\sigma Q^{1/2}}{\ID_n^{\infty}})\ID_n^{\infty}\Big] + \Big(\frac{1}{Q^{\frac{\nu}{2}}} + \frac{2\sigma}{Q^{\frac{\nu}{2}}H}\Big)C_n\\
    &\le& \frac{4c\sigma}{Q^{1/2}} + \Big(\frac{1}{Q^{\frac{\nu}{2}}} + \frac{2\sigma}{Q^{\frac{\nu}{2}}H}\Big)C_n.
\end{eqnarray*}
Since $\sup_{n\in\N}C_n < \infty$ and $\sigma$ is independent of $n$, the assertion follows for $Q \to \infty$.
\end{proof}

\subsection{Proofs of Section \ref{sec_asymptotic_tightness}}
\label{sec_asymptotic_tightness_supp}

\begin{proof}[Proof of Theorem \ref{thm_equicont}]
    In the following, we abbreviate $\IH(\delta) = \IH(\delta,\sF, V_n)$ and $\N(\delta) = \N(\delta,\sF, V_n)$. Choose $\delta_0 = \sigma$ and $\delta_j = 2^{-j}\delta_0$.
    
    For each $j\in\N_0$, we choose a covering by brackets $\sF_{jk}^{pre} := [l_{jk},u_{jk}] \cap \sF$, $k = 1,...,\N(\delta_j)$ such that $V_n(u_{jk} - l_{jk}) \le \delta_j$ and $\sup_{f,g \in \sF_{jk}}|f-g| \le u_{jk} - l_{jk} =: \Delta_{jk}$. We may assume w.l.o.g. that $l_{jk}, u_{jk}, \Delta_{jk} \in \sF$.
    
    If $l_{jk},u_{jk}$ do not belong to $\sF$, we can simply define new brackets by
    \[
        \tilde l_{jk}(z,u) := \inf_{f\in [l_{jk},u_{jk}]}f(z,u), \quad\quad \tilde u_{jk}(z,u) := \sup_{f\in [l_{jk},u_{jk}]}f(z,u)
    \]
    which fulfill $[l_{jk},u_{jk}] \cap \sF = [\tilde l_{jk}, \tilde u_{jk}]\cap \sF$, and
    \[
        |\tilde l_{jk}(z,u) - \tilde l_{jk}(z',u)| \le \sup_{f\in [l_{jk},u_{jk}]}|f(z,u) - f(z',u)|.
    \]
   Thus, we can add $\tilde l_{jk}, \tilde u_{jk}$ to $\sF$ without changing the bracketing numbers $\N(\varepsilon,\sF,\|\cdot\|)$ and the validity of Assumption \ref{ass1}.
    
    We now construct inductively a new nested sequence of partitions $(\sF_{jk})_k$ of $\sF$ from $(\sF_{jk}^{pre})_k$ in the following way: For each fixed $j \in\N_0$, put
    \[
        \{\sF_{jk}:k\} := \{\bigcap_{i=0}^{j}\sF_{ik_i}^{pre}:k_i \in \{1,...,\N(\delta_i)\}, i \in \{0,...,j\}\}
    \]
    as the intersections of all previous partitions and the $j$-th partition. Then $|\{\sF_{jk}:k\}|\le N_j := \N(\delta_0)\cdot ... \cdot \N(\delta_j)$. By monotonicity of $V_n$, we have
    \[
        \sup_{f,g\in \sF_{jk}}|f-g| \le \Delta_{jk}, \quad\quad V_n(\Delta_{jk}) \le \delta_j.
    \]
    
    In each $\sF_{jk}$, fix some $f_{jk} \in \sF$, and define $\pi_{j}f := f_{j,\psi_j f}$ where $\psi_{j}f := \min\{i \in \{1,...,N_j\}: f \in \sF_{ji}\}$. Put $\Delta_{j}f := \Delta_{j,\psi_j f}$ and
    \[
        I(\sigma) := \int_0^{\sigma}\sqrt{1 \vee \IH(\varepsilon,\sF, V_n)} d \varepsilon,
    \]
    we set
    \begin{equation}
        \tau := \min\Big\{j \ge 0: \delta_j \le \frac{I(\sigma)}{\sqrt{n}}\Big\} \vee 1.\label{definition_stoppingtime_chaining}
    \end{equation}
    Put
    \[
        m_j := \frac{1}{2}m(n,\delta_j, N_{j+1}),
    \]
    ($m(\cdot)$ from Lemma \ref{lemma_chaining}). Choose $M_n = \frac{1}{2}m_0$. We then have
    \[
        \E \sup_{f\in \sF}\big|\G_n(f)\big| \le \E \sup_{f\in \sF(M_n)}\big|\G_n(f)\big| + \frac{1}{\sqrt{n}}\sum_{i=1}^{n}\E\big[W_i(F\Ii_{\{F > M_n\}})\big],
    \]
    where $\sF(M_n) := \{\varphi_{M_n}^{\wedge}(f):f\in \sF\}$. Due to Lemma \ref{lemma_properties_truncationfunction}(iii), $\sF(M_n)$ still fulfills Assumption \ref{ass1}.

    Since $|f| \le g$ implies $|W_i(f)| \le W_i(g)$ and $\|W_i(g)\|_1 \le \|g(Z_i,\frac{i}{n})\|_1$, it holds that
\begin{eqnarray*}
    |\G_n(f)| &\le& \frac{1}{\sqrt{n}}\sum_{i=1}^{n}\big|W_i(f) - \IE W_i(f)\big|\\
    &\le& \G_n(g) + \frac{2}{\sqrt{n}}\sum_{i=1}^{n}\|W_i(g)\|_1 \le \G_n(g) + 2 \sqrt{n} \|g\|_{1,n}.
\end{eqnarray*}
 By \reff{truncation_decomposition_eq2} and \reff{truncation_decomposition_eq3} and the fact that $\|f - \pi_0 f\|_{\infty} \le 2M_n \le m_0$, we have the decomposition
\begin{eqnarray}
    \sup_{f\in\sF}|\G_n(f)| &\le& \sup_{f\in\sF}|\G_n(\pi_0 f)|\nonumber\\
    &&\quad\quad + \sup_{f\in \sF}|\G_n(\varphi_{m_\tau}^{\wedge}(f-\pi_\tau f))| + \sum_{j=0}^{\tau-1}\sup_{f\in\sF}\Big|\G_n(\varphi_{m_j-m_{j+1}}^{\wedge}(\pi_{j+1}f - \pi_j f))\Big|\nonumber\\
    &&\quad\quad\quad\quad+ \sum_{j=0}^{\tau-1}\sup_{f\in\sF}|\G_n(R(j))|\nonumber\\
    &\le& \sup_{f\in\sF}|\G_n(\pi_0 f)|\nonumber\\
    &&\quad\quad + \Big\{\sup_{f\in\sF}|\G_n(\varphi_{m_\tau}^{\wedge}(\Delta_\tau f))| + 2\sqrt{n}\sup_{f\in \sF}\|\Delta_\tau f\|_{1,n}\Big\}\nonumber\\
    &&\quad\quad + \sum_{j=0}^{\tau-1}\sup_{f\in\sF}\Big|\G_n(\varphi_{m_j-m_{j+1}}^{\wedge}(\pi_{j+1}f - \pi_j f))\Big|\nonumber\\
    &&\quad\quad + \sum_{j=0}^{\tau-1}\Big\{\sup_{f\in\sF}\Big|\G_n(\min\big\{\big|\varphi_{m_{j+1}}^{\vee}(\Delta_{j+1}f)\big|,2m_j\big\})\Big|\nonumber\\
    &&\quad\quad\quad\quad\quad\quad\quad\quad\quad\quad+ 2\sqrt{n}\sup_{f\in\sF}\| \Delta_{j+1}f \Ii_{\{\Delta_{j+1}f > m_{j+1}\}}\|_{1,n}\Big\}\nonumber\\
    &&\quad\quad + \sum_{j=0}^{\tau-1}\Big\{\sup_{f\in\sF}\Big|\G_n(\min\big\{\big|\varphi_{m_j-m_{j+1}}^{\vee}(\Delta_{j}f)\big|,2m_j\big\})\Big|\nonumber\\
    &&\quad\quad\quad\quad\quad\quad\quad\quad\quad\quad + 2\sqrt{n}\sup_{f\in\sF}\| \Delta_{j}f \Ii_{\{\Delta_{j}f > m_j - m_{j+1}\}}\|_{1,n}\Big\}\nonumber\\
    &=:& R_1 + R_2 + R_3 + R_4 + R_5.\label{truncation_decomposition_final}
\end{eqnarray}

    We now discuss the terms $R_i$, $i \in \{1,...,5\}$ from \reff{truncation_decomposition_final}. Therefore, put $C_n := c(1+\frac{\ID_n^{\infty}}{\ID_n}) + \frac{\ID_n}{\ID_n^{\infty}}$.
    
    Since $\Delta_{jk} = u_{jk} - l_{jk}$ with $l_{jk},u_{jk}\in \sF$, the class $\{\frac{1}{2}\Delta_{jk}:k \in \{1,...,\N(\delta_j)\}\}$ still fulfills Assumption \ref{ass1}. We conclude by Lemma \ref{lemma_properties_truncationfunction}(iii) that for arbitrary $m,\tilde m > 0$, the classes
    \begin{eqnarray*}
       && \{\frac{1}{2}\varphi_m^{\wedge}(\Delta_{jk}):k \in \{1,...,\N(\delta_j)\}\},\\ &&\{\frac{1}{2}\min\{\varphi_m^{\vee}(\Delta_{jk}),2\tilde m\}:k \in \{1,...,\N(\delta_j)\}\},\\
       &&\{\frac{1}{2}\varphi_m^{\wedge}(\pi_{j+1}f - \pi_j f):k \in \{1,...,\N(\delta_j)\}\}
    \end{eqnarray*}
    fulfill Assumption \ref{ass1}.
    
    \begin{itemize}
        \item Since $|\{\pi_0 f: f\in \sF(M_n)\}| \le \N(\delta_0) = \N(\sigma)$, $\|\pi_0 f\|_{\infty} \le M_n \le m(n,\delta_0,\N(\delta_1))$ and $V_n(\pi_0 f) \le \sigma = \delta_0$ (by assumption, every $f\in \sF$ fulfills $V_n(f) \le \sigma$), we have by \reff{proposition_rosenthal_bound_implication2}:
    \[
        \IE R_1 = \E \sup_{f\in \sF(M_n)}|\G_n(\pi_0 f)| \le C_n\delta_0 \sqrt{1 \vee \log \N(\delta_1)}.
    \]
    \item It holds that $|\{\varphi^{\wedge}_{m_{\tau}}(\Delta_{\tau} f): f\in \sF(M_n)\}| \le N_{\tau}$. If $g := \varphi^{\wedge}_{m_{\tau}}(\Delta_{\tau} f)$, then $\|g\|_{\infty} \le m_{\tau} \le  m(n,\delta_{\tau},N_{\tau+1})$ and $V_n(g) \le V_n(\Delta_{\tau} f) \le \delta_{\tau}$. We conclude by \reff{proposition_rosenthal_bound_implication2} that:
        \begin{equation}
        \E \sup_{f\in \sF(M_n)}|\G_n(\varphi^{\wedge}_{m_{\tau}}(\Delta_{\tau} f))| \le C_n \delta_{\tau}\cdot \sqrt{1 \vee \log N_{\tau+1}}.\label{proof_chaining_eq3}
    \end{equation}
    For the second term, we have by definition of $\tau$ in \reff{definition_stoppingtime_chaining} and the Cauchy Schwarz inequality:
    \begin{equation}
        \sqrt{n}\|\Delta_{\tau} f\|_{1,n} \le \sqrt{n}\|\Delta_{\tau} f\|_{2,n} \le \sqrt{n}V_n(\Delta_{\tau} f) \le \sqrt{n}\delta_{\tau} \le I(\sigma).\label{proof_chaining_eq4}
    \end{equation}
    From \reff{proof_chaining_eq3} and \reff{proof_chaining_eq4} we obtain
    \[
        \IE R_2 \le C_n \delta_{\tau}\sqrt{1 \vee \log N_{\tau+1}} + 2\cdot I(\sigma).
    \]
    \item Since the partitions are nested, it holds that $|\{\varphi^{\wedge}_{m_j-m_{j+1}}(\pi_{j+1}f - \pi_j f): f\in \sF(M_n)\}| \le N_{j+1}$. If $g := \varphi^{\wedge}_{m_j-m_{j+1}}(\pi_{j+1}f - \pi_j f)$, we have  $\|g\|_{\infty} \le m_j - m_{j+1} \le m_j \le m(n,\delta_j,N_{j+1})$ and
    \[
        |g| \le |\pi_{j+1}f - \pi_j f| \le \Delta_j f.
    \]
    Furthermore, $V_n(g) \le V_n(\Delta_j f) \le \delta_j$. We conclude by \reff{proposition_rosenthal_bound_implication2} that:
    \[
        \IE R_3 \le \sum_{j=0}^{\tau-1}\E \sup_{f\in \sF(M_n)}|\G_n(\varphi^{\wedge}_{m_j-m_{j+1}}(\pi_{j+1}f - \pi_j f))| \le C_n \sum_{j=0}^{\tau-1}\delta_j \sqrt{1 \vee \log N_{j+1}}.
    \]
    \item It holds that $|\{\min\{\varphi^{\vee}_{m_{j+1}}(\Delta_{j+1}f),2m_j\}: f\in \sF(M_n)\}| \le N_{j+1}$. If $g := \min\{\varphi^{\vee}_{m_{j+1}}(\Delta_{j+1}f),2m_j\}$, we have  $\|g\|_{\infty} \le 2m_j = m(n,\delta_j,N_{j+1})$ and
    \[
        |g| \le \Delta_{j+1}f.
    \]
    By monotonicity of $V_n$, we have $V_n(g) \le V_n(\Delta_{j+1}f) \le \delta_{j+1} \le \delta_j$. We conclude by \reff{proposition_rosenthal_bound_implication2} that:
    \begin{equation}
        \sum_{j=0}^{\tau-1}\E \sup_{f\in \sF(M_n)}|\G_n(\min\{\varphi^{\vee}_{m_{j+1}}(\Delta_{j+1}f),2m_j\})| \le C_n \sum_{j=0}^{\tau-1}\delta_j \sqrt{1 \vee \log N_{j+1}}.\label{proof_chaining_eq5}
    \end{equation}
    Note that $V_n(\Delta_{j+1}f) \le \delta_{j+1}$ and $m_{j+1} = \frac{1}{2}m(n,\delta_{j+1},N_{j+2})$. By \reff{proposition_rosenthal_bound_implication3}, we have
    \begin{equation}
        \sqrt{n}\|\Delta_{j+1}f \Ii_{\{\Delta_{j+1}f > m_{j+1}\}}\|_1 \le 2\delta_{j+1}\sqrt{1 \vee \log N_{j+2}}.\label{proof_chaining_eq6}
    \end{equation}
    From \reff{proof_chaining_eq5} and \reff{proof_chaining_eq6} we obtain
    \[
        \IE R_4 \le (C_n+4) \sum_{j=0}^{\tau}\delta_j \sqrt{1 \vee \log N_{j+1}}.
    \]
    \item It holds that $|\{\min\{\varphi^{\vee}_{m_j - m_{j+1}}(\Delta_{j}f),2m_j\}: f\in \sF(M_n)\}| \le N_{j+1}$. If $g := \min\{\varphi^{\vee}_{m_j - m_{j+1}}(\Delta_{j}f),2m_j\}$, we have  $\|g\|_{\infty} \le 2m_j = m(n,\delta_j,N_{j+1})$ and
    \[
        |g| \le \Delta_{j}f.
    \]
    Thus, $V_n(g) \le V_n(\Delta_{j} f) \le \delta_{j}$. We conclude by \reff{proposition_rosenthal_bound_implication2} that:
    \begin{equation}
        \sum_{j=0}^{\tau-1}\E \sup_{f\in \sF(M_n)}|\G_n(\min\{\varphi^{\vee}_{m_j - m_{j+1}}(\Delta_{j+1}f),2m_j\})| \le C_n \sum_{j=0}^{\tau-1}\delta_j \cdot \sqrt{1 \vee \log N_{j+1}}.\label{proof_chaining_eq7}
    \end{equation}
    Note that $V_n(\Delta_{j}f) \le \delta_{j}$ and
    \begin{eqnarray*}
        2(m_{j} - m_{j+1}) &=& m(n,\delta_{j},N_{j+1}) - m(n,\delta_{j+1},N_{j+2})\\
        &=& \ID_n^{\infty}n^{1/2}\Big[\frac{r(\frac{\delta_j}{\ID_n})}{\sqrt{1 \vee \log N_{j+1}}} - \frac{r(\frac{\delta_{j+1}}{\ID_n})}{\sqrt{1 \vee \log N_{j+2}}}\Big]\\
        &\ge& \frac{\ID_n^{\infty}n^{1/2}}{\sqrt{1 \vee \log N_{j+1}}}\big[r(\frac{\delta_j}{\ID_n}) - r(\frac{\delta_{j+1}}{\ID_n})\big]\\
        &\ge& \frac{1}{2}\frac{\ID_n^{\infty}n^{1/2}}{\sqrt{1 \vee \log N_{j+1}}} r(\frac{\delta_j}{\ID_n}) = m_j,
    \end{eqnarray*}
    where the last inequality is due to Lemma \ref{lemma_chaining_i}.
    By \reff{proposition_rosenthal_bound_implication3} we have
    \begin{equation}
        \sqrt{n}\|\Delta_{j}f \Ii_{\{\Delta_{j}f > m_{j} - m_{j+1}\}}\|_{1,n} \le \sqrt{n}\|\Delta_{j}f \Ii_{\{\Delta_{j}f > \frac{m_{j}}{2}\}}\|_{1,n} \overset{m_j = \frac{1}{2}m(n,\delta_j,N_{j+1})}{\le} 4\delta_{j}\sqrt{1 \vee \log N_{j+1}}.\label{proof_chaining_eq8}
    \end{equation}
    From \reff{proof_chaining_eq7} and \reff{proof_chaining_eq8} we obtain
    \[
        R_5 \le (C_n+8) \sum_{j=0}^{\tau-1}\delta_j \sqrt{1 \vee \log N_{j+1}}.
    \]
    \end{itemize}
    Summarizing the bounds for $R_i$, $i = 1,...,5$, we obtain that with some universal constant $\tilde c > 0$, 
    \begin{equation}
        \E \sup_{f\in \sF(M_n)}\Big|\G_n(f)\Big| \le \tilde c\cdot C_n\Big[\sum_{j=0}^{\tau}\delta_j \sqrt{1 \vee \log N_{j+1}} + I(\sigma)\Big].\label{proof_chaining_eq9}
    \end{equation}
    We have $(1 \vee \log N_{j})^{1/2} = \Big(1 \vee \sum_{i=0}^{j}\log \N(\delta_i)\Big)^{1/2} \le \Big(\sum_{i=0}^{j}(1 \vee \IH(\delta_i))\Big) \le \sum_{i=0}^{j}(1 \vee \IH(\delta_i))^{1/2}$, thus 
    \begin{eqnarray}
        \sum_{j=0}^{\tau}\delta_j \sqrt{1 \vee \log N_{j+1}} &\le& \sum_{j=0}^{\infty}\delta_j \sum_{i=0}^{j}\sqrt{1 \vee \IH(\delta_{i+1})} \le \sum_{i=0}^{\infty} \Big(\sum_{j=i}^{\infty}\delta_j\Big) \sqrt{1 \vee \IH(\delta_{i+1})}\nonumber\\
        &=& 2\sum_{i=0}^{\infty}\delta_i \sqrt{1 \vee \IH(\delta_{i+1})} \le 4\sum_{i=0}^{\infty}\delta_{i+1} \sqrt{1 \vee \IH(\delta_{i+1})}.\label{proof_chaining_eq10}
    \end{eqnarray}
    
    Since $H$ is increasing, we obtain
    \begin{eqnarray}
        \sum_{i=0}^{\infty}\delta_{i+1} \sqrt{1 \vee \IH(\delta_{i+1})} &\le& \sum_{i=0}^{\infty}\delta_{i} \sqrt{1 \vee \IH(\delta_{i})} = 2\sum_{i=0}^{\infty}\delta_{i+1} \sqrt{1 \vee \IH(\delta_{i})}\nonumber\\
        &=& 2\sum_{i=0}^{\infty}\int_{\delta_{i+1}}^{\delta_i}\sqrt{1 \vee \IH(\delta_{i})} d\varepsilon\nonumber\\
        &\le& 2\sum_{i=0}^{\infty}\int_{\delta_{i+1}}^{\delta_i}\sqrt{1 \vee \IH(\varepsilon)} d\varepsilon = 2 \int_0^{\sigma}\sqrt{1 \vee \IH(\varepsilon)} d\varepsilon = 2\cdot I(\sigma).\label{proof_chaining_eq11}
    \end{eqnarray}
    Inserting \reff{proof_chaining_eq11} into \reff{proof_chaining_eq10} and then into \reff{proof_chaining_eq9}, we obtain the result.
\end{proof}

\begin{proof}[Proof of Corollary \ref{cor_equicont}]
    Define $\tilde \sF := \{f-g: f,g\in \sF\}$. It is easily seen that $\N(\varepsilon,\tilde \sF, V_n) \le \N(\frac{\varepsilon}{2},\sF, V_n)^2$ (cf. \cite{Vaart98}, Theorem 19.5), thus
    \begin{equation}
        \IH(\varepsilon,\tilde \sF, V_n) \le 2 \IH(\frac{\varepsilon}{2},\sF, V_n)\label{cor_equicont_eq1}
    \end{equation}
    
    Let $\sigma > 0$. Define
    \[
        F(z,u) := 2 D_{n}^{\infty}(u)\cdot \bar F(z,u), \quad\quad \bar F(z,u) := \sup_{f\in\sF}|\bar f(z,u)|.
    \]
    Then obviously, $F$ is an envelope function of $\tilde\sF$. 
    
    
    By Markov's inequality, Theorem \ref{thm_equicont} and \reff{cor_equicont_eq1},
    \begin{eqnarray*}
        &&\IP\Big( \sup_{V_n(f-g) \leq \sigma, \, f,g \in \sF} |\G_n(f) - \G_n(g)| \geq \eta \Big)\\
        &\le& \frac{1}{\eta}\E \sup_{V_n(f-g) \leq \sigma, \, f,g \in \sF} |\G_n(f) - \G_n(g)|\\
        &=& \frac{1}{\eta} \E \sup_{\tilde f\in \tilde \sF, V_n(\tilde f) \le \sigma}|\G_n(\tilde f)|\\
        &\le& \frac{\tilde c}{\eta}\Big[(1+\frac{\ID_n^{\infty}}{\ID_n} + \frac{\ID_n}{\ID_n^{\infty}})\int_0^{\sigma}\sqrt{1 \vee \IH(\varepsilon,\tilde \sF,V_n)} d \varepsilon + \sqrt{n}\big\|F \Ii_{\{F > \frac{1}{4}m(n,\sigma,\N(\frac{\sigma}{2}))\}}\big\|_1\Big]\\
        &\le& \frac{\tilde c }{\eta}\Big[2\sqrt{2}(1+\frac{\ID_n^{\infty}}{\ID_n} + \frac{\ID_n}{\ID_n^{\infty}})\int_0^{\sigma/2}\sqrt{1 \vee \IH(u,\sF,V_n)} d u + \frac{4\sqrt{1\vee \IH(\frac{\sigma}{2})}}{r(\frac{\sigma}{\ID_n})}\big\|F^2 \Ii_{\{F > \frac{1}{4}n^{1/2}\frac{r(\sigma)}{\sqrt{1\vee \IH(\frac{\sigma}{2})}}\}}\big\|_{1,n}\Big].
    \end{eqnarray*}
    The first term converges to 0 by \reff{cor_equicont_cond0} and \reff{cor_equicont_cond1} for $\sigma \to 0$ (uniformly in $n$).
    
    We now discuss the second term. The continuity conditions from Assumption \ref{ass1} and Assumption \ref{ass_clt_fcont} transfer to $\bar F$ by the inequality
    \[
        |\bar F(z_1,u_1) - \bar F(z_2,u_2)| = |\sup_{f\in\sF}\bar f(z_1,u_1) - \sup_{f\in\sF}\bar f(z_2,u_2)| \le \sup_{f\in\sF}|f(z_1,u_1) - f(z_2,u_2)|
    \]
    We therefore have as in Lemma \ref{lemma_theorem_clt_mult}(ii) that for all $u,u_1,u_2,v_1,v_2 \in [0,1]$,
    \begin{eqnarray}
        \|\bar F(Z_i,u) - \bar F(\tilde Z_i(\frac{i}{n}),u)\|_2 \le C_{cont}\cdot n^{-\alpha s},\label{cor_equicont_eq4}\\
        \|\bar F(Z_i(v_1),u_1) - \bar F(\tilde Z_i(v_2),v_2)\|_2 \le C_{cont}\cdot\big(|v_1 - v_2|^{\alpha s} + |u_1 - u_2|^{\alpha s}\big).\label{cor_equicont_eq5}
    \end{eqnarray}
    Put $c_n = \frac{1}{8}\frac{n^{1/2}}{\sup_{i=1,...,n}D_n^{\infty}(\frac{i}{n})}\frac{r(\sigma)}{\sqrt{1\vee \IH(\frac{\sigma}{2})}}$. Then by Lemma \ref{lemma_lindeberg_analytic}(ii) and \reff{cor_equicont_eq4},
    \begin{eqnarray}
        &&\|F^2 \Ii_{\{F > \frac{1}{4}n^{1/2}\frac{r(\sigma)}{\sqrt{1\vee \IH(\frac{\sigma}{2})}}\}}\|_{1,n}\nonumber\\
        &\le& \frac{4}{n}\sum_{i=1}^{n}D_{n}^{\infty}(\frac{i}{n})^2\cdot \IE\Big[\bar F(Z_i,\frac{i}{n})^2\Ii_{\{|\bar F(Z_i,\frac{i}{n})| > c_n\}}\Big]\nonumber\\
        &\le& \frac{16}{n}\sum_{i=1}^{n}D_{n}^{\infty}(\frac{i}{n})^2\cdot \IE\Big[\bar F(\tilde Z_i(\frac{i}{n}),\frac{i}{n})^2\Ii_{\{|\bar F(\tilde Z_i(\frac{i}{n}),\frac{i}{n})| > c_n\}}\Big]\nonumber\\
        &&\quad\quad\quad\quad+ 16C_{cont}\cdot n^{-\alpha s}\cdot (\ID_n^{\infty})^2.\label{cor_equicont_eq6}
    \end{eqnarray}
    Put $\tilde W_i(u) := \bar F(\tilde Z_i(u),u)$ and $a_n(u) := (D_{n}^{\infty}(u))^2$. By \reff{cor_equicont_eq5}, $\|\tilde W_i(u_1) - \tilde W_i(u_2)\|_2 \le 2C_{cont}|u_1 - u_2|^{\alpha s}$. By the assumptions on $D_{f,n}(\cdot)$, $c_n \to \infty$ and $\limsup_{n\to\infty}\frac{1}{n}\sum_{i=1}^{n}|a_n(\frac{i}{n})| = \limsup_{n\to\infty}(\ID_n^{\infty})^2 < \infty$. We conclude with Lemma \ref{lemma_lindeberg}(i) that
    \[
        \frac{16}{n}\sum_{i=1}^{n}D_{n}^{\infty}(\frac{i}{n})^2\cdot \IE\Big[\bar F(\tilde Z_i(\frac{i}{n}),\frac{i}{n})^2\Ii_{\{|\bar F(\tilde Z_i(\frac{i}{n}),\frac{i}{n})| > c_n\}}\Big] \to 0,
    \]
    that is, the first summand in \reff{cor_equicont_eq6} tends to $0$. Since $\limsup_{n\to\infty}\ID_n^{\infty} < \infty$, we obtain that \reff{cor_equicont_eq6} tends to $0$.
\end{proof}

\subsection{Proofs of Section \ref{sec_maximal_chaining}}
\label{sec_maximal_chaining_supp}

\begin{proof}[Proof of Lemma \ref{lemma_properties_truncationfunction}]
    \begin{enumerate}
        \item[(i)] Since $|x_1| + |x_2| \le m$ implies $|x_1|,|x_2| \le m$, we have
        \[
            I := \big|\varphi_m^{\wedge}(x_1+x_2+x_3) - \varphi_m^{\wedge}(x_1) - \varphi_m^{\wedge}(x_2)\big| = \big|\varphi_m^{\wedge}(x_1+x_2+x_3) - x_1 - x_2|.
        \]
        Case 1: $x_1+x_2+x_3 > m$. Then, since $|x_1|+|x_2| \le m$, we have $I = |m-x_1-x_2| = m-x_1-x_2 < x_3 \le |x_3|$.\\
        Case 2: $x_1+x_2+x_3 \in [-m,m]$. Then $I = |x_1+x_2+x_3-x_1-x_2| = |x_3|$.\\
        Case 3: $x_1+x_2+x_3 < -m$. Then, since $|x_1|+|x_2| \le m$, we have $I = |-m-x_1-x_2| = m+x_1+x_2 < -x_3 \le |x_3|$.\\
        Furthermore, $I \le |\varphi_m(x_1+x_2+x_3)| + |x_1+x_2| \le m+m = 2m$.
        \item[(ii)] The first assertion is obvious. If $|x| \le y$, we have
        \begin{eqnarray*}
            |\varphi_m^{\vee}(x)| &=& \begin{cases}
        x-m, & x > m\\
        0, & x \in [-m,m]\\
        -x-m, & x < -m
    \end{cases} = \begin{cases}
        |x|-m, & x > m\\
        0, & x \in [-m,m]\\
        |x|-m, & x < -m
    \end{cases} = (|x|-m)\Ii_{|x| > m}\\
    &\le& (y-m)\Ii_{y > m} = (y-m)\vee 0 = (y-m)\Ii_{\{y-m > 0\}} \le y \Ii_{y > m},
        \end{eqnarray*}
        which shows the second assertion.
        \item[(iii)] We will show that for all $z,z'\in\R^{\N}$ it holds that
        \begin{equation}
            |\varphi_m^{\wedge}(f)(z) - \varphi_m^{\wedge}(f)(z')| \le |f(z)-f(z')|, \quad\quad |\varphi_m^{\vee}(f)(z) - \varphi_m^{\vee}(f)(z')| \le |f(z)-f(z')|\label{lipschitz_truncation}
        \end{equation}
        from which the assertion follows. For real numbers $a_i,b_i$, we have
        \[
            \max_i\{a_i\}= \max_i\{a_i - b_i + b_i\} \le \max_{i}\{a_i-b_i\} + \max_i\{b_i\},
        \]
        thus $|\max_{i}\{a_i\} - \max_i\{b_i\}| \le \max_i|a_i - b_i|$. This implies $|\max\{a,y\} - \max\{a,y'\}| \le |y-y'|$ and therefore
        \begin{eqnarray*}
            |\varphi_m^{\wedge}(f)(z) - \varphi_m^{\wedge}(f)(z')| &=& |(-m)\vee (f(z)\wedge m) - (-m)\vee (f(z')\wedge m)| \le |f(z)\wedge m - f(z') \wedge m|\\
            &=& |(-f(z'))\vee (-m) - (-f(z))\vee (-m)| \le |f(z) - f(z')|.
        \end{eqnarray*}
        For the second inequality in \reff{lipschitz_truncation}, note that
        \[
            \varphi_m^{\vee}(f)(z) = (f(z) - m) \vee 0 + (f(z) + m) \wedge 0.
        \]
        We therefore have
        \[
            |\varphi_m^{\vee}(f)(z) - \varphi_m^{\vee}(f)(z')| = \big|(f(z) - m) \vee 0 - (f(z')-m)\vee 0 + (f(z) + m) \wedge 0 - (f(z')+m)\wedge 0|.
        \]
        If $f(z),f(z') \ge m$, then
        \[
            |\varphi_m^{\vee}(f)(z) - \varphi_m^{\vee}(f)(z')| \le \big|(f(z) - m) \vee 0 - (f(z')-m)\vee 0| \le |f(z) - f(z')|.
        \]
        A similar result is obtained for $f(z),f(z') \le -m$. If $f(z) \ge m$, $f(z') < m$, then
        \begin{eqnarray*}
            && |\varphi_m^{\vee}(f)(z) - \varphi_m^{\vee}(f)(z')|\\
            &\le& \big|(f(z) - m) - (f(z') + m) \wedge 0|\\
            &=& \begin{cases}
                |f(z) - f(z') - 2m| = f(z) - f(z') - 2m \le f(z) - f(z'), & f(z') \le -m,\\
                |f(z) - m| = f(z) - m \le f(z) - f(z'), & f(z') > -m
            \end{cases}.
        \end{eqnarray*}
        A similar result is obtained for $f(z) \ge m$, $f(z') \le m$, which proves \reff{lipschitz_truncation}.
    \end{enumerate}
\end{proof}

\begin{lem}[Properties of $r(\cdot)$]\label{lemma_chaining_i}
    $r(\cdot)$ is well-defined and for each $a > 0$, $\frac{r(a)}{2} \ge r(\frac{a}{2})$ and $r(a) \le a$.
\end{lem}
\begin{proof}
    $q^{*}(\cdot)$ and $r(\cdot)$ are well-defined since $\beta_{norm}(\cdot)$ is decreasing (at a rate $\ll q^{-1}$) and $r \mapsto q^{*}(r)r$ is increasing (at a rate $\ll r$) and $\lim_{r \downarrow 0}q^{*}(r)r = 0$.
    
    Let $a > 0$. We show that $r = 2r(\frac{a}{2})$ fulfills $q^{*}(r)r \le a$. By definition of $r(a)$, we obtain $r(a) \ge r = 2r(\frac{a}{2})$ which gives the result. Since $\beta_{norm}$ is decreasing, $q^{*}$ is decreasing. We conclude that
    \[
        q^{*}(r)r = 2\cdot q^{*}(2r(\frac{a}{2})) r(\frac{a}{2}) \le 2 \cdot q^{*}(r(\frac{a}{2})) r(\frac{a}{2}) \le 2\cdot \frac{a}{2} = a.
    \]
    
    The second inequality $r(a) \le a$ follows from the fact that $q^{*}(r)r$ is increasing and $q^{*}(a)a \ge a$.
\end{proof}

\subsection{Proofs of Section \ref{sec_clt}}
\label{sec_clt_supp}

\begin{proof}[Proof of Theorem \ref{theorem_clt_mult}]
    Denote $W_{i}(f) := f(Z_i,\frac{i}{n})$ and $\IW_i := (W_i(f_1),...,W_i(f_m))'$. Let $a = (a_1,...,a_m)'\in\R^{m} \backslash \{0\}$. We use the decomposition
    \begin{eqnarray*}
        \frac{1}{\sqrt{n}}\sum_{i=1}^{n}a'(\IW_i - \E \IW_i) = \sum_{j=0}^{\infty}\frac{1}{\sqrt{n}}\sum_{i=1}^{n}a'P_{i-j}\IW_i.
    \end{eqnarray*}
    For fixed $J\in \N \cup \{\infty\}$, put
    \[
        (S_n(J))_{k=1,...,m} := S_n(J) := \sum_{j=0}^{J-1}\frac{1}{\sqrt{n}}\sum_{i=1}^{n}P_{i-j}\IW_i.
    \]
    Then, since $P_{i-j}W_i(f_k)$, $i = 1,...,n$ is a martingale difference sequence and by Lemma \ref{lemma_theorem_clt_mult}(i), 
    \begin{eqnarray*}
        \|S_n(\infty)_k - S_n(J)_k\|_2 &\le& \sum_{j=J}^{\infty}\big\|\frac{1}{\sqrt{n}}\sum_{i=1}^{n}P_{i-j}W_i(f_k)\big\|_2 = \sum_{j=J}^{\infty}\Big(\frac{1}{n}\sum_{i=1}^{n}\|P_{i-j}W_i(f_k)\|_2^2\Big)^{1/2}\\
        &\le& \Big(\frac{1}{n}\sum_{i=1}^{n}D_{f_k,2,n}(\frac{i}{n})^2\Big)^{1/2}\cdot \sum_{j=J}^{\infty}\Delta(j),
    \end{eqnarray*}
    thus
    \begin{equation}
        \limsup_{J,n\to\infty}\|S_n(\infty)_k - S_n(J)_k\|_2 \le \sup_{n\in\N}\Big(\frac{1}{n}\sum_{i=1}^{n}D_{f_k,2,n}(\frac{i}{n})^2\Big)^{1/2}\cdot \limsup_{J\to\infty}\sum_{j=J}^{\infty}\Delta(j) = 0.\label{theorem_clt_mult_step1}
    \end{equation}
    
    Define
    \[
        (S_n^{\circ}(J)_k)_{k=1,...,m} := S_n^{\circ}(J) := \frac{1}{\sqrt{n}}\sum_{i=1}^{n-J+1}\sum_{j=0}^{J-1}P_{i}\IW_{i+j}.
    \]
    Then we have
    \begin{eqnarray*}
        \|S_n^{\circ}(J)_k - S_n(J)_k\|_2 &\le& \sum_{j=0}^{J-1}\|\frac{1}{\sqrt{n}}\sum_{i=1}^{j}P_{i-j}W_i(f_k)\|_2 + \frac{1}{\sqrt{n}}\sum_{j=0}^{J-1}\|\sum_{i=n-J+j+1}^{n}P_{i-j}W_i(f_k)\|_2\\
        &\le& \frac{2J^2}{\sqrt{n}}\cdot \sup_{i=1,...,n+j}\|P_{i-j}W_i(f_k)\|_2\\
        &\le& \frac{2J^2}{\sqrt{n}}\cdot \sup_{i=1,...,n+j}\|f_k(Z_i,\frac{i}{n})\|_2.
    \end{eqnarray*}
    By Lemma \ref{lemma_theorem_clt_mult}(i),
    \[
        \sup_{i=1,...,n+j}\|f_k(Z_i,\frac{i}{n})\|_2 \le C_{\Delta,2}\cdot D_{2,n}(\frac{i}{n}),
    \]
    which gives
    \begin{equation}
        \lim_{n\to\infty}\|S_n^{\circ}(J)_k - S_n(J)_k\|_2 = 0.\label{theorem_clt_mult_step2}
    \end{equation}
    
    \emph{Stationary approximation: } Put $\tilde S_n^{\circ}(J) = (\tilde S_n^{\circ}(J)_k)_{k=1,...,m}$, where
    \[
        \tilde S_n^{\circ}(J)_k := \frac{1}{\sqrt{n}}\sum_{i=1}^{n-J+1}\sum_{j=0}^{J-1} P_i f_k(\tilde Z_{i+j}(\frac{i}{n}), \frac{i}{n}).
    \]
    Then we have
    \begin{eqnarray*}
        &&\|S_n^{\circ}(J)_k - \tilde S_n^{\circ}(J)_k\|_2\\
        &\le& \sum_{j=0}^{J-1} \Big(\frac{1}{n}\sum_{i=1}^{n-J+1}\Big\|P_i f_k(Z_{i+j},\frac{i+j}{n}) - P_i f_k(\tilde Z_{i+j}(\frac{i}{n}), \frac{i}{n})\Big\|_2^2\Big)^{1/2}.
    \end{eqnarray*}
    For each $j,k$, it holds that
    \begin{eqnarray*}
        && \frac{1}{n}\sum_{i=1}^{n-J+1}\|P_i f_k(Z_{i+j},\frac{i+j}{n}) - P_i f_k(\tilde Z_{i+j}(\frac{i}{n}), \frac{i}{n})\|_2^2\\
        &\le& \frac{2}{n}\sum_{i=1}^{n-J+1}\Big(D_{f_k,n}(\frac{i+j}{n}) - D_{f_k,n}(\frac{i}{n})\Big)^2\cdot \sup_{i}\|\bar f(Z_{i+j},\frac{i+j}{n})\|_{2}^2\\
        &&\quad\quad + \frac{2}{n}\sum_{i=1}^{n-J+1}D_{f,n}(\frac{i}{n})^2\cdot \sup_{i}\Big\|\bar f_k(Z_{i+j},\frac{i+j}{n}) - \bar f_k(\tilde Z_{i+j}(\frac{i}{n}),\frac{i}{n})]\|_2^2.
    \end{eqnarray*}
    By Lemma \ref{lemma_theorem_clt_mult}, we have $\sup_{i}\|\bar f(Z_{i+j},\frac{i+j}{n})\|_2^2 < \infty$. Since $\frac{1}{\sqrt{n}}D_{f_k,n}(\cdot)$ has bounded variation uniformly in $n$,
    \[
        \frac{1}{n}\sum_{i=1}^{n-J+1}\Big(D_{f_k,n}(\frac{i+j}{n}) - D_{f_k,n}(\frac{i}{n})\Big)^2 \le \sup_{i=1,...,n}\frac{1}{\sqrt{n}}D_{f_k,n}(\frac{i}{n}) \cdot \frac{1}{\sqrt{n}}\sum_{i=1}^{n-J+1}\Big|D_{f_k,n}(\frac{i+j}{n}) - D_{f_k,n}(\frac{i}{n})\Big| \to 0.
    \]
    By Lemma \ref{lemma_theorem_clt_mult}(ii),
    \begin{eqnarray*}
        \sup_i\Big\|\bar f_k(Z_{i+j},\frac{i+j}{n}) - \bar f_k(\tilde Z_{i+j}(\frac{i}{n}),\frac{i}{n})\Big\|_2 \to 0.
    \end{eqnarray*}
    We therefore obtain
    \begin{equation}
        \|S_n^{\circ}(J)_k - \tilde S_n^{\circ}(J)_k\|_2 \to 0.\label{theorem_clt_mult_step3}
    \end{equation}
    
    Note that
    \[
        M_{i,k} := \frac{1}{\sqrt{n}}\sum_{j=0}^{J} P_i f_k(\tilde Z_{i+j}(\frac{i}{n}), \frac{i}{n}), \quad i = 1,...,n
    \]
    is a martingale difference sequence with respect to $\sG_{i-1}$, and
    \[
        \tilde S_n^{\circ}(J)_k = \sum_{i=1}^{n-J+1}M_{i,k}.
    \]
    We can therefore apply a central limit theorem for martingale difference sequences to $a'\tilde S_n^{\circ}(J) = \sum_{i=1}^{n-J+1}(\sum_{k=1}^{m}a_k M_{i,k})$.
    
    \emph{The Lindeberg condition:} Let $\varsigma > 0$. Iterated application of Lemma \ref{lemma_lindeberg_analytic}(i) yields that there are constants $c_1,c_2 > 0$ only depending on $m,J$ such that
    \begin{eqnarray*}
        && \sum_{i=1}^{n-J+1}\IE[(\sum_{k=1}^{m}a_k M_{i,k})^2 \Ii_{\{|\sum_{k=1}^{m}a_k M_{i,k}| > \varsigma \sqrt{n}\}}]\\
        &\le& c_1\sum_{l=0,1}\sum_{j=0}^{J-1}\sum_{k=1}^{m}|a_k|^2\cdot \frac{1}{n}\sum_{i=1}^{n-J}\IE\Big[\IE[f_k(\tilde Z_{i+j}(\frac{i}{n}),\frac{i}{n})|\sG_{i-l}]^2 \Ii_{\{|\IE[f_k(\tilde Z_{i+j}(\frac{i}{n}),\frac{i}{n})|\sG_{i-l}]| > \sqrt{n}\frac{\varsigma}{c_2|a|_{\infty}}\}}\Big].
    \end{eqnarray*}
    
    For each $l,j,k$, we have
    \begin{eqnarray}
        && \frac{1}{n}\sum_{i=1}^{n-J}\IE\Big[\IE[f_k(\tilde Z_{i+j}(\frac{i}{n}),\frac{i}{n})|\sG_{i-l}]^2 \Ii_{\{|\IE[f_k(\tilde Z_{i+j}(\frac{i}{n}),\frac{i}{n})|\sG_{i-l}]| > \sqrt{n}\frac{\varsigma}{c_2|a|_{\infty}}\}}\Big]\nonumber\\
        &=& \frac{1}{n}\sum_{i=1}^{n-J}D_{f_k,n}(\frac{i}{n})^2\IE\Big[\IE[\bar f_k(\tilde Z_i(\frac{i}{n}),\frac{i}{n})|\sG_{i-l}]^2 \Ii_{\{|\IE[\bar f_k(\tilde Z_i(\frac{i}{n}),\frac{i}{n})|\sG_{i-l}]| > \frac{\sqrt{n}}{\sup_{i=1,...,n}|D_{f,n}(\frac{i}{n})|}\frac{\varsigma}{c_2|a|_{\infty}}\}}\Big]\nonumber\\
        &=& \frac{1}{n}\sum_{i=1}^{n-J}D_{f_k,n}(\frac{i}{n})^2\IE\Big[\tilde W_i(\frac{i}{n})^2 \Ii_{\{|\tilde W_i(\frac{i}{n})| > c_n\}}\Big],\label{theorem_clt_mult_step3_eq2}
    \end{eqnarray}
    where we have put
    \[
        \tilde W_i(u) := \IE[\bar f_k(\tilde Z_i(u),u)|\sG_{i-l}], \quad\quad c_n := \frac{\sqrt{n}}{\sup_{i=1,...,n}|D_{f,n}(\frac{i}{n})|}\frac{\varsigma}{c_2|a|_{\infty}}.
    \]
    By Lemma \ref{lemma_theorem_clt_mult}(ii), $\tilde W_i(u)$ satisfies the assumptions \reff{lemma_ergodic_ass} of Lemma \ref{lemma_lindeberg}. By assumption, $c_n \to \infty$. With $a_n(u) := D_{f_k,n}(u)^2$, we obtain from Lemma \ref{lemma_lindeberg} that \reff{theorem_clt_mult_step3_eq2} converges to $0$, which shows that the Lindeberg condition is satisfied.
    
    \emph{Convergence of the variance:} We have
    \begin{eqnarray*}
        &&\sum_{i=1}^{n-J+1}\IE[(\sum_{k=1}^{m}M_{i,k})^2|\sG_{i-1}]\\
        &=& \sum_{j_1,j_2=0}^{J-1}\sum_{k_1,k_2=1}^{m}a_k a_l \cdot \frac{1}{n}\sum_{i=1}^{n-J+1}D_{f_k,n}(\frac{i}{n})D_{f_l,n}(\frac{i}{n})\cdot \IE\big[P_{i}\bar f_k(\tilde Z_{i+j_1}(\frac{i}{n}),\frac{i}{n})\cdot P_i \bar f_l(\tilde Z_{i+j_2}(\frac{i}{n}),\frac{i}{n})|\sG_{i-1}\big].
    \end{eqnarray*}
    For each $j_1,j_2,k_1,k_2$, we define
    \[
        \tilde W_i(u) := \IE\big[P_{i}\bar f_k(\tilde Z_{i+j_1}(u),u)\cdot P_i \bar f_l(\tilde Z_{i+j_2}(u),u)|\sG_{i-1}\big], \quad\quad a_n(u) := D_{f_k,n}(u)D_{f_l,n}(u).
    \]
    Then
    \begin{eqnarray*}
        && \frac{1}{n}\sum_{i=1}^{n-J+1}D_{f_k,n}(\frac{i}{n})D_{f_l,n}(\frac{i}{n})\cdot \IE\big[P_{i}\bar f_k(\tilde Z_{i+j_1}(\frac{i}{n}),\frac{i}{n})\cdot P_i \bar f_l(\tilde Z_{i+j_2}(\frac{i}{n}),\frac{i}{n})|\sG_{i-1}\big]\\
        &=& \frac{1}{n}\sum_{i=1}^{n-J+1}a_n(\frac{i}{n}) \tilde W_i(\frac{i}{n}).
    \end{eqnarray*}
    By Lemma \ref{lemma_theorem_clt_mult}(i),(ii), we have
    \begin{eqnarray*}
        \|\tilde W_0(u) - \tilde W_0(v)\|_1 &\le& \|\bar f_k(\tilde Z_0(u),u) - \bar f_k(\tilde Z_0(v),v)\|_2\cdot \|\bar f_l(\tilde Z_0(u))\|_2\\
        &&\quad\quad + \|\bar f_l(\tilde Z_0(u),u) - \bar f_l(\tilde Z_0(v),v)\|_2\cdot \|\bar f_k(\tilde Z_0(v))\|_2\\
        &\le& 2C_{cont}C_{\bar f}\cdot |u-v|^{\varsigma s/2}
    \end{eqnarray*}
    Let $A_n := \sup_{i=1,...,n}|a_n(\frac{i}{n})|$. Since $\frac{D_{f,n}(\cdot)}{D_{f,n}^{\infty}}$ has bounded variation uniformly in $n$, it follows that $\frac{a_n(\cdot)}{A_n}$ has bounded variation uniformly in $n$. From $\frac{D_{f,n}^{\infty}}{\sqrt{n}} \to 0$ we conclude $\frac{A_n}{n} \to 0$.
    
    By assumption and the Cauchy-Schwarz inequality,
    \[
        \sup_n\Big[\frac{1}{n}\sum_{i=1}^{n}|a_n(\frac{i}{n})|\Big] \le \sup_{n}\Big(\frac{1}{n}\sum_{i=1}^{n}D_{f_k,n}(\frac{i}{n})^2\Big)^{1/2}\cdot \Big(\frac{1}{n}\sum_{i=1}^{n}D_{f_l,n}(\frac{i}{n})^2\Big)^{1/2} < \infty.
    \]
    It holds that $\sup_{n}(h_n \cdot A_n) \le  \sup_{n}(h_n^{1/2}D_{f_k,n}^{\infty}) \cdot \sup_n (h_n^{1/2}D_{f_l,n}^{\infty}) < \infty$, and
    \[
        |v - u|> h_n \quad\Rightarrow\quad D_{f_k,n}(u)=0, D_{f_l,n}(u) = 0, \quad\Rightarrow\quad a_n(u)=0.
    \]
    Thus, Lemma \ref{lemma_lindeberg}(ii) is applicable.
    
    Case $\IK = 1$: If $u \mapsto \IE[P_{0}\bar f_k(\tilde Z_{j_1}(u),u)\cdot P_0 \bar f_l(\tilde Z_{j_2}(u),u)]$ has bounded variation, we have
    \begin{eqnarray*}
        &&\frac{1}{n}\sum_{i=1}^{n-J+1}D_{f_k,n}(\frac{i}{n})D_{f_l,n}(\frac{i}{n})\cdot \IE\big[P_{i}\bar f_k(\tilde Z_{i+j_1}(\frac{i}{n}),\frac{i}{n})\cdot P_i \bar f_l(\tilde Z_{i+j_2}(\frac{i}{n}),\frac{i}{n})|\sG_{i-1}\big]\\
        &\pto& \lim_{n\to\infty}\int_0^{1}D_{f_k,n}(u)D_{f_l,n}(u)\cdot \IE[P_{0}\bar f_k(\tilde Z_{j_1}(u),u)\cdot P_0 \bar f_l(\tilde Z_{j_2}(u),u)] du.
    \end{eqnarray*}
    and thus
    \begin{eqnarray*}
        &&\sum_{i=1}^{n-J+1}\IE[(\sum_{k=1}^{m}M_{i,k})^2|\sG_{i-1}]\\
        &\pto& \sum_{k,l=1}^{m}a_k a_l \cdot \lim_{n\to\infty}\int_0^{1}D_{f_k,n}(u)D_{f_l,n}(u)\cdot \sum_{j_1,j_2=0}^{J-1}\IE[P_{0}\bar f_k(\tilde Z_{j_1}(u),u)\cdot P_0 \bar f_l(\tilde Z_{j_2}(u),u)] du\\
        &=& a'\Sigma_{kl}^{(1)}(J)a
    \end{eqnarray*}
    Here, for $f,g\in \sF$, we have that $\IE[P_{0}\bar f(\tilde Z_{j_1}(u),u)\cdot P_0 \bar g(\tilde Z_{j_2}(u),u)]$ can be written as
    \begin{eqnarray*}
        && \IE[P_{0}\bar f(\tilde Z_{j_1}(u),u)\cdot P_0 \bar g(\tilde Z_{j_2}(u),u)]\\
        &=& \IE[\IE[\bar f(\tilde Z_{j_1}(u),u)|\sG_0]]\cdot \IE[\bar g(\tilde Z_{j_2}(u),u)|\sG_0]] - \IE[\IE[\bar f(\tilde Z_{j_1}(u),u)|\sG_{-1}]]\cdot \IE[\bar g(\tilde Z_{j_2}(u),u)|\sG_{-1}]]
    \end{eqnarray*}
    which shows that the condition stated in the assumption guarantees the bounded variation of $u \mapsto \IE[P_{0}\bar f(\tilde Z_{j_1}(u),u)\cdot P_0 \bar g(\tilde Z_{j_2}(u),u)]$.
    
    Case $\IK = 2$: If $h_n \to 0$, then we obtain similarly
    \begin{eqnarray*}
        &&\sum_{i=1}^{n-J+1}\IE[(\sum_{k=1}^{m}M_{i,k})^2|\sG_{i-1}]\\
        &\pto& \sum_{k,l=1}^{m}a_k a_l \cdot \lim_{n\to\infty}\int_0^{1}D_{f_k,n}(u)D_{f_l,n}(u) du \cdot \sum_{j_1,j_2=0}^{J-1}\IE[P_{0}\bar f_k(\tilde Z_{j_1}(v),v)\cdot P_0 \bar f_l(\tilde Z_{j_2}(v),v)] du\\
        &=& a'\Sigma_{kl}^{(2)}(J)a.
    \end{eqnarray*}
    By the martingale central limit theorem and \reff{theorem_clt_mult_step2}, \reff{theorem_clt_mult_step3}, we obtain that 
    \begin{equation}
        a'S_n(J) \dto N(0, a'\Sigma_{kl}^{(\IK)}(J)a).\label{theorem_clt_mult_step4}
    \end{equation}
    \emph{Conclusion:} For $\IK \in \{1,2\}$, we have
    \begin{equation}
        a'\Sigma_{kl}^{(\IK)}(J) a \to a'\Sigma_{kl}^{(\IK)}(\infty) a \quad\quad (J\to\infty)\label{theorem_clt_mult_step5}
    \end{equation}
    due to
    \begin{eqnarray*}
        &&\sum_{j_1,j_2: \max\{j_1,j_2\} \ge J}\|P_0 \bar f_k(\tilde Z_{j_1}(u),u)\cdot P_0 \bar f_l(\tilde Z_{j_2}(u),u)\|_1\\
        &\le& \sum_{j_1,j_2: \max\{j_1,j_2\} \ge J}\|P_0 \bar f_k(\tilde Z_{j_1}(u),u)\|_2 \|P_0 \bar f_l(\tilde Z_{j_2}(u),u)\|_2 \to 0 \quad (J\to\infty)
    \end{eqnarray*}
    uniformly in $n$ and
    \[
        \sup_n\int_0^{1}|D_{f_k,n}(u)D_{f_l,n}(u)| du \le \sup_n\big(\int_0^{1}D_{f_k,n}(u)^2 du\big)^{1/2}\big(\int_0^{1}D_{f_l,n}(u)^2 du\big)^{1/2} < \infty.
    \]
    By \reff{theorem_clt_mult_step1}, \reff{theorem_clt_mult_step4}, \reff{theorem_clt_mult_step5},
    \[
        \sum_{j\in\Z}\text{Cov}(\bar f_k(\tilde Z_0(u),u),\bar f_l(\tilde Z_{j}(u),u)) = \sum_{j_1,j_2=0}^{\infty}\IE[P_{0}\bar f_k(\tilde Z_{j_1}(u),u)\cdot P_0 \bar f_l(\tilde Z_{j_2}(u),u)]
    \]
    and the Cramer-Wold device, the assertion of the theorem follows.

\end{proof}

\begin{lem}\label{lemma_lindeberg_analytic}
    Let $c \in \R$, $c > 0$.
    \begin{enumerate}
        \item[(i)] For $x,y\in \R$, it holds that
        \[
            (x+y)^2\Ii_{\{|x+y| > c\}} \le 8x^2 \Ii_{\{|x| > \frac{c}{2}\}} + 8y^2 \Ii_{\{|y| > \frac{c}{2}\}}.
        \]
        \item[(ii)] For random variables $W,\tilde W$, it holds that
        \[
            \IE[W^2 \Ii_{\{|W| > c\}}] \le 4\IE[(W - \tilde W)^2] + 4\IE[\tilde W^2\Ii_{\{|\tilde W| > \frac{c}{2}\}}].
        \]
    \end{enumerate}
\end{lem}
\begin{proof}[Proof of Lemma \ref{lemma_lindeberg_analytic}]
    \begin{enumerate}
        \item[(i)] It holds that
        \begin{eqnarray*}
        (x+y)^2\Ii_{\{|x+y| > c\}} &\le& 2\big[x^2 + y^2\big]\Ii_{\{|x| > \frac{c}{2} \text{ or } |y| > \frac{c}{2}\}}\\
        &\le& 2\big[x^2 + y^2\big]\big\{2\Ii_{\{|x| > \frac{c}{2},|y| > \frac{c}{2}\}} + \Ii_{\{|x| > \frac{c}{2},|y| \le  \frac{c}{2}\}} + \Ii_{\{|x| \le \frac{c}{2},|y| > \frac{c}{2}\}}\big\}\\
        &\le& 4\big[x^2 \Ii_{\{|x| > \frac{c}{2}\}} + y^2 \Ii_{\{|y| > \frac{c}{2}\}}\big] + 4x^2\Ii_{\{|x| > \frac{c}{2}\}} + 4y^2 \Ii_{\{|y| > \frac{c}{2}\}}\\
        &\le& 8x^2 \Ii_{\{|x| > \frac{c}{2}\}} + 8y^2 \Ii_{\{|y| > \frac{c}{2}\}}.
    \end{eqnarray*}
    \item[(ii)] We have
    \begin{eqnarray}
        \IE[W^2 \Ii_{\{|W| > c\}}] &\le&  2\IE[(|W|-\tilde W)^2\Ii_{\{|W| > c\}}] + 2\IE[\tilde W^2\Ii_{\{|W| > c\}}]\nonumber\\
        &\le& 2\IE[(W-\tilde W)^2] + 2\IE[\tilde W^2\Ii_{\{|W-\tilde W| + |\tilde W| > c\}}].\label{lindeberg_replacement}
    \end{eqnarray}
    Furthermore, with Markov's inequality,
    \begin{eqnarray*}
        &&\IE[\tilde W^2\Ii_{\{|W-\tilde W| + |\tilde W| > c\}}]\\
        &\le& \IE[\tilde W^2\Ii_{\{|W-\tilde W| > \frac{c}{2}\}}] + \IE[\tilde W^2\Ii_{\{|\tilde W| > \frac{c}{2}\}}]\\
        &\le& (\frac{c}{2})^2 \IP(|W - \tilde W| > \frac{c}{2}) + \IE[\tilde W^2\Ii_{\{|W-\tilde W| > \frac{c}{2}\}}\Ii_{\{|\tilde W| > \frac{c}{2}\}}] + \IE[\tilde W^2\Ii_{\{|\tilde W| > \frac{c}{2}\}}]\\
        &\le& \IE[(W - \tilde W)^2] + 2\IE[\tilde W^2\Ii_{\{|\tilde W| > \frac{c}{2}\}}].
    \end{eqnarray*}
    Inserting this inequality into \reff{lindeberg_replacement}, we obtain the assertion.
    \end{enumerate}
\end{proof}

The following lemma generalizes some results from \cite{dahlhaus2019} using similar techniques as therein.
\begin{lem}\label{lemma_lindeberg}
    Let $q \in \{1,2\}$. Let $\tilde W_i(u)$ be a stationary sequence with
    \begin{equation}
        \sup_{u\in[0,1]}\|\tilde W_0(u)\|_q < \infty, \quad\quad \|\tilde W_0(u) - \tilde W_0(v)\|_q \le C_W|u-v|^{\varsigma}.\label{lemma_ergodic_ass}
    \end{equation}
    Let $a_n:[0,1] \to \R$ be some sequence of functions with $\limsup_{n\to\infty}\frac{1}{n}\sum_{i=1}^{n}|a_n(\frac{i}{n})| < \infty$.
    \begin{enumerate}
        \item[(i)] Let $q = 2$. Let $c_n$ be some sequence with $c_n \to \infty$. Then
        \[
            \frac{1}{n}\sum_{i=1}^{n}|a_n(\frac{i}{n})|\cdot \IE[\tilde W_i(\frac{i}{n})^2\Ii_{\{|\tilde W_i(\frac{i}{n})| > c_n\}}] \to 0,
        \]
        \item[(ii)] Let $q = 1$. Suppose that there exists $h_n > 0, v \in [0,1]$ such that for all $u\in[0,1]$, $|v-u| > h_n$ implies $a_n(u)=0$. Put $A_n = \sup_{i=1,...,n}|a_n(\frac{i}{n})|$ and suppose that
        \[
            \sup_{n\in\N}(h_n\cdot A_n) < \infty, \quad\quad \frac{A_n}{n} \to 0, \quad\quad \frac{a_n(\cdot)}{A_n} \text{ has bounded variation uniformly in $n$}.
        \]
        
        Suppose that the limits on the following right hand sides exist. If $u\mapsto \IE \tilde W_0(u)$ has bounded variation, then
        \[
            \frac{1}{n}\sum_{i=1}^{n}a_n(\frac{i}{n}) \tilde W_i(\frac{i}{n}) \pto \lim_{n\to\infty} \int_0^{1}a_n(u) \IE \tilde W_0(u) d u.
        \]
        If $h_n \to 0$, then
        \[
            \frac{1}{n}\sum_{i=1}^{n}a_n(\frac{i}{n}) \tilde W_i(\frac{i}{n}) \pto \lim_{n\to\infty}\int_0^{1}a_n(u) du\cdot \IE \tilde W_0(v).
        \]
    \end{enumerate}
\end{lem}
\begin{proof}[Proof of Lemma \ref{lemma_lindeberg}]
    Let $J\in\N$ be fixed and assume that $n \ge 2\cdot 2^J$. For $j \in \{1,...,2^J\}$, Define $I_{j,J,n} := \{i\in \{1,...,n\}: \frac{i}{n}\in (\frac{j-1}{2^J},\frac{j}{2^J}]\}$. Then $(I_{j,J,n})_j$ forms a decomposition of $\{1,...,n\}$ in the sense that  $\sum_{j=1}^{2^J}I_{j,J,n} = \{1,...,n\}$. Since $\frac{i}{n}\in (\frac{j-1}{2^J},\frac{j}{2^J}] \Longleftrightarrow \frac{j-1}{2^J}\cdot n < i \le n\cdot \frac{j-1}{2^J} \le \frac{n}{2^J}$, we conclude that $\frac{n}{2^J}-1 \le |I_{j,J,n}| \le \frac{n}{2^J}$. Thus, since $n \ge 2\cdot 2^J$, 
    \begin{equation}
        \Big|\frac{I_{j,J,n}|}{n} - \frac{1}{2^J}\Big| \le \frac{1}{n}, \quad\quad |I_{j,J,n}| \ge \frac{1}{2}\frac{n}{2^J}.\label{lemma_lindeberg_eq0}
    \end{equation}
    Let $w_i$, $i\in\N$ be an arbitrary sequence. Then it holds that
    \begin{eqnarray}
        \Big|\frac{1}{n}\sum_{i=1}^{n}w_i - \frac{1}{2^J}\sum_{j=1}^{2^J} \frac{1}{|I_{j,J,n}|}\sum_{i \in I_{j,J,n}}w_i\Big| &\le& \sum_{j=1}^{2^J}\Big|\frac{|I_{j,J,n}|}{n} - \frac{1}{2^J}\Big|\cdot \Big|\frac{1}{|I_{j,J,n}|}\sum_{i \in I_{j,J,n}} w_i\Big|\nonumber\\
        &\le& \frac{1}{n}\sum_{j=1}^{2^J}\frac{1}{|I_{j,J,n}|}\sum_{i\in I_{j,J,n}}|w_i|\nonumber\\
        &\le& \frac{2^J}{n^2}\sum_{i=1}^{n}|w_i|\label{lemma_lindeberg_eq00}
    \end{eqnarray}
    \begin{enumerate}
        \item[(i)] Application of \reff{lemma_lindeberg_eq00} with $w_i = a_n(\frac{i}{n})\IE[\tilde W_i(\frac{i}{n})^2\Ii_{\{|\tilde W_i(\frac{i}{n})| > c_n\}}]$ yields
    \begin{eqnarray}
       && \frac{1}{n}\sum_{i=1}^{n}\IE[\tilde W_i(\frac{i}{n})^2\Ii_{\{|\tilde W_i(\frac{i}{n})| > c_n\}}]\nonumber\\
        &\le& \frac{1}{2^J}\sum_{j=1}^{2^J} \frac{1}{|I_{j,J,n}|}\sum_{i \in I_{j,J,n}}\IE[\tilde W_i(\frac{i}{n})^2\Ii_{\{|\tilde W_i(\frac{i}{n})| > c_n\}}] + \frac{2^J}{n}\cdot \frac{1}{n}\sum_{i=1}^{n}a_n(\frac{i}{n})\cdot \sup_u\|\tilde W_0(u)\|_2^2.\label{lemma_lindeberg_eq1}
    \end{eqnarray}
    By Lemma \ref{lemma_lindeberg_analytic}(ii), 
    \begin{eqnarray}
        &&\frac{1}{2^J}\sum_{j=1}^{2^J} \frac{1}{|I_{j,J,n}|}\sum_{i \in I_{j,J,n}}|a_n(\frac{i}{n})|\cdot \IE[\tilde W_i(\frac{i}{n})^2\Ii_{\{|\tilde W_i(\frac{i}{n})| > c_n\}}]\nonumber\\
        &\le& \frac{1}{2^J}\sum_{j=1}^{2^J} \frac{1}{|I_{j,J,n}|}\sum_{i \in I_{j,J,n}}|a_n(\frac{i}{n})|\cdot \IE[\tilde W_0(\frac{j}{2^J})^2\Ii_{\{|\tilde W_0(\frac{j}{2^J})| > c_n\}}]\nonumber\\
        &&\quad\quad\quad + \frac{1}{2^J}\sum_{j=1}^{2^J}\frac{1}{|I_{j,J,n}|}\sum_{i\in I_{j,J,n}}|a_n(\frac{i}{n})|\cdot \big\|\tilde W_0(\frac{i}{n}) - \tilde W_0(\frac{j}{2^J})\big\|_2^2\nonumber\\
        &\le& \Big[\sup_{j=1,...,2^J}\IE[\tilde W_0(\frac{j}{2^J})^2\Ii_{\{|\tilde W_0(\frac{j}{2^J})| > c_n\}}] + C_W(2^{-J})^{\varsigma}\Big]\cdot \frac{1}{2^J}\sum_{j=1}^{2^J}\frac{1}{|I_{j,J,n}|}\sum_{i\in I_{j,J,n}}|a_n(\frac{i}{n})|.\label{lemma_lindeberg_eq2}
    \end{eqnarray}
    By \reff{lemma_lindeberg_eq0},
    \[
        \frac{1}{2^J}\sum_{j=1}^{2^J}\frac{1}{|I_{j,J,n}|}\sum_{i\in I_{j,J,n}}|a_n(\frac{i}{n})| \le \frac{2}{n}\sum_{i=1}^{n}|a_n(\frac{i}{n})|.
    \]
    By the dominated convergence theorem,
    \[
        \limsup_{n\to\infty}\IE[\tilde W_0(\frac{j}{2^J})^2\Ii_{\{|\tilde W_0(\frac{j}{2^J})| > c_n\}}].
    \]
    Furthermore, $\limsup_{n\to\infty}\frac{2^J}{n}\cdot \sup_u\|\tilde W_0(u)\|_2^2 = 0$.
    Inserting \reff{lemma_lindeberg_eq2} into \reff{lemma_lindeberg_eq1} and applying $\limsup_{n\to\infty}$ and afterwards, $\limsup_{J\to\infty}$, yields the assertion.
    \item[(ii)] Since \reff{lemma_ergodic_ass} also holds for $\tilde W_0(u)$ replaced by $\tilde W_0(u) - \IE \tilde W_0(u)$, we may assume in the following that w.l.o.g. that $\IE \tilde W_0(u) = 0$.
    
    By \reff{lemma_lindeberg_eq00} applied to $w_i = a(\frac{i}{n}) W_i(\frac{i}{n})$, we obtain
    \begin{eqnarray}
        && \Big\|\frac{1}{n}\sum_{i=1}^{n}a_n(\frac{i}{n}) \tilde W_i(\frac{i}{n}) - \frac{1}{2^J}\sum_{j=1}^{2^J}\frac{1}{|I_{j,J,n}|}\sum_{i\in I_{j,J,n}}a_n(\frac{i}{n}) \tilde W_i(\frac{i}{n})\Big\|_1\nonumber\\
        &\le& \frac{2^J}{n}\cdot \frac{1}{n}\sum_{i=1}^{n}|a_n(\frac{i}{n})|\cdot \sup_{u}\|W_0(u)\|_1 \to 0 \quad(n\to\infty).\label{lemma_ergodic_eq0}
    \end{eqnarray}
    
    We furthermore have
    \begin{eqnarray}
        && \Big\|\frac{1}{2^J}\sum_{j=1}^{2^J}\frac{1}{|I_{j,J,n}|}\sum_{i\in I_{j,J,n}}a_n(\frac{i}{n}) \tilde W_i(\frac{i}{n})\nonumber\\
        &&\quad\quad\quad\quad\quad\quad\quad\quad\quad\quad - \frac{1}{2^J}\sum_{j=1}^{2^J}\frac{1}{|I_{j,J,n}|}\sum_{i\in I_{j,J,n}}a_n(\frac{i}{n}) \tilde W_i(\frac{j-1}{2^J})\Big\|_1\nonumber\\
        &\le& \frac{1}{2^J}\sum_{j=1}^{2^J}\frac{1}{|I_{j,J,n}|}\sum_{i\in I_{j,J,n}}|a_n(\frac{i}{n})| \cdot \big\|\tilde W_0(\frac{i}{n}) - \tilde W_0(\frac{j-1}{2^J})\big\|_1\nonumber\\
        &\le& \frac{2}{n}\sum_{i=1}^{n}|a_n(\frac{i}{n})|\cdot C_W (2^{-J})^{\varsigma}.\label{lemma_ergodic_eq00}
    \end{eqnarray}
    Fix $j\in \{1,...,2^J\}$. Put $u_j := \frac{j-1}{2^J}$ and, for a real-valued positive $x$, define $[x] := \max\{k\in\N: k > x\}$. By stationarity, the following equality holds in distribution:
    \begin{equation}
        \frac{1}{|I_{j,J,n}|}\sum_{i\in I_{j,J,n}}a_n(\frac{i}{n}) \tilde W_i(u_j) \overset{d}{=} \frac{1}{|I_{j,J,n}|}\sum_{i=1}^{|I_{j,J,n}|}a_n(\frac{i}{n}+\frac{[u_j n]-1}{n})\tilde W_i(u_j).\label{lemma_ergodic_eq000}
    \end{equation}
    Put $\tilde W_i(u)^{\circ} := \tilde W_i(u)\Ii_{\{\frac{i}{n} + \frac{[u_j n]-1}{n}\in [\underline{r}_n,\overline{r}_n]\}}$. By partial summation and since $\frac{a_n(\cdot)}{A_n}$ has bounded variation $B_a$ uniformly in $n$,
    \begin{eqnarray}
        &&\frac{1}{|I_{j,J,n}|}\sum_{i=1}^{|I_{j,J,n}|}a_n(\frac{i}{n}+[u_j n]-1)\tilde W_i(u_j)\nonumber\\
        &=& \frac{1}{|I_{j,J,n}|}\sum_{i=1}^{|I_{j,J,n}|-1}\big\{a_n(\frac{i}{n}+[u_j n]-1) - a_n(\frac{i+1}{n}+[u_j n]-1)\big\}\sum_{l=1}^{i}\tilde W_l(u_j)^{\circ}\nonumber\\
        &&\quad\quad + \frac{1}{|I_{j,J,n}|}A_n\cdot \sum_{l=1}^{|I_{j,J,n}|}\tilde W_l(u_j)^{\circ}\nonumber\\
        &\le& \frac{B_a+1}{|I_{j,J,n}|}A_n\cdot \sup_{i=1,...,|I_{j,J,n}|}\Big|\sum_{l=1}^{i}\tilde W_l(u_j)^{\circ}\Big| \label{lemma_ergodic_eq1}
    \end{eqnarray}
    By stationarity, we have
    \begin{eqnarray*}
        &&\sup_{i=1,...,|I_{j,J,n}|}\Big|\sum_{l=1}^{i}\tilde W_l(u_j)^{\circ}\Big|\\
        &=& \sup_{i=1,...,|I_{j,J,n}|}\Big|\sum_{l=1 \vee (\lceil n(v+h_n)\rceil - [u_j n]+1)}^{i \wedge (\lfloor n(v-h_n)\rfloor - [u_j n]+1)}\tilde W_l(u_j)\Big| \overset{d}{=} \sup_{i=1,...,m_n}\Big|\sum_{l=1}^{i}\tilde W_l(u_j)\Big|,
    \end{eqnarray*}
    since $(|I_{j,J,n}| \wedge (\lfloor n(v+h_n)\rfloor - [u_j n]+1)) - (1 \vee (\lceil n(v-h_n)\rceil - [u_j n]+1)) \le m_n := 2nh_n$. By assumption, $m_n = \frac{2n}{A_n}\cdot A_nh_n \to \infty$.
    
    By the ergodic theorem,
    \[
        \lim_{m\to\infty}\Big|\frac{1}{m}\sum_{l=1}^{m}\tilde W_l(u_j)\Big| = 0 \quad a.s.
    \]
    and especially $(\frac{1}{m}\sum_{l=1}^{m}\tilde W_l(u_j))_m$ is bounded a.s. We conclude that 
    \begin{eqnarray*}
        &&\frac{1}{m_n}\sup_{i=1,...,m_n}\Big|\sum_{l=1}^{i}\tilde W_l(u_j)\Big|\\
        &\le& \frac{1}{\sqrt{m_n}}\sup_{i=1,...,\sqrt{m_n}}\Big|\frac{1}{i}\sum_{l=1}^{i}\tilde W_l(u_j)\Big| + \sup_{i=\sqrt{m_n}+1,...,m_n}\Big|\frac{1}{i}\sum_{l=1}^{i}\tilde W_l(u_j)\Big| \to 0.
    \end{eqnarray*}
    We conclude from \reff{lemma_ergodic_eq1} that
    \begin{eqnarray}
        &&\frac{1}{|I_{j,J,n}|}\sum_{i=1}^{|I_{j,J,n}|}a_n(\frac{i}{n}+[u_j n]-1)\tilde W_i(u_j)\nonumber\\
        &\le& 2\cdot 2^J (B_a+1)\cdot A_n\cdot \frac{m_n}{n} \cdot \frac{1}{m_n}\sup_{i=1,...,|I_{j,J,n}|}\Big|\sum_{l=1}^{i}\tilde W_l(u_j)^{\circ}\Big| \to 0.\label{lemma_ergodic_eq0000}
    \end{eqnarray}
    Combination of \reff{lemma_ergodic_eq0}, \reff{lemma_ergodic_eq00}, \reff{lemma_ergodic_eq000} and \reff{lemma_ergodic_eq0000} and applying $\limsup_{n\to\infty}$ and afterwards $\limsup_{J\to\infty}$, we obtain
    \[
        \frac{1}{n}\sum_{i=1}^{n}a_n(\frac{i}{n})\big\{\tilde W_i(\frac{i}{n}) - \IE \tilde W_0(\frac{i}{n})\big\} \pto 0.
    \]
    \end{enumerate}
    If $u \mapsto \IE \tilde W_0(u)$ has bounded variation, we have with some intermediate value $\xi_{i,n} \in [\frac{i-1}{n},\frac{i}{n}]$, 
    \begin{eqnarray*}
        &&\Big|\frac{1}{n}\sum_{i=1}^{n}a_n(\frac{i}{n})\IE \tilde W_0(\frac{i}{n}) - \int_0^{1}a_n(u) \IE \tilde W_0(u) du\Big|\\
        &\le& \frac{1}{n}\sum_{i=1}^{n}\big|a_n(\frac{i}{n})\IE \tilde W_0(\frac{i}{n}) - a_n(\xi_{i,n}) \IE \tilde W_0(\xi_{i,n})\big|\\
        &\le& \frac{A_n}{n}\cdot \frac{1}{A_n}\sum_{i=1}^{n}|a_n(\frac{i}{n}) - a_n(\xi_{i,n})|\cdot \sup_u\| \tilde W_0(u)\|_1\\
        &&\quad\quad + \frac{A_n}{n}\sum_{i=1}^{n}\big|\IE \tilde W_0(\frac{i}{n}) - \IE \tilde W_0(\xi_{i,n})\big| \to 0.
    \end{eqnarray*}
    If instead $h_n \to 0$, we have with some intermediate value $\xi_{i,n} \in [\frac{i-1}{n},\frac{i}{n}]$,
    \begin{eqnarray*}
        &&\Big|\frac{1}{n}\sum_{i=1}^{n}a_n(\frac{i}{n})\IE \tilde W_0(\frac{i}{n}) - \frac{1}{n}\sum_{i=1}^{n}a_n(\frac{i}{n})\IE \tilde W_0(v)\Big|\\
        &\le& \frac{1}{n}\sum_{i=1}^{n}|a_n(\frac{i}{n})|\cdot \sup_{|u-v| \le h_n}\|\tilde W_0(u) - \tilde W_0(v)\|_1 \to 0.
    \end{eqnarray*}
    Since $\frac{a_n(\cdot)}{A_n}$ has bounded variation uniformly in $n$,
    \[
        \Big|\frac{1}{n}\sum_{i=1}^{n}a_n(\frac{i}{n}) - \int_0^{1}a_n(u) du\Big| \le \frac{A_n}{n}\cdot \frac{1}{A_n}\sum_{i=1}^{n}|a_n(\frac{i}{n}) - a_n(\xi_{i,n})| \to 0.
    \]
\end{proof}

\begin{lem}\label{lemma_theorem_clt_mult}
    Let $\sF$ satisfy Assumptions \ref{ass_clt_process}, \ref{ass_clt_fcont} and \ref{ass1_vorher}. Then there exist constants $C_{cont} > 0, C_{\bar f} > 0$ such that for any $f\in \sF$,
    \begin{enumerate}
        \item[(i)] for any $j\ge 1$,
        \begin{eqnarray*}
            \|P_{i-j}f(Z_i,u)\|_2 &\le& D_{f,n}(u) \Delta(j),\\
            \sup_{i=1,...,n}\|f(Z_i,u)\|_2 &\le& C_{\Delta}\cdot D_{f,n}(u),\\
            \sup_{i,u}\|\bar f(Z_i,u)\|_2 \le C_{\bar f}, && \sup_{v,u}\|\bar f(\tilde Z_0(v),u)\|_2 \le C_{\bar f}.
        \end{eqnarray*}
        \item[(ii)] \begin{eqnarray}
        \|\bar f(Z_i,u) - \bar f(\tilde Z_i(\frac{i}{n}),u)\|_2 &\le& C_{cont}\cdot n^{-\varsigma s},\label{lemma_theorem_clt_mult_1}\\
        \|\bar f(\tilde Z_i(v_1),u_1) - \bar f(\tilde Z_i(v_2),u_2)\|_2 &\le& C_{cont}\cdot \big(|v_1 - v_2|^{\varsigma s} + |u_1 - u_2|^{\varsigma s}\big).\label{lemma_theorem_clt_mult_2}
    \end{eqnarray}
    \end{enumerate}
\end{lem}
\begin{proof}[Proof of Lemma \ref{lemma_theorem_clt_mult}]
    \begin{enumerate}
        \item[(i)] If Assumption \ref{ass1_vorher} is satisfied, we have by Lemma \ref{depend_trans} that
    \[
        \|P_{i-j}f(Z_i,u)\|_2 \le \|f(Z_i,u) - f(Z_i^{*(i-j)},u)\|_2 = \delta_2^{f(Z,u)}(j) \le D_{f,n}(u)\Delta(j).
    \]
     The second assertion follows from Lemma \ref{depend_trans}.
     \item[(ii)] Let $\bar C_R := \sup_{v,u_1,u_2}\|\frac{|\bar f(\tilde Z_0(v),u_1) - f(\tilde Z_0(v),u_2)|)}{|u_1 - u_2|^{\varsigma}}\|_{2} < \infty$ (by Assumption \ref{ass_clt_fcont}) and\\
     $C_R := \max\{\sup_{i,u}\|R(Z_i,u)\|_{2}, \sup_{u,v}\|R(\tilde Z_0(v),u)\|_{2}\}$. Then
     \begin{eqnarray}
         \|\bar f(\tilde Z_i(v),u_1) - \bar f(\tilde Z_i(v),u_2)\|_2 \le \bar C_R |u_1-u_2|^{\varsigma}.\label{lemma_theorem_clt_mult_contarg2}
    \end{eqnarray}
    We then have
    \begin{eqnarray*}
        \|\bar f(Z_i,u) - \bar f(\tilde Z_i(v),u)\|_2 &\le& \| |Z_i - \tilde Z_i(v)|_{L_{\sF,s}}^s(R(Z_i,u) + R(\tilde Z_i(v),u) \|_2\\
        &\le& \| |Z_i - \tilde Z_i(v)|_{L_{\sF,s}}^s \|_{\frac{2p}{p-1}}\big(\|R(Z_i,u)\|_{2p} + \|R(\tilde Z_i(v),u)\|_{2p}\big)\\
        &\le& 2C_R\| |Z_i - \tilde Z_i(v)|_{L_{\sF,s}}^s \|_{\frac{2p}{p-1}}.
    \end{eqnarray*}
    Furthermore,
    \begin{eqnarray*}
        \| |Z_i - \tilde Z_i(v)|_{L_{\sF},s}^s\|_{\frac{2 p}{ \bar p-1}} &\le& \sum_{l=0}^{\infty}L_{\sF,l}\| |X_{i-l} - \tilde X_{i-l}(v)|^s \|_{\frac{2p}{p -1}}\\
        &=& \sum_{l=0}^{i}L_{\sF,l}\|X_{i-l} - \tilde X_{i-l}(v) \|_{\frac{2p s}{p -1}}^s \\
        &\le& \sum_{l=0}^{i}L_{\sF,l}C_X^s\big(|v-\frac{i}{n}|^{\varsigma} + l^{\varsigma}n^{-\varsigma}\big)^s\\
        &\le& |v-\frac{i}{n}|^{\varsigma}\cdot C_X |L_{\sF}|_{1} + n^{-\varsigma}\cdot C_X \sum_{l=0}^{\infty}L_{\sF,l}l^{\varsigma s}\big\}.
    \end{eqnarray*}
    We obtain with $C_{cont} := 2\bar C_R + 2C_RC_X\big\{|L_{\sF}|_{1} + \sum_{j=0}^{\infty}L_{\sF,j}j^{\varsigma s}\big\}$ that 
    \begin{equation}
        \|\bar f(Z_{i},u) - \bar f(\tilde Z_i(v),u)\|_2 \le C_{cont}\cdot \Big[|v-\frac{i}{n}|^{\varsigma s} + n^{-\varsigma s}\Big].\label{lemma_theorem_clt_mult_contarg1_ass1}
    \end{equation}
    Furthermore, as above,
    \begin{eqnarray}
        \|f(\tilde Z_i(v_1),u) - f(\tilde Z_i(v_2),u)\|_2 &\le& 2C_R \| |\tilde Z_0(v_1) - \tilde Z_0(v_2)|_{L_{\sF},s}^s \|_{\frac{2p}{p-1}}\nonumber\\
        &\le& 2C_R\sum_{l=0}^{i}L_{\sF,l}\|\tilde X_{0}(v_1) - \tilde X_{0}(v_2)\|_{\frac{2ps}{p-1}}^s\nonumber\\
        &\le& 2C_R C_X |L_{\sF}|_1\cdot |v_1 - v_2|^{\varsigma s}\label{lemma_theorem_clt_mult_contarg12_ass1}
    \end{eqnarray}
    From \reff{lemma_theorem_clt_mult_contarg1_ass1}, we obtain \reff{lemma_theorem_clt_mult_1} with $v = \frac{i}{n}$. From \reff{lemma_theorem_clt_mult_contarg2} and \reff{lemma_theorem_clt_mult_contarg12_ass1}, we conclude \reff{lemma_theorem_clt_mult_2}. 
    \end{enumerate}
\end{proof}

\subsection{Proofs of Section \ref{sec_examples}}
\label{sec_examples_supp}

\begin{proof}[Proof of Lemma \ref{example_maximumlikelihood}]
    Put $D_{v,n}(u) = \sqrt{h}K_h(u-v)$. From (A1) and Assumption \ref{ass_clt_process} we obtain that  $\Delta(k) = O(\delta_{2M}^{X}(k))$, $C_R = 1+k\max\{C_X,1\}^{2M}$.
    
    Since $K$ is Lipschitz continuous and (A2) holds, we have
    \begin{eqnarray*}
        &&\sup_{|v-v'| \le n^{-3}, |\theta - \theta'|_2 \le n^{-3}}\big|\big(\nabla_{\theta}^j L_{n,h}(v,\theta) - \E \nabla_{\theta}^j L_{n,h}(v,\theta)\big)\\
        &&\quad\quad\quad\quad\quad\quad\quad\quad\quad\quad\quad\quad - \big(\nabla_{\theta}^j L_{n,h}(v',\theta') - \E \nabla_{\theta}^j L_{n,h}(v',\theta')\big)\big|_{\infty}\\
        &\le& \sup_{|v-v'| \le n^{-3}, |\theta - \theta'|_2 \le n^{-3}} \frac{C_R}{h^2}\big[ L_K|v-v'| + C_{\Theta}|\theta - \theta'|_2\big]\\
        &&\quad\quad\quad\quad\quad\quad\quad\quad\quad\quad\quad\quad\times \frac{1}{n}\sum_{i=k}^{n}\big(1 + |Z_i|_1^{M} + \E|Z_i|_1^{M}\big)\\
        &=& O_p(n^{-1}).
    \end{eqnarray*}
    
    Let $\Theta_{n}$ be a grid approximation of $\Theta$ such that for any $\theta \in \Theta$, there exists some $\theta'\in \Theta_n$ such that $|\theta-\theta'|_2 \le n^{-3}$. Since $\Theta \subset \R^{d_{\Theta}}$, it is possible to choose $\Theta_n$ such that $|\Theta_n| = O(n^{-6d_{\Theta}})$. Furthermore, define $V_n := \{in^{-3}:i=1,...,n\}$ as an approximation of $[0,1]$.
    
    As in Example \ref{example_regeression}, Corollary \ref{proposition_rosenthal_bound_rates} applied to
    \[
        \sF_j' = \{f_{v,\theta}: \theta \in \Theta_n, v \in V_n\}
    \]
    yields for $j \in \{0,1,2\}$ that
    \begin{equation}
        \sup_{v\in [\frac{h}{2},1-\frac{h}{2}]}\big|\nabla_{\theta}^j L_{n,h}(v,\theta) - \IE \nabla_{\theta}^j L_{n,h}(v,\theta)\big|_{\infty} = O_p\big(\tau_n\big).\label{example_maximumlikelihood_proof_eq1}
    \end{equation}
    
    Put $\tilde L_{n,h}(v,\theta) = \frac{1}{n}\sum_{i=1}^{n}K_h(i/n-v)\ell_{\theta}(\tilde Z_i(v))$. With (A1) it is easy to see that
    \begin{eqnarray}
        &&\big|\IE \nabla_{\theta}^j L_{n,h}(v,\theta) - \IE \nabla_{\theta}^j\tilde L_{n,h}(v,\theta)\big|_{\infty}\nonumber\\
        &\le& \frac{d_{\Theta}^j C_R}{n}\sum_{i=1}^{n}|K_h(i/n-v)|\cdot \| |Z_i - \tilde Z_i(v)|_1 \|_M\nonumber\\
        &&\quad\quad\quad\quad\quad\quad\quad\quad\quad\quad\quad\quad \times\big(1+\| |Z_i|_1\|_M^{M-1} + \| |\tilde Z_i(v)|_1 \|_M^{M-1}\big)\nonumber\\
        &\le& d_{\Theta}^j C_R|K|_{\infty}C_X(1+2C_X^{M-1})\big(n^{-1} + h\big).\label{example_maximumlikelihood_proof_eq2}
    \end{eqnarray}
    Finally, since $K$ has bounded variation and $\int K(u) du = 1$, uniformly in $v \in [\frac{h}{2},1-\frac{h}{2}]$ it holds that
    \begin{equation}
        \IE \nabla_{\theta}^j\tilde L_{n,h}(v,\theta) = \frac{1}{n}\sum_{i=1}^{n}K_h(i/n-v)\IE \nabla_{\theta}^j\ell_{\theta}(\tilde Z_1(v)) = \IE \nabla_{\theta}^j\ell_{\theta}(\tilde Z_1(v)) + O((nh)^{-1}).\label{example_maximumlikelihood_proof_eq3}
    \end{equation}
    From \reff{example_maximumlikelihood_proof_eq1}, \reff{example_maximumlikelihood_proof_eq2} and \reff{example_maximumlikelihood_proof_eq3} we obtain
    \begin{equation}
        \sup_{v\in [\frac{h}{2},1-\frac{h}{2}]}\sup_{\theta \in \Theta}\big|\nabla_{\theta}^j L_{n,h}(v,\theta) - \IE \nabla_{\theta}^j\ell_{\theta}(\tilde Z_1(v))\big|_{\infty} = O_p(\tau_n^{(j)}),\label{example_maximumlikelihood_proof_eq4}
    \end{equation}
    where
    \[
        \tau_n^{(j)} :=  \tau_n + (nh)^{-1} + h, \quad j\in \{0,2\}, \quad\quad \tau_n^{(1)} :=  \tau_n + (nh)^{-1} + B_h.
    \]
    
    By (A3) and \reff{example_maximumlikelihood_proof_eq4} for $j=0$, we obtain with standard arguments that if $\tau_n^{(0)} = o(1)$,
    \[
        \sup_{v\in[\frac{h}{2},1-\frac{h}{2}]}\big|\hat \theta_{n,h}(v) - \theta_0(v)\big|_{\infty} = o_p(1).
    \]
    
    Since $\hat \theta_{n,h}(v)$ is a minimizer of $\theta \mapsto L_{n,h}(v,\theta)$ and $\ell_{\theta}$ is twice continuously differentiable, we have the representation
    \begin{equation}
        \hat \theta_{n,h}(v) - \theta_0(v) = -\nabla_{\theta}^2 L_{n,h}(v,\bar\theta_v)^{-1}\nabla_{\theta} L_{n,h}(v,\theta_0(v)),\label{example_maximumlikelihood_proof_rep}
    \end{equation}
    where $\bar \theta_v \in \Theta$ fulfills $|\bar \theta_v - \theta_0(v)|_{\infty} \le |\hat \theta_{n,h}(v) - \theta_0(v)|_{\infty} = o_p(1)$.
    
    By (A2), we have
    \[
        \big|\IE \nabla_{\theta}^2 \ell_{\theta}(\tilde Z_0(v))\big|_{\theta = \theta_0(v)} - \IE \nabla_{\theta}^2 \ell_{\theta}(\tilde Z_0(v))\big|_{\theta = \bar \theta_v}\big|_{\infty} = O(|\theta_0(v) - \bar \theta_v|_{2}) = o_p(1).
    \]
    and thus with \reff{example_maximumlikelihood_proof_eq4},
    \begin{equation}
        \sup_{v\in [\frac{h}{2},1-\frac{h}{2}]}\big|\nabla_{\theta}^2 L_{n,h}(v,\bar\theta_v) - \IE \nabla_{\theta}^2\ell_{\theta}(\tilde Z_1(v))\big|_{\theta = \theta_0(v)}\big|_{\infty} = O_p(\tau_n^{(2)}) + o_p(1).\label{example_maximumlikelihood_proof_eq5}
    \end{equation}
    By (A3) and the dominated convergence theorem, $\IE \nabla_{\theta}\ell(\tilde Z_0(v)) = \nabla_{\theta}\IE\ell(\tilde Z_0(v)) = 0$. By \reff{example_maximumlikelihood_proof_eq4},
    \begin{eqnarray}
        \sup_{v\in [\frac{h}{2},1-\frac{h}{2}]}\big|\nabla_{\theta} L_{n,h}(v,\theta_0(v))\big|_{\infty} &=& \sup_{v\in [\frac{h}{2},1-\frac{h}{2}]}\big|\nabla_{\theta} L_{n,h}(v,\theta_0(v)) - \IE \nabla_{\theta}\ell(\tilde Z_0(v))\big|_{\infty}\nonumber\\
        &=& O_p(\tau_n^{(1)}).\label{example_maximumlikelihood_proof_eq6}
    \end{eqnarray}
    Inserting \reff{example_maximumlikelihood_proof_eq5} and \reff{example_maximumlikelihood_proof_eq6} into \reff{example_maximumlikelihood_proof_rep}, we obtain
    \[
        \sup_{v\in[\frac{h}{2},1-\frac{h}{2}]}\big|\hat \theta_{n,h}(v) - \theta_0(u)\big|_{\infty} = O_p(\tau_n^{(1)}).
    \]
    This yields an improved version of \reff{example_maximumlikelihood_proof_eq5}:
    \begin{equation}
        \sup_{v\in [\frac{h}{2},1-\frac{h}{2}]}\big|\nabla_{\theta}^2 L_{n,h}(v,\bar\theta_v) - \IE \nabla_{\theta}^2\ell_{\theta}(\tilde Z_1(v))\big|_{\theta = \theta_0(v)}\big|_{\infty} = O_p(\tau_n^{(2)}).\label{example_maximumlikelihood_proof_eq7}
    \end{equation}
    Inserting \reff{example_maximumlikelihood_proof_eq6} and \reff{example_maximumlikelihood_proof_eq7} into \reff{example_maximumlikelihood_proof_rep}, we obtain the assertion.
\end{proof}

\subsection{Form of the $V_n$-norm and connected quantities}
\label{sec_V_imply_supp}

\begin{lem}[Summation of polynomial and geometric decay]\label{lemma_summation} Let $\alpha > 1$ and $q \in \N$. Then it holds that
\begin{enumerate}
    \item[(i)]
\[
    \frac{1}{\alpha-1}q^{-\alpha+1}\le \sum_{j=q}^{\infty}j^{-\alpha} \le \frac{\max\{\alpha,2^{-\alpha+1}\}}{\alpha-1}q^{-\alpha+1}.
\]
    \item[(ii)] For $\sigma > 0$, $\kappa_2 \ge 1$
    \begin{eqnarray*}
        b_{\rho,\kappa_2,l}\cdot \sigma \cdot \log(\sigma^{-1}) \le \sum_{j=1}^{\infty}\min\{\sigma, \kappa_2 \rho^j\} &\le& b_{\rho,\kappa_2}\cdot \sigma \cdot \log(\sigma^{-1} \vee e),\\
        b_{\alpha,\kappa_2,l}\cdot \sigma \cdot \sigma^{-\frac{1}{\alpha}} \le \sum_{j=1}^{\infty}\min\{\sigma, \kappa_2 j^{-\alpha}\} &\le& b_{\alpha,\kappa_2}\cdot \sigma \cdot \max\{\sigma^{-\frac{1}{\alpha}} ,1\},
    \end{eqnarray*}
    where $b_{\rho,\kappa_2}, b_{\rho,\kappa_2,l}$, $b_{\alpha,\kappa_2}, b_{\alpha,\kappa_2,l}$ are constants only depending on $\rho,\kappa_2,\alpha$.
\end{enumerate}
\end{lem}
\begin{proof}[Proof of Lemma \ref{lemma_summation}] \begin{enumerate}
    \item[(i)] Upper bound: If $q \ge 2$, then
\begin{eqnarray*}
    \sum_{j=q}^{\infty}j^{-\alpha} &=& \sum_{j=q}^{\infty}\int_{j-1}^{j}j^{-\alpha} dx \le  \sum_{j=q}^{\infty}\int_{j-1}^{j}x^{-\alpha} dx = \int_{q-1}^{\infty}x^{-\alpha} dx = \frac{1}{-\alpha+1}x^{-\alpha+1}\Big|_{q-1}^{\infty}\\
    &=& \frac{1}{\alpha-1}(q-1)^{-\alpha+1} = \frac{1}{\alpha-1}q^{-\alpha+1}\cdot (\frac{q-1}{q})^{-\alpha+1} \le \frac{2^{-\alpha+1}}{\alpha-1}q^{-\alpha+1}.
\end{eqnarray*}
If $q = 1$, then $\sum_{j=q}^{\infty}j^{-\alpha} = 1+ \sum_{j=q+1}^{\infty}j^{-\alpha} \le 1 + \frac{1}{\alpha-1}q^{-\alpha+1} = \frac{\alpha}{\alpha-1}$.

Lower bound: Using similar decomposition arguments as above, we have
\begin{eqnarray*}
    \sum_{j=q}^{\infty}j^{-\alpha} &\ge& \sum_{j=q}^{\infty}\int_{j}^{j+1}x^{-\alpha} dx = \int_q^{\infty}x^{-\alpha} dx = \frac{1}{-\alpha+1}x^{-\alpha+1}\Big|_q^{\infty}= \frac{1}{\alpha-1}q^{-\alpha+1}.
\end{eqnarray*}
    \item[(ii)]
    \begin{itemize}
        \item \emph{Exponential decay:} Upper bound: First let $a := \max\{\lfloor \frac{\log(\sigma/\kappa_2)}{\log(\rho)}\rfloor,0\}+1$. Then we have
    \begin{eqnarray*}
        \sum_{j=0}^{\infty}\min\{\sigma, \kappa_2\rho^j\} &\le& \sum_{j=0}^{a-1}\sigma + \kappa_2\sum_{j=a}^{\infty}\rho^j = a \sigma + \kappa_2 \frac{\rho^a}{1-\rho}\\
        &\le& a\sigma + \frac{\kappa_2}{1-\rho}\min\{\frac{\sigma}{\kappa_2},1\} \le a\sigma + \frac{\sigma}{1-\rho}\\
        &\le& \sigma\cdot \Big[\frac{1}{\log(\rho^{-1})}\max\{\log(\kappa_2/\sigma),0\} + \frac{2}{1-\rho}\Big]\\
        &\le& \sigma\cdot \Big[\frac{1}{\log(\rho^{-1})}\max\{\log(\sigma^{-1}),0\} + \frac{\log(\kappa_2)\vee 0}{\log(\rho^{-1})} + \frac{2}{1-\rho}\Big]\\
        &\le& b_{\rho,\kappa_2}\cdot \sigma\cdot \log(\sigma^{-1}\vee e),
    \end{eqnarray*}
    where $b_{\rho,\kappa_2} := 2(\log(\kappa_2) \vee 1)\cdot \frac{1}{\log(\rho^{-1})}\big[1 + \frac{2\log(\rho^{-1})}{1-\rho}\big]$.
    
    Lower Bound: Put $\beta(q) = \kappa_2 \sum_{j=q}^{\infty}\rho^j = \frac{\kappa_2}{1-\rho}\rho^q$. Then
    \begin{eqnarray*}
        \sum_{j=1}^{\infty}\min\{\sigma,\kappa_2 \rho^j\} &\ge& \sigma (\hat q-1) + \beta(\hat q),
    \end{eqnarray*}
    where $\hat q = \min\{q\in\N: \frac{\sigma}{\kappa_2} \ge \rho^q\}$. We have $\hat q \ge  \frac{\log(\sigma/\kappa_2)}{\log(\rho)} =: \underline{q}$ and $\hat q \le \underline{q}+1$. Thus
    \[
        \sum_{j=1}^{\infty}\min\{\sigma,\kappa_2 \rho^j\} \ge \sigma (\underline{q}-1) + \beta(\underline{q}+1).
    \]
    Now consider the case $\frac{\sigma}{\kappa_2}<\rho^2$, that is, $\frac{\log(\sigma/\kappa_2)}{\log(\rho)} \ge 2$. Then, $\underline{q}-1 \ge \frac{1}{2}\underline{q}$, and $\overline{q} \le 2\frac{\log(\sigma/\kappa_2)}{\log(\rho)}$. We obtain
    \begin{eqnarray*}
        \sum_{j=1}^{\infty}\min\{\sigma,\kappa_2 \rho^j\} &\ge& \frac{1}{2}\sigma  \frac{\log(\sigma/\kappa_2)}{\log(\rho)} + \frac{\kappa_2 \rho}{1-\rho}\rho^{\frac{\log(\sigma/\kappa_2)}{\log(\rho)}} = \frac{1}{2}\sigma  \frac{\log(\sigma/\kappa_2)}{\log(\rho)} + \frac{\rho}{1-\rho}\sigma\\
        &\ge& \frac{1}{2}\Big(\frac{\rho}{1-\rho} + \frac{1}{\log(\rho^{-1})}\Big) \sigma \log(\sigma^{-1}\kappa_2),
    \end{eqnarray*}
    that is, the assertion holds with $b_{\rho,\kappa_2,l} := \frac{1}{2}\big(\frac{\rho}{1-\rho} + \frac{1}{\log(\rho^{-1})}\big)$.

    \item \emph{Polynomial decay:} Upper bound: Let $a := \lfloor (\frac{\sigma}{\kappa_2})^{-\frac{1}{\alpha}} \rfloor+1 \ge (\frac{\sigma}{\kappa_2})^{-\frac{1}{\alpha}}$. Then we have by (i):
    \begin{eqnarray*}
        \sum_{j=1}^{\infty}\min\{\sigma, \kappa_2j^{-\alpha}\} &\le& \sum_{j=1}^{a}\sigma + \kappa_2\sum_{j=a+1}^{\infty}j^{-\alpha} = a \sigma + \frac{\kappa_2}{\alpha-1}a^{-\alpha+1}\\
        &\le& a \sigma + \frac{\kappa_2^{\frac{1}{\alpha}}}{\alpha-1}\sigma^{\frac{\alpha-1}{\alpha}}\\
        &\le& \sigma\cdot \Big[\kappa_2^{\frac{1}{\alpha}}\sigma^{-\frac{1}{\alpha}} + 1 + \frac{\kappa_2^{\frac{1}{\alpha}}}{\alpha-1}\sigma^{-\frac{1}{\alpha}}\Big]\\
        &\le& \sigma\cdot \Big[\frac{\alpha}{\alpha-1}\kappa_2^{\frac{1}{\alpha}}\sigma^{-\frac{1}{\alpha}} + 1\Big]\\
        &\le& b_{\alpha,\kappa_2}\cdot \sigma \cdot \max\{\sigma^{-\frac{1}{\alpha}},1\},
    \end{eqnarray*}
    where $b_{\alpha,\kappa_2} := 2\frac{\alpha}{\alpha-1}(\kappa_2 \vee 1)^{\frac{1}{\alpha}}$.
    
    Lower Bound: Put $\beta(q) = \kappa_2 \sum_{j=q}^{\infty}j^{-\alpha}$. By (i), $\beta(q) \ge \frac{\kappa_2}{\alpha-1}q^{-\alpha+1}$. Then
    \begin{eqnarray*}
        \sum_{j=1}^{\infty}\min\{\sigma,\kappa_2 j^{-\alpha}\} &\ge& \min_{q\in\N}\{\sigma q + \beta(q)\}\\
        &\ge& \min_{q\in\N}\{\sigma q + \frac{\kappa_2}{\alpha-1}q^{-\alpha+1}\}.
    \end{eqnarray*}
    Elementary analysis yields that the minimum is achieved for $q = \kappa_2^{\frac{1}{\alpha}}\cdot \sigma^{-\frac{1}{a}} = (\frac{\kappa_2}{\sigma})^{\frac{1}{\alpha}}$, that is,
    \[
        \sum_{j=1}^{\infty}\min\{\sigma,\kappa_2 j^{-\alpha}\} \ge \frac{\alpha}{\alpha-1} \kappa_2^{\frac{1}{\alpha}}\cdot \sigma^{\frac{\alpha-1}{\alpha}},
    \]
    the assertion holds with $b_{\alpha,\kappa_2,l} := \frac{\alpha}{\alpha-1} \kappa_2^{\frac{1}{\alpha}}$.
        \end{itemize}
\end{enumerate}
\end{proof}

\begin{lem}[Values of $q^{*}$, $r(\delta)$] \label{r_d_values}
    \begin{itemize}
        \item Polynomial decay $\Delta(j) = \kappa j^{-\alpha}$ $(\alpha > 1$). Then there exist constants $c_{\alpha,\kappa}^{(i)}, C_{\alpha,\kappa}^{(i)} > 0$, $i = 1,2$ only depending on $\kappa,\alpha$ such that
        \[
            c_{\alpha,\kappa}^{(1)}\max\{ x^{-\frac{1}{\alpha}},1\} \le q^{*}(x) \le C_{\alpha,\kappa}^{(1)}\max\{ x^{-\frac{1}{\alpha}},1\},
        \]
        and
        \[
            c_{\alpha,\kappa}^{(2)} \min\{\delta^{\frac{\alpha}{\alpha-1}},\delta\} \le r(\delta) \le C_{\alpha,\kappa}^{(2)} \min\{\delta^{\frac{\alpha}{\alpha-1}},\delta\}.
        \]
        \item Geometric decay $\Delta(j) = \kappa\rho^j$ ($\rho \in (0,1)$). Then there exist constants $c_{\rho,\kappa}^{(i)},C_{\rho,\kappa}^{(i)} > 0$, $i = 1,2$ only depending on $\kappa,\rho$ such that
        \[
            c_{\rho,\kappa}^{(1)}\max\{ \log(x^{-1}),1\} \le q^{*}(x) \le C_{\rho,\kappa}^{(1)}\max\{ \log(x^{-1}),1\},
        \]
        and
        \[
            c_{\rho,\kappa}^{(2)} \frac{\delta}{\log(\delta^{-1} \vee e)} \le r(\delta) \le C_{\rho,\kappa}^{(2)} \frac{\delta}{\log(\delta^{-1} \vee e)}.
        \]
    \end{itemize}
\end{lem}
\begin{proof}[Proof of Lemma \ref{r_d_values}]
\begin{enumerate}
    \item By Lemma \ref{lemma_summation}(i), $\beta_{norm}(q) = \frac{\beta(q)}{q} \in [c_{\alpha,\kappa}q^{-\alpha},C_{\alpha,\kappa}q^{-\alpha}]$ with $c_{\alpha,\kappa} = \frac{\kappa}{\alpha-1}$, $C_{\alpha,\kappa} = \kappa \frac{\max\{\alpha,2^{-\alpha+1}\}}{\alpha-1}$. In the following we assume w.l.o.g. that $C_{\alpha,\kappa} > 1$ and $c_{\alpha,\kappa} < 1$.
    
    \begin{itemize}
        \item $q^{*}(x)$ Upper bound: For any $x > 0$,
        \[
            q^{*}(x) = \min\{q\in\N: \beta_{norm}(q) \le x\} \le \min\{q \in \N: q \ge (\frac{x}{C_{\alpha,\kappa}})^{-\frac{1}{\alpha}}\}  = \lceil (\frac{x}{C_{\alpha,\kappa}})^{-\frac{1}{\alpha}}\rceil.
        \]
    Especially we obtain $q^{*}(x) \le (\frac{x}{C_{\alpha,\kappa}})^{-\frac{1}{\alpha}}+1 \le 2C_{\alpha,\kappa}^{\frac{1}{\alpha}}\max\{x^{-\frac{1}{\alpha}},1\}$. The assertion holds with $C_{\alpha,\kappa}^{(1)} := 2\max\{C_{\alpha,\kappa},1\}^{\frac{1}{\alpha}}$.
    
    \item $q^{*}(x)$ Lower bound: Similarly to above,
    \[
        q^{*}(x) \ge \lceil (\frac{x}{c_{\alpha,\kappa}})^{-\frac{1}{\alpha}}\rceil \ge \big(\frac{x}{c_{\alpha,\kappa}}\big)^{-\frac{1}{\alpha}} = c_{\alpha,\kappa}^{\frac{1}{\alpha}}x^{-\frac{1}{\alpha}}.
    \]
    On the other hand, $q^{*}(x) \ge 1 \ge c_{\alpha,\kappa}^{\frac{1}{\alpha}}$, which yields the assertion with $c_{\alpha,\kappa}^{(1)} = \min\{c_{\alpha,\kappa},1\}^{\frac{1}{\alpha}}$.
    \item $r(\delta)$ Upper bound: Put $r = 2^{\frac{\alpha}{\alpha-1}}c_{\alpha,\kappa}^{-\frac{1}{\alpha-1}}\delta^{\frac{\alpha}{\alpha-1}}$. Then we have
    \[
        q^{*}(r)r \ge \lceil(\frac{r}{c_{\alpha,\kappa}})^{-\frac{1}{\alpha}}\rceil r = 2^{\frac{\alpha}{\alpha-1}}c_{\alpha,\kappa}^{-\frac{1}{\alpha-1}} \lceil 2^{-\frac{1}{\alpha-1}}c_{\alpha,\kappa}^{\frac{1}{\alpha-1}}\delta^{-\frac{1}{\alpha-1}}\rceil  \delta^{\frac{\alpha}{\alpha-1}} \ge 2\delta > \delta.
    \]
    By definition of $r(\cdot)$, $r(\delta) \le r$. It was already shown in Lemma \ref{lemma_chaining_i} that $r(\delta)\le \delta$ holds for all $\delta > 0$. We obtain the assertion with $C_{\alpha,\kappa}^{(2)} =  2^{\frac{\alpha}{\alpha-1}}c_{\alpha,\kappa}^{-\frac{1}{\alpha-1}}$.
    
    \item $r(\delta)$ Lower bound: First consider the case $\delta < C_{\alpha,\kappa}$.
    
    Put $r = 2^{-\frac{\alpha}{\alpha-1}}C_{\alpha,\kappa}^{-\frac{1}{\alpha-1}}\delta^{\frac{\alpha}{\alpha-1}}$. Since $x:= 2^{\frac{1}{\alpha-1}} C_{\alpha,\kappa}^{\frac{1}{\alpha-1}}\delta^{-\frac{1}{\alpha-1}} > 1$, $\lceil x \rceil \le 2x$ and thus
    \[
        q^{*}(r)r \le \lceil(\frac{r}{C_{\alpha,\kappa}})^{-\frac{1}{\alpha}}\rceil r = 2^{-\frac{\alpha}{\alpha-1}}C_{\alpha,\kappa}^{-\frac{1}{\alpha-1}}\lceil 2^{\frac{1}{\alpha-1}} C_{\alpha,\kappa}^{\frac{1}{\alpha-1}}\delta^{-\frac{1}{\alpha-1}}\rceil\delta^{\frac{\alpha}{\alpha-1}} \le 2\cdot 2^{-1}\delta \le \delta.
    \]
    By definition of $r(\cdot)$, $r(\delta) \ge r = 2^{-\frac{\alpha}{\alpha-1}}\min\{(\frac{\delta}{C_{\alpha,\kappa}})^{\frac{1}{\alpha-1}},1\} \delta$.
    
    In the case $\delta > C_{\alpha,\kappa}$, we have
    \[
        q^{*}(\delta)\delta = \lceil(\frac{\delta}{C_{\alpha,\kappa}})^{-\frac{1}{\alpha}}\rceil \delta \le 1\cdot \delta \le \delta,
    \]
    thus $r(\delta)\ge \delta = \min\{(\frac{\delta}{C_{\alpha,\kappa}})^{\frac{1}{\alpha-1}},1\} \delta \ge 2^{-\frac{\alpha}{\alpha-1}}\min\{(\frac{\delta}{C_{\alpha,\kappa}})^{\frac{1}{\alpha-1}},1\} \delta$. We conclude that the assertion holds with $c_{\alpha,\kappa}^{(2)} = 2^{-\frac{\alpha}{\alpha-1}}C_{\alpha,\kappa}^{-\frac{1}{\alpha-1}}$.
    \end{itemize}
     \item We have $\beta_{norm}(q) = \frac{\beta(q)}{q} =  C_{\rho,\kappa}\frac{\rho^{q}}{q}$, where $C_{\rho,\kappa} = \frac{\kappa \rho}{1-\rho}$. In the following we assume w.l.o.g. that $C_{\rho,\kappa} > 8$.
        
    \begin{itemize}
        \item $q^{*}(x)$ Upper bound: Put $\psi(x) = \max\{\log(x^{-1}),1\}$. Define $\tilde q = \lceil\frac{\psi(\frac{x}{C_{\rho,\kappa}\log(\rho^{-1})})}{\log(\rho^{-1})}\rceil$. Then we have
        \[
            \beta_{norm}(\tilde q) \le C_{\rho,\kappa} \frac{\rho^{\log(\big(\frac{x}{C_{\rho,\kappa}\log(\rho^{-1})}\big)^{-1})/\log(\rho^{-1})}}{\tilde q} \le \frac{\frac{x}{\log(\rho^{-1})}}{\tilde q} \le \frac{x}{\psi(\frac{x}{C_{\rho,\kappa}\log(\rho^{-1})})} \le x,
        \]
        thus
        \[
            q^{*}(x) = \min\{q\in\N: \beta_{norm}(q) \le x\} \le \tilde q = \Big\lceil\frac{\psi(\frac{x}{C_{\rho,\kappa}\log(\rho^{-1})})}{\log(\rho^{-1})}\Big\rceil.
        \]
        Especially we obtain
        \[
            q^{*}(x) \le \frac{1}{\log(\rho^{-1})}\big(\psi(x) + \log(C_{\rho,\kappa}\log(\rho{-1}))\big) + 1 \le \frac{2(1 + \log(C_{\rho,\kappa}\log(\rho^{-1})))}{\log(\rho^{-1})} \psi(x),
        \]
        that is, the assertion holds with $C_{\rho,\kappa}^{(1)} = \frac{2(1 + \log(C_{\rho,\kappa}\log(\rho{-1})))}{\log(\rho^{-1})}$.
        
        \item $q^{*}(x)$ Lower Bound: Case 1: Assume that $x < C_{\rho,\kappa}\log(\rho^{-1})\rho^4$. Define $\tilde q = \lceil\frac{1}{4}\frac{\log( (\frac{x}{C_{\rho,\kappa}\log(\rho^{-1})})^{-1})}{\log(\rho^{-1})}\rceil \ge 1$. Then $\tilde q \le \frac{1}{2}\frac{\log( (\frac{x}{C_{\rho,\kappa}\log(\rho^{-1})})^{-1})}{\log(\rho^{-1})}$, and thus
        \[
            \beta_{norm}(\tilde q) \ge C_{\rho,k}\frac{\Big(\frac{x}{C_{\rho,\kappa}\log(\rho^{-1})}\Big)^{1/2}}{\tilde q} \ge (C_{\rho,\kappa}\log(\rho^{-1}))^{1/2}\frac{x^{1/2}}{\log( (\frac{x}{C_{\rho,\kappa}\log(\rho^{-1})})^{-1/2})} > x
        \]
        since
        \[
            (\frac{x}{C_{\rho,\kappa}\log(\rho^{-1})})^{-1/2} > \log( (\frac{x}{C_{\rho,\kappa}\log(\rho^{-1})})^{-1/2}).
        \]
        We have therefore shown that for $x < C_{\rho,\kappa}\log(\rho^{-1})\rho^{4}$,
        \begin{equation}
            q^{*}(x) \ge \tilde q = \max\{1,\tilde q\}.\label{qstar_upperbound_geometric}
        \end{equation}
        Case 2: If $x \ge C_{\rho,\kappa}\log(\rho^{-1})\rho^{4}$, then $\tilde q \le 1$, that is,
        \[
            q^{*}(x) \ge 1 = \max\{1, \tilde q\}.
        \]
        
        We have shown that for all $x > 0$,
        \[
            q^{*}(x) \ge \max\{1,\tilde q\}.
        \]
        Since
        \begin{eqnarray*}
            \tilde q &\ge& \frac{1}{4}\frac{\log( (\frac{x}{C_{\rho,\kappa}\log(\rho^{-1})})^{-1})}{\log(\rho^{-1})} \ge \frac{1}{4\log(\rho^{-1})}\big[\log(x^{-1}) + \log(C_{\rho,\kappa}\log(\rho^{-1}))\big]\\
            &\ge& \frac{1}{4\log(\rho^{-1})}\log(x^{-1}),
        \end{eqnarray*}
        the assertion follows with $c_{\rho,\kappa}^{(1)} = \frac{1}{4\log(\rho^{-1})}$.
        \item $r(\delta)$ Upper bound: Put $\tilde r = \frac{2(c_{\rho,\kappa}^{(1)})^{-1}\delta}{\log((2^{-1}c_{\rho,\kappa}^{(1)}\delta^{-1}) \vee e)}$. Then we have
        \begin{eqnarray*}
            q^{*}(\tilde r) \tilde r &\ge& c_{\rho,\kappa}^{(1)} \log(\tilde r^{-1} \vee e) \cdot \tilde r\\
            &=& \frac{2\delta}{\log((2^{-1}c_{\rho,\kappa}^{(1)}\delta^{-1}) \vee e)}\cdot \log( [2^{-1}c_{\rho,\kappa}^{(1)} \delta^{-1} \log((2^{-1}c_{\rho,\kappa}^{(1)}\delta^{-1})\vee e)] \vee e)\\
            &\ge& \frac{2\delta}{\log((2^{-1}c_{\rho,\kappa}^{(1)}\delta^{-1}) \vee e)}\cdot \log( [2^{-1}c_{\rho,\kappa}^{(1)} \delta^{-1}] \vee e) = 2\delta > \delta.
        \end{eqnarray*}
        By definition of $r(\cdot)$, we obtain
        \[
            r(\delta) \le \tilde r.
        \]
        For $a \in (0,1)$, the function $(0,\infty) \to (0,\infty), x \mapsto \frac{\log(x^{-1} \vee e)}{\log((a x^{-1}) \vee e)}$ attains its maximum at $x = ae^{-1}$ with maximum value $1+\log(a^{-1})$. Thus
        \[
            \tilde r\le 2(c_{\rho,\kappa}^{(1)})^{-1}(1+\log(2^{-1}(c_{\rho,\kappa}^{(1)})^{-1}))\cdot \frac{\delta}{\log(\delta^{-1}\vee e)},
        \]
        that is, the assertion holds with $C_{\rho,\kappa}^{(2)} = 2(c_{\rho,\kappa}^{(1)})^{-1}(1+\log(2^{-1}(c_{\rho,\kappa}^{(1)})^{-1}))$.
        \item $r(\delta)$ Lower Bound: Put $\tilde r = \frac{2^{-1}(C_{\rho,\kappa}^{(1)})^{-1}\delta}{\log((2C_{\rho,\kappa}^{(1)}\delta^{-1}) \vee e)}$. Then
        \begin{eqnarray*}
            q^{*}(\tilde r) \tilde r &\le& C_{\rho,\kappa}^{(1)} \log(\tilde r^{-1} \vee e) \cdot \tilde r\\
            &=& \frac{2^{-1}\delta}{\log((2C_{\rho,\kappa}^{(1)}\delta^{-1}) \vee e)}\cdot \log( [2C_{\rho,\kappa}^{(1)} \delta^{-1} \log((2C_{\rho,\kappa}^{(1)}\delta^{-1})\vee e)] \vee e)\\
            &\le& \frac{2^{-1}\delta}{\log((C_{\rho,\kappa}^{(1)}\delta^{-1}) \vee e)}\cdot \big[\log( (2C_{\rho,\kappa}^{(1)} \delta^{-1}) \vee e) + \log \log( (2C_{\rho,\kappa}^{(1)} \delta^{-1}) \vee e)\big]\\
            &\le& \delta,
        \end{eqnarray*}
        where the last step is due to $\log(x) + \log\log(x) \le 2\log(x)$ for $x \ge e$. By definition of $r(\cdot)$, we obtain
        \[
            r(\delta) \ge \tilde r.
        \]
        For $a > 1$, the function $(0,\infty) \to (0,\infty), x \mapsto \frac{\log(x^{-1} \vee e)}{\log((a x^{-1}) \vee e)}$ attains its minimum at $x = e^{-1}$ with minimum value $\frac{1}{1+\log(a)}$. We therefore obtain
        \[
            \tilde r\ge \frac{(C_{\rho,\kappa}^{(1)})^{-1}}{2(1+\log(2 C_{\rho,\kappa}^{(1)}))} \frac{\delta}{\log(\delta^{-1}\vee e)},
        \]
        that is, the assertion holds with $c_{\rho,\kappa}^{(2)} = \frac{(C_{\rho,\kappa}^{(1)})^{-1}}{2(1+\log(2 C_{\rho,\kappa}^{(1)}))}$.
    \end{itemize}
\end{enumerate}
\end{proof}

\begin{lem}[Form of $V_n$]\label{lemma_form_v2}
\begin{enumerate}
    \item Polynomial decay $\Delta(j) = \kappa j^{-\alpha}$ (where $\alpha > 1$): Then there exist some constants $C_{\alpha,\kappa}^{(3)}, c_{\alpha,\kappa}^{(3)}$ only depending on $\kappa,\alpha,\ID_n$ such that
    \[
        c_{\alpha,\kappa}^{(3)}\|f\|_{2,n}\max\{\|f\|_{2,n}^{-\frac{1}{\alpha}},1\} \le V_n(f) \le C_{\alpha,\kappa}^{(3)}\|f\|_{2,n}\max\{\|f\|_{2,n}^{-\frac{1}{\alpha}},1\}.
    \]
    \item Geometric decay $\Delta(j) = \kappa\rho^j$ (where $\rho\in (0,1)$): Then there exist some constants $c_{\rho,\kappa}^{(3)}, C_{\rho,\kappa}^{(3)}$ only depending on $\kappa,\rho,\ID_n$ such that
    \[
        c_{\rho,\kappa}^{(3)}\|f\|_{2,n}\max\{\log(\norm{f}_{2,n}^{-1}),1\} \le V_n(f) \le C_{\rho,\kappa}^{(3)}\|f\|_{2,n}\max\{\log(\norm{f}_{2,n}^{-1}),1\}.
    \]
\end{enumerate}
\end{lem}
\begin{proof}[Proof of Lemma \ref{lemma_form_v2}] The assertions follow from Lemma \ref{lemma_summation}(ii) by taking $\kappa_2 = \kappa \ID_n$. The maximum in the lower bounds is obtained due to the additional summand $\|f\|_{2,n}$ in $V_n(f)$.

\end{proof}

The following lemma formulates the entropy integral in terms of the well-known bracketing numbers in terms of the $\|\cdot\|_{2,n}$-norm in the case that $\sup_{n\in\N}\ID_n < \infty$. For this, we use the upper bounds of $V_n$ given in Lemma \ref{lemma_form_v2}.

\begin{lem}\label{cor_integralcondition_explicit}
    \begin{enumerate}
        \item Polynomial decay $\Delta(j) = \kappa j^{-\alpha}$ (where $\alpha > 1$). Then for any $\sigma \in (0,C_{\alpha,\kappa}^{(3)})$,
        \[
            \int_0^{\sigma}\sqrt{\IH(\varepsilon,\sF,V_n)} d\varepsilon \le C_{\alpha,\kappa}^{(3)}\frac{\alpha-1}{\alpha}\int_0^{(\frac{\sigma}{C_{\alpha,\kappa}^{(3)}})^{\frac{\alpha}{\alpha-1}}}u^{-\frac{1}{\alpha}}\sqrt{\IH(u,\sF,\|\cdot\|_{2,n})} du,
        \]
        where $C_{\alpha,\kappa}^{(3)}$ is from lemma \ref{lemma_form_v2}.
        \item Exponential decay $\Delta(j) = \kappa \rho^j$ (where $\rho \in (0,1)$). Then for any $\sigma \in (0,e^{-1}C_{\rho,\kappa}^{(3)})$,
        \[
            \int_0^{\sigma}\sqrt{\IH(\varepsilon,\sF,V_n)} d\varepsilon \le C_{\rho,\kappa}^{(3)}\int_0^{E^{-}(\frac{\sigma}{C_{\rho,\kappa}^{(3)}})} \log(u^{-1})\sqrt{\IH(u,\sF,\|\cdot\|_{2,n})} du,
        \]
        where $E^{-}(x) = \frac{x}{\log(x^{-1})}$ and $C_{\rho,\kappa}^{(3)}$ is from lemma \ref{lemma_form_v2}.
    \end{enumerate}
\end{lem}
\begin{proof}[Proof of Lemma \ref{cor_integralcondition_explicit}]
\begin{enumerate}
    \item By Lemma \ref{lemma_form_v2}, $V_n(f) \le C_{\alpha,\kappa}^{(3)}\|f\|_{2,n}\max\{\|f\|_{2,n}^{-\frac{1}{\alpha}},1\}$. We abbreviate $c = C_{\alpha,\kappa}^{(3)}$ in the following.
        
    Let $\varepsilon \in (0,c)$ and $(l_j,u_j)$, $j = 1,...,N$  brackets such that $\|u_j - l_j\|_{2,n} \le (\frac{\varepsilon}{c})^{\frac{\alpha}{\alpha-1}}$. Then
    \[
        V_n(u_j-l_j) \le c\max\{\|u_j - l_j\|_{2,n},\|u_j - l_j\|_{2,n}^{\frac{\alpha-1}{\alpha}}\} \le c\max\Big\{(\frac{\varepsilon}{c})^{\frac{\alpha}{\alpha-1}},\frac{\varepsilon}{c}\Big\} \le c \cdot \frac{\varepsilon}{c} = \varepsilon.
    \]
    Therefore, the bracketing number fulfill the relation
    \[
        \N(\varepsilon,\sF, V_n) \le \N\Big((\frac{\varepsilon}{c})^{\frac{\alpha}{\alpha-1}},\sF, \|\cdot\|_{2,n}\Big).
    \]
    We conclude that for $\sigma \in (0,c)$, 
    \begin{eqnarray*}
        \int_0^{\sigma}\sqrt{\IH(\varepsilon,\sF, V_n)} d \varepsilon &\le& \int_0^{\sigma} \sqrt{\IH\Big((\frac{\varepsilon}{c})^{\frac{\alpha}{\alpha-1}},\sF, \|\cdot\|_{2,n}\Big)} d\varepsilon\\
        &=& c\frac{\alpha-1}{\alpha}\int_0^{(\frac{\sigma}{c})^{\frac{\alpha}{\alpha-1}}} u^{-\frac{1}{\alpha}}\sqrt{\IH(u,\sF, \|\cdot\|_{2,n})} du.
    \end{eqnarray*}
    In the last step, we used the substitution $u = (\frac{\varepsilon}{c})^{\frac{\alpha}{\alpha-1}}$ which leads to $\frac{du}{d\varepsilon}=\frac{\alpha}{\alpha-1}\cdot \frac{1}{c}\cdot (\frac{\varepsilon}{c})^{\frac{1}{\alpha-1}} = \frac{\alpha}{\alpha-1}\cdot \frac{1}{c}\cdot u^{\frac{1}{\alpha}}$.
    \item By Lemma \ref{lemma_form_v2}, $V_n(f) \le C_{\rho,\kappa}^{(3)} E(\|f\|_{2,n})$ with $E(x) = x \max\{\log(x^{-1}),1\}$. We abbreviate $c = C_{\rho,\kappa}^{(3)}$ in the following.
    
    We first collect some properties of $E$. Put $E^{-}(x) = \frac{x}{\log(x^{-1}\vee e)}$. In the case $x > e^{-1}$, we have $E(E^{-}(x)) = x$. In the case $x \le e^{-1}$, we have
        \[
            E(E^{-}(x)) = \frac{x}{\log(x^{-1})}\cdot \log\Big(\frac{x^{-1}}{\log(x^{-1})^{-1}}\Big) \le \frac{x}{\log(x^{-1})}\log(x^{-1}) = x.
        \]
        This shows that for all $x > 0$,
        \begin{equation}
            E(E^{-}(x)) \le x.\label{cor_integralcondition_explicit_eq1}
        \end{equation}
        Furthermore, for $x < e^{-1}$,
        \begin{equation}
            \log(E^{-}(x)^{-1}) = \log(x^{-1} \log(x^{-1})) \ge \log(x^{-1}).\label{cor_integralcondition_explicit_eq2}
        \end{equation}

    Now let $\varepsilon \in (0,1)$ and $(l_j,u_j)$, $j = 1,...,N$  brackets such that $\|u_j - l_j\|_{2,n} \le E^{-}(\frac{\varepsilon}{c})$. Then by \reff{cor_integralcondition_explicit_eq1},
    \[
        V_n(u_j-l_j) \le c E(E^{-}(\frac{\varepsilon}{c})) \le c\cdot \frac{\varepsilon}{c} = \varepsilon.
    \]
    Therefore, we have the following relation between the bracketing numbers
    \[
        \N(\varepsilon,\sF, V_n) \le \N\Big(E^{-}(\frac{\varepsilon}{c}),\sF, \|\cdot\|_{2,n}\Big).
    \]
    We conclude that for $\sigma \in (0,c e^{-1})$, 
    \[
        \int_0^{\sigma}\sqrt{\IH(\varepsilon,\sF, V_n)} d \varepsilon \le \int_0^{\sigma} \sqrt{\IH\Big(E^{-}(\frac{\varepsilon}{c}),\sF, \|\cdot\|_2\Big)} d\varepsilon \le c\int_0^{E^{-}(\frac{\sigma}{c})} \log(u^{-1})\sqrt{\IH(u,\sF, \|\cdot\|_2)} du.
    \]
    In the last step, we used the substitution $u = E^{-}(\frac{\varepsilon}{c})$ which leads to $\frac{du}{d\varepsilon}=\frac{1}{c}\cdot \frac{1+\log((\varepsilon/c)^{-1})}{\log((\varepsilon/c)^{-1})^2}$, and with \reff{cor_integralcondition_explicit_eq2} we obtain
    \[
        d\varepsilon = c\frac{\log((\varepsilon/c)^{-1})^2}{1+\log((\varepsilon/c)^{-1})} du \le c\log((\varepsilon/c)^{-1}) du \le c \log(E^{-}(\frac{\varepsilon}{c})^{-1}) du = c \log(u^{-1}) du.
    \]
\end{enumerate}
\end{proof}


\end{document}